\documentclass[11pt, twoside]{article}
\usepackage{amsfonts,amssymb,amsmath,amsthm}
\usepackage{graphicx}
\usepackage[all]{xy}

\usepackage{psfrag,xmpmulti,amscd,color}

\newtheorem{thm}{Theorem}[section]
\newtheorem{cor}[thm]{Corollary}
\newtheorem{lem}[thm]{Lemma}
\newtheorem{prop}[thm]{Proposition}
\newtheorem*{propA}{\bf Proposition A (Finiteness of fixed ray pairs)}
\newtheorem*{mainthm}{\bf Main Theorem (Separation Theorem)}
\newtheorem*{thmB}{\bf Theorem B (Global counting)}
\newtheorem*{thmC}{\bf Theorem C (Separation Theorem for periodic rays)}
\newtheorem*{corD}{\bf Corollary D}
\newtheorem*{corE}{\bf Corollary E (Hidden components of Siegel disks)}
\newtheorem{claim}[thm]{Claim}

\theoremstyle{definition}
\newtheorem{defn}[thm]{Definition}
\newtheorem{rem}[thm]{Remark}
\newtheorem*{rem*}{Remark}

\numberwithin{equation}{section}

 \divide\oddsidemargin by 2
\renewcommand{\Im}{\operatorname{Im}}
\renewcommand{\Re}{\operatorname{Re}}

\newcommand{\diam}{\operatorname{diam}}
\newcommand{\eucl}{\operatorname{eucl}}
\newcommand{\dist}{\operatorname{dist}}

\newcommand{\ind}{\operatorname{Ind}}

\renewcommand{\AA}{{\cal A}}

\newcommand{\BB}{{\cal B}}

\newcommand{\FF}{{\cal F}}
\newcommand{\GG}{{\cal G}}

\newcommand{\PP}{{\cal P}}

\renewcommand{\SS}{{\cal S}}
\newcommand{\TT}{{\cal T}}

\newcommand{\C}{{\mathbb C}}

\newcommand{\D}{{\mathbb D}}
\newcommand{\Hyp}{{\mathbb H}}

\newcommand{\N}{{\mathbb N}}

\newcommand{\R}{{\mathbb R}}

\newcommand{\Z}{{\mathbb Z}}

\newcommand{\ra}{\rightarrow}
\newcommand{\ov}{\overline}

\renewcommand{\epsilon}{\varepsilon}
\renewcommand{\phi}{\varphi}

\newcommand{\BBt}{\widehat{\BB}}
\newcommand{\Bt}{\hat{\BB}}
\newcommand{\dt}{{\widetilde{\delta}}}

\newcommand{\dtj}{{\widetilde{\delta_j}}}

\newcommand{\Sti}{{\widetilde{S_i}}}
\newcommand{\St}{{\widetilde{S}}}
\newcommand{\TTt}{{\widetilde{\TT}}}
\newcommand{\ft}{\widetilde{f}}

\newcommand{\Ft}{{\widetilde{F}}}
\newcommand{\Tt}{{\widetilde{T}}}
\newcommand{\Tta}{{\widetilde{T}_{\alpha,i}}}
\newcommand{\Ttb}{{\widetilde{T}_{\beta,i}}}
\newcommand{\Ht}{{\widetilde{H}}}
\newcommand{\st}{{\widetilde{s}}}
\newcommand{\zt}{{\widetilde{z}}}
\newcommand{\Ta}{{T_{\alpha}}}
\newcommand{\Dbar}{{\overline{D}}}

\newcommand{\deltat}{{\widetilde{\delta}}}
\newcommand{\s}{{s}}

\newcommand{\FFa}{{\FF_\alpha}}
\newcommand{\gammah}{\widehat{\gamma}}
\newcommand{\sigmah}{\widehat{\sigma}}
\newcommand{\Gammah}{\widehat{\Gamma}}
\newcommand{\chat}{\widehat{\C}}
\newcommand{\Vhat}{{\widehat{V}}}

\newcommand{\GV}{{\Gamma_\Vhat}}

\newcommand{\dap}{{\delta^{+}_{\alpha}}}
\newcommand{\dam}{{\delta^{-}_{\alpha}}}
\newcommand{\dapm}{{\delta^{\pm}_{\alpha}}}

\newcommand{\damp}{{\delta^{+}_{\alpha-1}}}

\newcommand{\Pap}{{P^{+}_{\alpha}}}
\newcommand{\Pam}{{P^{-}_{\alpha}}}

\newcommand{\Pamp}{{P^{+}_{\alpha-1}}}

\newcommand{\ftmF}{{\ft}^{-1}_{\Ft}}
\newcommand{\ftmT}{{{\ft}^{-1}_{\Tt}}}

\renewcommand{\hat}{\widehat}
\renewcommand{\ss}{\scriptsize}

\renewcommand{\b}{\color{blue}}

\newcommand{\Gt}{\tilde{G}}
 \title{A separation theorem for entire transcendental maps}
\author{\small Anna Miriam Benini \thanks{Partially supported by the EU network MRTN-CT-2006-035651 CODY}\\  
\small CRM Ennio de Giorgi\\
\small Piazza dei Cavalieri 3 \\   
\small 56100 Pisa, Italy\\ 
\small {\tt ambenini$@$gmail.com} 
\and 
\small N\'uria Fagella\thanks{Partially supported by the catalan grant 2009SGR-792, by the spanish grants MTM-2008-01486, MTM2006- 05849 Consolider (including a FEDER contribution) and MTM2011-26995-C02-02, and by the european network MRTN-CT-2006-035651-2-CODY. }\\   
\small Dept.~de Mat.~Aplicada i An\`alisi\\ 
\small Univ.~de Barcelona, Gran Via 585 \\ 
\small 08007 Barcelona, Spain\\  
\small {\tt fagella$@$maia.ub.es} 
}  

\begin{document} 

\maketitle  
\begin{abstract}
 We study the distribution of periodic points for a wide class of
maps, namely entire transcendental functions of finite order and  with
bounded set of singular values,  or compositions thereof.  Fix $p\in\N$ and assume that 
all dynamic rays which are invariant under $f^p$  land. An interior $p$-periodic point  is a  fixed point of $f^p$  which is not the landing point of any periodic ray invariant under $f^p$.  
Points belonging to  attracting, Siegel  or  Cremer cycles
are  examples of interior periodic points. 
For functions as above, we show that rays which are invariant under $f^p$, together with their landing points, separate the plane into finitely many regions, each containing {\em exactly} one interior $p-$periodic point or one parabolic immediate basin invariant under $f^p$.
This result generalizes the Goldberg-Milnor Separation Theorem 
 for polynomials \cite{GM}, and  has several corollaries.  
 It follows, for example,  that two periodic Fatou components can always be separated by a pair of periodic rays landing together; that there cannot be Cremer
points on the boundary of Siegel disks; that  
``hidden components'' of a bounded Siegel disk have to be either wandering domains or preperiodic to the Siegel disk itself; or that there are only finitely many non-repelling cycles of any given period, regardless of the number of singular values. 
\end{abstract}

\section{Introduction}
 
 Given a holomorphic map $f:\C \ra \C$, we are interested in the dynamical system generated by the iterates of $f$. In this setup, there is a dynamically meaningful partition of the phase space into two completely invariant subsets: the {\em Fatou set}, where the dynamics are  in some sense stable, and the {\em Julia set}, where they are chaotic. More precisely, the Fatou set is defined as the open set 
\[
\FF(f):=\left\{ z\in\C; \{f^n\}\text{ is normal in a neighborhood of $z$ }\right\},
\] 
and the {\em Julia set} $J(f)$ as its complement. Another dynamically interesting set is the set of escaping points or  {\em escaping set} 
\[
I(f):=\{z\in\C; \ f^n(z)\ra \infty, \text{\ as $n\to \infty$}\}.
\]
The relation between them is that, in general, $J(f)=\partial I(f)$, although for some classes of functions $I(f)\subset J(f)$ (\cite{EL}) and hence  $J(f)=\overline{I(f)}$.

In this paper we are mainly concerned with {\em entire transcendental} maps (abbreviated transcendental maps), that is those entire maps for which infinity is an essential singularity.  There are several important differences between the dynamics of  transcendental maps and that of polynomials, coming from the  very different behavior of iterates in a neighborhood of infinity.  For example, while  the Julia set of polynomials is always a compact set disjoint from $I(f)$,  its analogue for transcendental maps is always unbounded and may contain points of $I(f)$.

Another relevant difference concerns the singularities of the inverse function, which always play a crucial role in holomorphic dynamics. For polynomials (or rational maps) all branches of $f^{-1}$ are well defined in a neighborhood of any point, with the exception of {\em critical values}, that is, $v=f(c)$ where $c$ is a  zero of $f'$ (i.e. a {\em  critical point}). If $f$ is entire, one more type of isolated singularity of $f^{-1}$ is allowed, namely the {\em asymptotic values}, or points $a\in\C$ for which there is a curve $\gamma(t)$ such that $|\gamma(t)|\ra\infty$ as $t\ra\infty$, and $f(\gamma(t))\ra a$ as  $t\ra\infty$. Informally, asymptotic values have some of their preimages at infinity.

A holomorphic function is a covering outside the set $S(f)$ of \emph{singular values} of $f$, where
 \[
 S(f):=\overline{\{z\in\C;  \text{\ $z$ is an asymptotic or critical value for $f$} \}}.
 \]
Observe that rational maps, and in particular polynomials,  have a finite number of singular values while transcendental maps may have infinitely many such singularities. This is in fact part of the reason why  
for polynomials, every connected component of the Fatou set (also called {\em Fatou component})  is preperiodic to the basin of an   attracting or  parabolic  periodic point, or to a periodic rotation domain (Siegel disk). Instead, transcendental maps allow for additional types of Fatou components like {\em wandering components} (that is non-preperiodic)  or {\em Baker domains} (i.e. periodic  components for which all their points tend to infinity under iteration).  There exist certain natural classes of transcendental maps for which the dynamics is better understood. One such class is the Eremenko-Lyubich class
\[
\BB=\{f:\C \to \C \text{\ entire transcendental; $S(f)$ is bounded}\}.
\] 
For functions in class $\BB$, the escaping set is a subset of the Julia set (see \cite{EL}, Theorem 1) and therefore there are neither  Baker domains nor wandering domains in which the iterates tend to infinity. Recently (\cite{Bi}) has constructed an example of function in class $\BB$ with wandering domains whose orbits are unbounded but do not tend to infinity; it is still unknown in general whether there can be functions with wandering domains which have bounded orbits. 
If $S(f)$ is finite, then $f$ is said to be of {\em finite type} and has no wandering domains of any kind (\cite{EL}, \cite{GK}). 

Recall that a function is of \emph{finite order} if $\log\log|f(z)|$ is of the order of $\log|z|$ as $|z|\ra\infty$. For example, $f(z)=e^{z^p}$ has order $p$.   In this paper we restrict   to the class 
\[
\BBt=\{f=g_1\circ{\dots} \circ g_k; \text{\ $k\geq1$, $g_i\in\BB$ and $g_i$ has finite order}\}.
\]
Observe that class $\BBt$ is contained in $\BB$ and  is closed under composition, as is $\BB$, while the class of functions of finite order is not. 

  The reason why we restrict ourselves to class $\BBt$, is because it is the most general (natural) class for which \emph{(dynamic) rays} or {\em hairs} have been proven to exist. Informally, a ray is a maximal unbounded injective curve  $g:(0,\infty)\to I(f) \setminus C(f)$, where $C(f)$ is the set of critical points of $f$,  such that all its iterates are also unbounded curves with the same properties (see Section~\ref{Tools} for a precise definition).  With this definition, rays are always pairwise disjoint. Rays which end at a critical point (or an iterated preimage thereof) but which are not maximal in $I(f)$ are called {\em broken}. We say that a ray {\em lands} at $z_0\in\C$ if  it is not broken and $g(t)\to z_0$ as $t\to 0$.  Broken rays do not land but a ray can land at a critical point. Observe that the image of a landing ray is a landing ray. 
  
 For polynomials {the assumption that there are no broken rays is equivalent to the assumption that the Julia set is connected; in this case} rays foliate the attracting basin of infinity and belong to the Fatou set (see for example \cite{DH}, \cite{Mi}). They are in one to one correspondence with straight rays in the complement of $\ov{\D}$ via B\"ottcher coordinates,  and therefore 
their dynamics is governed by $\theta \mapsto d \theta$, where $\theta$ is the angle that identifies the ray and $d$ is the degree of the polynomial (recall that a polynomial behaves like $z\mapsto z^d$ in a neighborhood of infinity). 
  
 For transcendental maps  in class $\BB$, the situation is quite different since rays belong to the Julia set.  The existence of  rays, and their organization with respect to some symbolic dynamics  was settled initially in \cite{DT} for functions of finite type (with some extra technical condition). Recently in \cite{R3S}, these results have been extended to the class $\BBt$ (in fact, to a more general but less natural class than $\BBt$). We refer to Section~\ref{Tools} for background and precise statements. Of special importance for us are the {\em $p-$periodic} rays, i.e., rays that are invariant under $f^p$.  As it is the case  for periodic points, $p$ is called the {\em period} of the ray. Unless otherwise specified periods are never assumed to be exact. When the period is 1, the ray is called {\em invariant} or {\em fixed}.   If two $p-$periodic rays land together at a common point,  the curve which consists of the two rays and their common landing point is called a \emph{ray pair} (of period $p$). In contrast, we say that a $p-$periodic ray {\em lands alone} if its landing point is not the landing point of any other ray  of the same period. Observe that, although it is expected to be impossible,  it is not yet known in general  whether rays with different (exact) periods can land together at the same point (see e.g. \cite{BL} for a proof that this cannot happen in the exponential case). Therefore  a priori the concept of landing alone depends on the period $p$.

Our goal in the present work is to prove some results concerning the relation between periodic rays, their landing points and the distribution of some special periodic points in the complex plane. More precisely, a $p-$periodic point is called an \emph{interior $p-$periodic point} if there are no $p-$periodic rays landing at it. 

Periodic points which are neither repelling nor parabolic (hence (super)attracting, Siegel or Cremer) are necessarily interior periodic points. However, repelling or parabolic $p-$periodic points  can be interior $p-$periodic points as well, for example if they are landing points only of rays which are not fixed by $f^p$, or if they are not landing points of any ray. 

We first restrict the attention to fixed points and fixed rays.  As opposed to polynomials, functions in class $\BB$ in general have infinitely many of each. We assume that all fixed rays land in $\C$, and in fact this will be a standing assumption for most of the  paper. 
In the polynomial case, it is enough to assume that the Julia set is connected  to ensure that all fixed rays land (in fact that all periodic rays land).   For transcendental functions, there is no known necessary and sufficient condition to show that fixed rays  land. There are however many cases in which this can be shown, for example if the postsingular set 
\[
\PP(f):=\ov{\bigcup_{s\in\SS(f),n\in\N}\{f^n(s)\}}
\] 
does not intersect the set of fixed rays (see {\cite{den,Fa,Rel}}).

Let $\Gamma$ be the graph formed by the fixed rays of $f$ together with their endpoints. A connected  component of $\C\setminus \Gamma$ is called a \emph{basic region} for $f$. 

Let us observe some facts which are obvious for polynomials but {not even true} for transcendental maps. 
A polynomial of degree $d$ always has $d-1$ fixed rays (fixed points of $\theta \mapsto d \theta$) and $d$ fixed points, counted with multiplicity. This means that  there is always at least one interior fixed point, or one multiple fixed point, even if all fixed rays land alone (this would be the case with only one basic region). If some fixed rays land together, then more fixed points become multiple or otherwise they become interior as they cannot be landing points of any fixed ray. This elementary counting, and the fact that there are only finitely many basic regions,  is possible  because  of the finiteness of the degree. In the transcendental case there are infinitely many fixed points and infinitely  many fixed rays, so a priori one might have infinitely many fixed ray pairs, basic regions and interior fixed points. Our first simple, but rather surprising result is the following proposition.

\begin{propA}
Let $f\in \Bt$. Then all but finitely many fixed rays land {alone} at repelling fixed points. Consequently, {if all fixed rays land}, there are only finitely many fixed ray pairs and finitely many basic regions.
\end{propA}

Additionally, a global counting of interior fixed points similar to the polynomial case can be done. 

\begin{thmB}\label{thmB}
Let $f\in\Bt$ such that all fixed rays land.  Let  $G$ be  a collection of $L \geq 0$  fixed rays such that it contains the set of rays which do {\em not} land alone at repelling fixed points. Then there are exactly $L+1$ fixed points, counted with multiplicity, which are either landing points of some $g\in G$ or interior fixed points. 
\end{thmB}

An easy but surprising consequence of this theorem (see Corollary \ref{onlyoneregion}) is that a function in class $\Bt$ without fixed ray pairs (but which may have infinitely many singular values) can have at most one interior fixed point, in particular at most one non-repelling fixed point. 

To state our main theorem we need to recall the local dynamics around a parabolic fixed point (see \cite{Mi}).  In a neighborhood of a  fixed point $z_0$ with multiplicity $m+1\geq 2$, a holomorphic map $f$ can be written as $f(z)=z + a (z-z_0)^{m+1}+ \mathcal{O}(z-z_0)^{m+2}$. In this case, $z_0$ is in the boundary of $m$ distinct immediate basins of attraction which are fixed by $f$, formed by points whose orbits converge to $z_0$ in a given asymptotic direction. Following Goldberg and Milnor \cite{GM}, we call each of these basins a {\em virtual fixed point} of $f$.  More generally, any immediate basin of attraction for a  parabolic $p-$periodic point (i.e. a point with multiplier one under $f^p$) is called  a  {\em virtual p-periodic point} whenever it is fixed by $f^p$.

 The following is the main theorem in the paper. For completeness, it also contains 
the statement of Proposition A. 

 \begin{mainthm}\label{thmA}
Let  $f\in\Bt$ and assume  that all fixed rays land. Then there are finitely many basic regions, and each basic region  contains exactly one interior fixed point or virtual fixed point.
\end{mainthm}

{It follows that interior fixed points can not be multiple. They can still be parabolic but only if the multiplier is different from 1. Otherwise they would have a virtual fixed point attached contradicting the Separation Theorem.} 

 The same Separation Theorem concerning polynomials with a connected Julia set was proven by Goldberg and Milnor in \cite[Theorem~3.3]{GM} and we shall refer to it  as the Goldberg-Milnor Theorem.  In the polynomial case the finiteness of the set of basic regions is immediate, since there are only finitely many fixed rays and fixed points. 
 In their proof, the finiteness of the degree of $f$ is used in an essential way. In a suitable restriction of each basic region, they construct a weakly polynomial-like map and use the  Lefschetz fixed point theory to count the number of critical points and fixed points.
 
The strategy as it is  cannot be adapted to the transcendental setting, where the degree is hardly well defined on any restriction, and where the role of critical values can be carried out also by asymptotic values. However,  to prove the Main Theorem we do use Goldberg and Milnor's idea of doing a counting  in the basic regions, and we manage to do it successfully by counting directly  the fixed points instead of going through critical points. 

As in the polynomial case, the statement about periodic points follows from  the fixed case and from the fact that the class $\Bt$ is invariant under composition. Since this is the most general result and the one that is used in the applications, we state it as  a theorem although it follows immediately from the Main Theorem.

\begin{thmC}\label{thmC}
Let $f\in\Bt$ such that $p$-periodic rays land. Then there are finitely many basic regions for $f^p$, and each basic region  contains exactly one interior p-periodic point or virtual p-periodic point.
\end{thmC}

The true power of the Goldberg-Milnor Theorem, and therefore of our Separation Theorem, resides in their remarkable number of corollaries. Some of them were not yet known for transcendental functions of any class. We summarize them in the following statement (see Section~\ref{Fixed points on the boundary} for the definition of petals). 

\begin{corD}\label{corD}
If $f$ is a function in $\Bt$ whose periodic rays land,
\begin{enumerate}
\item There are only finitely many interior periodic  points of any given period.  In particular, there are only finitely many non-repelling cycles of any given period.
\item Any two periodic Fatou components can be separated by a periodic ray pair. 
\item There are no Cremer   periodic points on the boundary of periodic Fatou components.
\item If $z_0$ is a  parabolic point, for each repelling petal there is at least  one  ray  which lands at $z_0$ through that repelling petal. In particular, any two disjoint attracting petals invariant under $f^p$ are separated by a ray pair of period $p$.
\item For any given period $p$, there are only finitely many (possibly none) {$p-$periodic} points which are landing points of more than one periodic ray. None of them is the landing point of infinitely many rays of the same period.
\end{enumerate}
\end{corD}

Note that statement 1 holds regardless of the presence of infinitely many singular values.

From the fact that Fatou components can be separated by periodic ray pairs we obtain the following additional corollary. Given an invariant Siegel disk $\Delta$, we say that $U$ is a \emph{hidden component} of $\Delta$ if $U$ is a bounded connected component of $\C\setminus\overline{\Delta}$.  Siegel disks with hidden components have never been found but their existence has never been ruled out either, not even for polynomials.

\begin{corE}
If $f\in\BBt$ and all periodic rays land, then any hidden component of a bounded invariant Siegel disk is either a wandering domain or  preperiodic to the Siegel disk itself.
\end{corE}

This article is structured as follows. In Section \ref{Tools} we state and prove a result concerning the existence of fixed rays for functions in $\BBt$ and some expansion properties (Lemma~\ref{Technical}). We call it the Structural Lemma and it will be deduced from results in \cite{R3S}. Hence, we must introduce logarithmic coordinates and the setup of \cite{R3S} to prove it. We also  introduce some additional symbolic dynamics to be able to transfer their statements to the original dynamical plane. This setup has interest on its own and we hope that it can be used for other purposes as well.  With the Structural Lemma in hand, in Section \ref{Where are interior periodic points?}  we present  some results about the topology and general distribution of fixed points in the plane, proving Proposition A. In Section \ref{Index} we   prove some  lemmas about homotopies used in the two following sections, which are dedicated to the proof of Theorem B and to  the proof of the Main Theorem respectively. The latter uses the setup and notation of the first so it cannot be read independently. Finally, in Section \ref{Corollaries} we prove corollaries  D  and E.

\subsection*{Acknowledgments}
We are thankful to Arnaud Ch\'eritat and to Pascale Roesch for sharing with us the preprint of their recent work on Siegel disks for polynomials with two critical values \cite{CR}, from which comes the original idea that ended up generating this present paper. We thank Xavier Buff for helpful discussions and 
we are also thankful to IMUB,  to the Complex Dynamics group at Universitat de Barcelona, to the CRM Ennio de Giorgi, and to the CODY  Marie Curie Network thanks to which this project could be carried out.  We  give very special thanks to the referee for the very thoughtful comments which brought to the clarification of many details, as well as suggesting several simplifications for our proofs. Thanks to them the exposition was largely improved.

\section{Tools and preliminary results}\label{Tools}

In this section, we present the main features of the dynamical plane for $f\in\BBt$ and state a Structural Lemma (Lemma~\ref{Technical}) which is needed for the rest of the paper.
We prove it using results of \cite{R3S} together with some study of the combinatorics of $f$. In further sections we will refer to the statement of the Structural Lemma but not to its proof; the faithful reader can skip Sections~\ref{Logarithmic coordinates}, \ref{Expansion near infinity} and \ref{raysandprop}.

\subsection{Structure of the dynamical plane and main technical lemma} \label{structure}

{The following construction is from \cite{EL} (see also \cite{R3S}).} Let $f\in\BB$, and $D$ be a Jordan domain containing $S(f)$, $0$ and  $f(0)$.  For simplicity we assume that the boundary of $D$ is real analytic {and for most of the applications, one can think of $D$ as a round disk}.  Then the preimage $\TT$ of $\C\setminus \Dbar$ under $f$ consists of countably many unbounded connected components $\{\Ta, \alpha \in \AA\}$ called \emph{tracts} of $f$, and $f:{\Ta}\ra \C\setminus \Dbar$ is a universal covering of infinite degree. If f has finite order then the number of tracts is finite.

 Observe that since fibers are discrete, tracts accumulate only at infinity. In particular, only finitely many tracts intersect any compact set.
 
 \begin{lem}\label{Existence of delta} There is a simple curve $\delta\subset\C\setminus(\Dbar\cup\ov{\TT})$ connecting $\Dbar$ to infinity.
 \end{lem}
 
  Existence of $\delta$ seems to be a rather standard fact for functions with finitely many tracts. However we are not aware of any reference covering the case with infinitely many tracts, so   we include a proof here.
 \begin{proof}Consider the closure $\hat{\TT}$ of the union of all tracts in the Riemann sphere $\hat{\C}$. This set is connected and full so its complement in $\hat{\C}$  is simply connected. By Caratheodory's Theorem, to show that any point in $\partial \hat{\TT}$, including infinity, is accessible by curves in $\hat{\C}\setminus\hat{\TT}$, it is enough to show that ${\partial}\hat{\TT}$ is locally connected, or equivalently that $\hat{\TT}$ is locally connected.  Local connectivity at interior points is immediate. At any point $z\in \partial\hat{\TT}, z\neq \infty$, it follows from the fact that $z$ is on the boundary of some tract and cannot be accumulated by any other point of $\hat{\TT}$, neither from the same tract (the boundary is a  smooth curve), nor from other tracts. To show local connectivity at infinity let $r>0$ and consider the ball $B_r$ of spherical  radius $r$ centered at infinity. Since tracts do not accumulate in any compact set, $\partial B_r$ intersects only finitely many tracts, and $B_r\cap \hat{\TT}$ has only finitely many connected components, exactly  one of which contains infinity. It follows that the finitely many components of $B_r\cap \hat{\TT}$ not containing infinity can be removed from $B_r$ giving a simply connected open neighborhood of infinity whose intersection with $\hat{\TT}$ is connected.
 \end{proof}
 
 Observe that, in fact, there are many accesses to infinity,  providing plenty of freedom in the choice of $\delta$.  Let $\delta$ be any curve as in Lemma~\ref{Existence of delta}.   The preimages $\{f^{-1}\delta\}$ subdivide each tract $T_\alpha$ into \emph{fundamental domains} $F_{\alpha,i}$ with $ i\in \Z$ on which $f: F_{\alpha,i}\ra\C\setminus(\delta\cup \ov{D})$ is univalent (see Figure~\ref{tractsandfd}).

 \begin{figure}[hbt!]
\begin{center}
\def\svgwidth{0.5\textwidth}
\begingroup%
  \makeatletter%
  \providecommand\color[2][]{%
    \errmessage{(Inkscape) Color is used for the text in Inkscape, but the package 'color.sty' is not loaded}%
    \renewcommand\color[2][]{}%
  }%
  \providecommand\transparent[1]{%
    \errmessage{(Inkscape) Transparency is used (non-zero) for the text in Inkscape, but the package 'transparent.sty' is not loaded}%
    \renewcommand\transparent[1]{}%
  }%
  \providecommand\rotatebox[2]{#2}%
  \ifx\svgwidth\undefined%
    \setlength{\unitlength}{517.48867188bp}%
    \ifx\svgscale\undefined%
      \relax%
    \else%
      \setlength{\unitlength}{\unitlength * \real{\svgscale}}%
    \fi%
  \else%
    \setlength{\unitlength}{\svgwidth}%
  \fi%
  \global\let\svgwidth\undefined%
  \global\let\svgscale\undefined%
  \makeatother%
  \begin{picture}(1,0.68251966)%
    \put(0,0){\includegraphics[width=\unitlength]{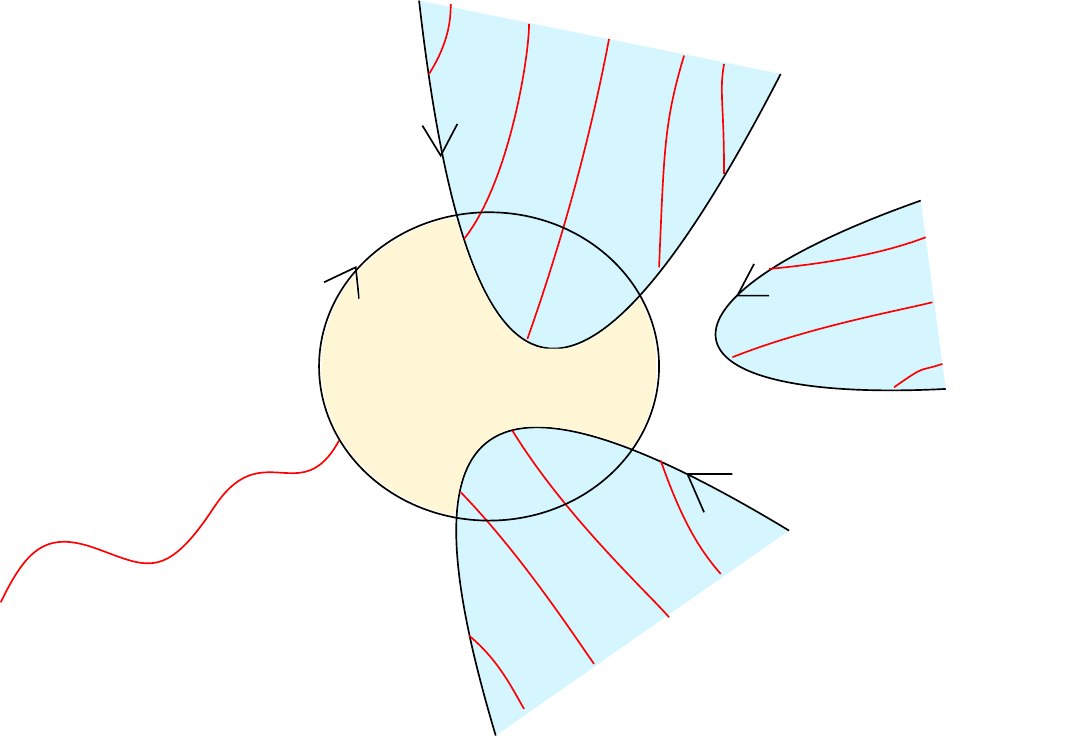}}%
    \put(0.08106458,0.19190807){\color[rgb]{0,0,0}\makebox(0,0)[lb]{\smash{$\delta$}}}%
    \put(0.61904737,0.05741238){\color[rgb]{0,0,0}\makebox(0,0)[lb]{\smash{\large $T_1$}}}%
    \put(0.8957684,0.4176135){\color[rgb]{0,0,0}\makebox(0,0)[lb]{\smash{\large $T_2$}}}%
    \put(0.54948063,0.66959969){\color[rgb]{0,0,0}\makebox(0,0)[lb]{\smash{\large $T_3$}}}%
    \put(0.50928651,0.13934654){\color[rgb]{0,0,0}\makebox(0,0)[lb]{\smash{\ss $F_{1,0}$}}}%
    \put(0.62368515,0.50109359){\color[rgb]{0,0,0}\makebox(0,0)[lb]{\smash{\ss $F_{3,-1}$}}}%
    \put(0.41962271,0.60621666){\color[rgb]{0,0,0}\makebox(0,0)[lb]{\smash{\ss $F_{3,2}$}}}%
    \put(0.4860976,0.56756847){\color[rgb]{0,0,0}\makebox(0,0)[lb]{\smash{\ss $F_{3,1}$}}}%
    \put(0.54948063,0.53201214){\color[rgb]{0,0,0}\makebox(0,0)[lb]{\smash{\ss $F_{3,0}$}}}%
    \put(0.73962972,0.39906237){\color[rgb]{0,0,0}\makebox(0,0)[lb]{\smash{\ss $F_{2,0}$}}}%
    \put(0.76591049,0.34186305){\color[rgb]{0,0,0}\makebox(0,0)[lb]{\smash{\ss $F_{2,-1}$}}}%
    \put(0.44744941,0.09760649){\color[rgb]{0,0,0}\makebox(0,0)[lb]{\smash{\ss $F_{1,-1}$}}}%
    \put(0.6546037,0.18417844){\color[rgb]{0,0,0}\makebox(0,0)[lb]{\smash{\ss $F_{1,2}$}}}%
    \put(0.58194511,0.16098952){\color[rgb]{0,0,0}\makebox(0,0)[lb]{\smash{\ss $F_{1,1}$}}}%
    \put(0.35469375,0.39751644){\color[rgb]{0,0,0}\makebox(0,0)[lb]{\smash{$D$}}}%
  \end{picture}%
\endgroup%
\end{center}
\caption{\small The topological  disk $D$, the tracts of $f$ and the fundamental domains.}
\label{tractsandfd}
\end{figure}
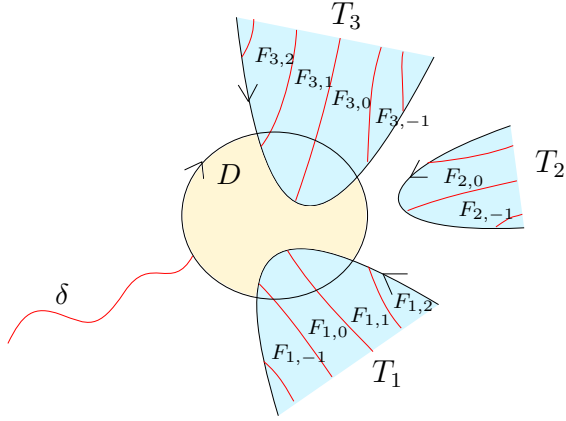
  
 When it is not relevant  which tract the fundamental domains are contained in, or their position inside the tract,   we  call a fundamental domain simply $F$. Observe that for each $F$, because  $f|_F$ is univalent, there is a well defined bijective inverse branch  $f^{-1}_F: \C\setminus(\delta\cup \ov{D})\ra F$.

\begin{defn}[\bf Dynamic ray]
A \emph{(dynamic) ray} for $f$ is an injective curve  $g:(0,\infty)\ra I(f)$ such that:
\begin{itemize}
\item[(a)] $\underset{t\ra\infty}\lim |f^n(g(t))|=\infty\,\  \forall n\geq 0$;
\item[(b)] $\underset{n\ra\infty}\lim |f^n(g(t))|=\infty$  uniformly in $[t_0,\infty)$ for all $t_0>0$;
\item[(c)] $f^n(g(t))$ is not a critical point for any $t>0$ and $n>0$;
\end{itemize} 
and such that $g(0,\infty)$ is maximal with respect to these properties. If  $g(0,\infty)$ is not maximal  when excluding (c),  then we call the ray {\em broken}.
\end{defn} 

Broken rays could therefore be continued if we allowed critical points and their iterated preimages to be part of the ray, as it is the case in the definition in \cite{R3S}, where branching might occur and several rays might share one same arc. This  situation cannot happen in our setting, i.e.  rays are pairwise disjoint. 
 
Recall that a  dynamic ray $g$ is  \emph{periodic} (or $p-$periodic) if $f^p(g)\subset g$ for some $p\geq 1$, and \emph{fixed} if $p=1$. We say that a dynamic ray \emph{lands} at a point $z_0\in\C$ if {it is not broken and} $\lim g(t)= z_0$ as $t\ra 0 $. Observe that dynamic rays are allowed to land at critical points, but that broken rays do not land.

We say that a dynamic  ray $g$ is \emph{asymptotically contained} in a fundamental domain $F$ if $g(t)\in F$ for all $t$ sufficiently large. It is easy to see that this is always the case.

\begin{lem}\label{asymptotically contained} Let $f\in \BB$. Then any dynamic ray is asymptotically contained in some fundamental domain.
\end{lem}
\begin{proof}
A dynamic ray must intersect  $\TT$ because all its images must be unbounded. Now suppose it is not asymptotically contained in some fundamental domain.  Then there exists some sequence $t_n\ra\infty$ such that $g(t_n)$ belongs to preimages of $\delta$. Hence $f(g)$ intersects $\delta$ infinitely many times and $f^2(g)$ intersects $D$ infinitely many times as $t\ra\infty,$ contradicting again property (a) in the definition.   \end{proof}

The following lemma summarizes all of the results from this section that we shall use in the rest of the paper. It relies on the results and constructions from \cite{R3S}, however it requires independent proofs that we will give in the following subsections.  Recall that maps in class $\BBt$, as defined in the introduction, are also in class $\BB$. For $R>0$, we denote by  $D_R$ the open disk centered at 0 and  of radius $R$,  and by $C_R$ the boundary of $D_R$. 

\begin{lem}[\bf Structural Lemma]\label{Technical}
Let $f\in\BBt$. Then
\begin{enumerate}
\item[(a)] Tracts and fundamental domains have a cyclic order at infinity and can be labeled according to that order.
\item[(b)]
Let $\FF$ be the union of finitely many  fundamental domains. Then for any $R$ sufficiently large, the set $f^{-1}(C_R) \cap \FF$  is contained in $D_R$. 
\item[(c)]  For each fundamental domain $F$, there is a unique fixed  dynamic  ray $g_{F}$ which is asymptotically contained in $F$. Conversely, any ray is asymptotically contained in some fundamental domain. 
\end{enumerate}
\end{lem}

The three following subsections are dedicated to the proof of the three statements of Lemma~\ref{Technical}.

\subsection{Logarithmic coordinates and lift of $f$}\label{Logarithmic coordinates} 

For an entire transcendental function  in class $\BB$ we consider logarithmic coordinates following \cite{EL} (see also \cite{R3S}).
The logarithm is a well defined multivalued function on $\C\setminus\Dbar$. Define the set  $\Ht:=\exp^{-1}(\C\setminus \Dbar )$ and observe that it  contains  a right half plane. 
Also $\exp^{-1}(\delta)$ is a countable family of curves $\deltat_i$ dividing $\Ht$ into semi-strips $\Sti$; without loss of generality we  label one of them $\St_0$ and find a branch $L_0$ of the logarithm mapping $\C\setminus(\Dbar\cup\delta)$ into $\St_0$. If we define branches $L_n$ of the logarithm as $L_n(w)=L_0(w)+2\pi i n$, we obtain a consecutive labeling of the strips $\St_n$ as images of $L_n$. Call $\deltat_n, \deltat_{n+1}$ the preimages of $\delta$ bounding the semi-strip $\St_n$.
Observe that the tracts do not contain the origin, so the branches of the logarithm are defined on the entire tracts
(see Figure~\ref{logcoordinatesnew2}).
The set  $\TTt:=\exp^{-1}(\TT)$ then consists of infinitely many connected components $\Tta$ which do not accumulate in any compact set; the set $\Tta$ is defined as $L_i(T_\alpha)$ so an unbounded part of $\Tta$  is always contained in $\St_i$. 

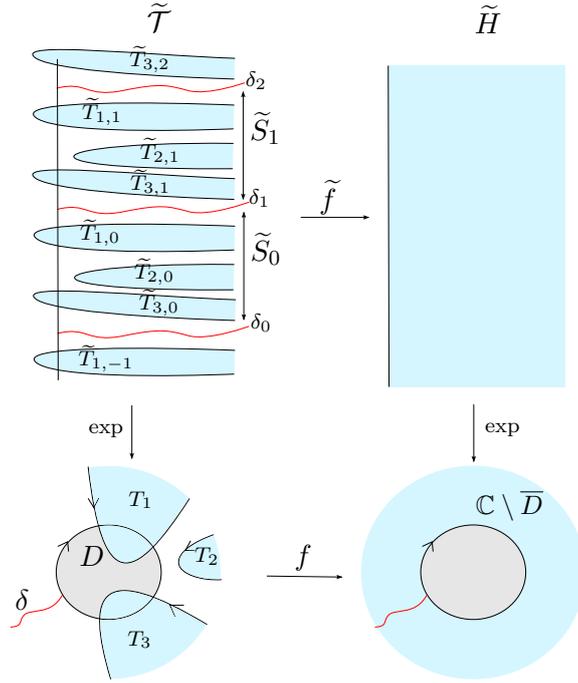
\begin{figure}[hbt!]
\begin{center}
\def\svgwidth{0.5\textwidth}
\begingroup%
  \makeatletter%
  \providecommand\color[2][]{%
    \errmessage{(Inkscape) Color is used for the text in Inkscape, but the package 'color.sty' is not loaded}%
    \renewcommand\color[2][]{}%
  }%
  \providecommand\transparent[1]{%
    \errmessage{(Inkscape) Transparency is used (non-zero) for the text in Inkscape, but the package 'transparent.sty' is not loaded}%
    \renewcommand\transparent[1]{}%
  }%
  \providecommand\rotatebox[2]{#2}%
  \ifx\svgwidth\undefined%
    \setlength{\unitlength}{915.2bp}%
    \ifx\svgscale\undefined%
      \relax%
    \else%
      \setlength{\unitlength}{\unitlength * \real{\svgscale}}%
    \fi%
  \else%
    \setlength{\unitlength}{\svgwidth}%
  \fi%
  \global\let\svgwidth\undefined%
  \global\let\svgscale\undefined%
  \makeatother%
  \begin{picture}(1,1.12125677)%
    \put(0,0){\includegraphics[width=\unitlength]{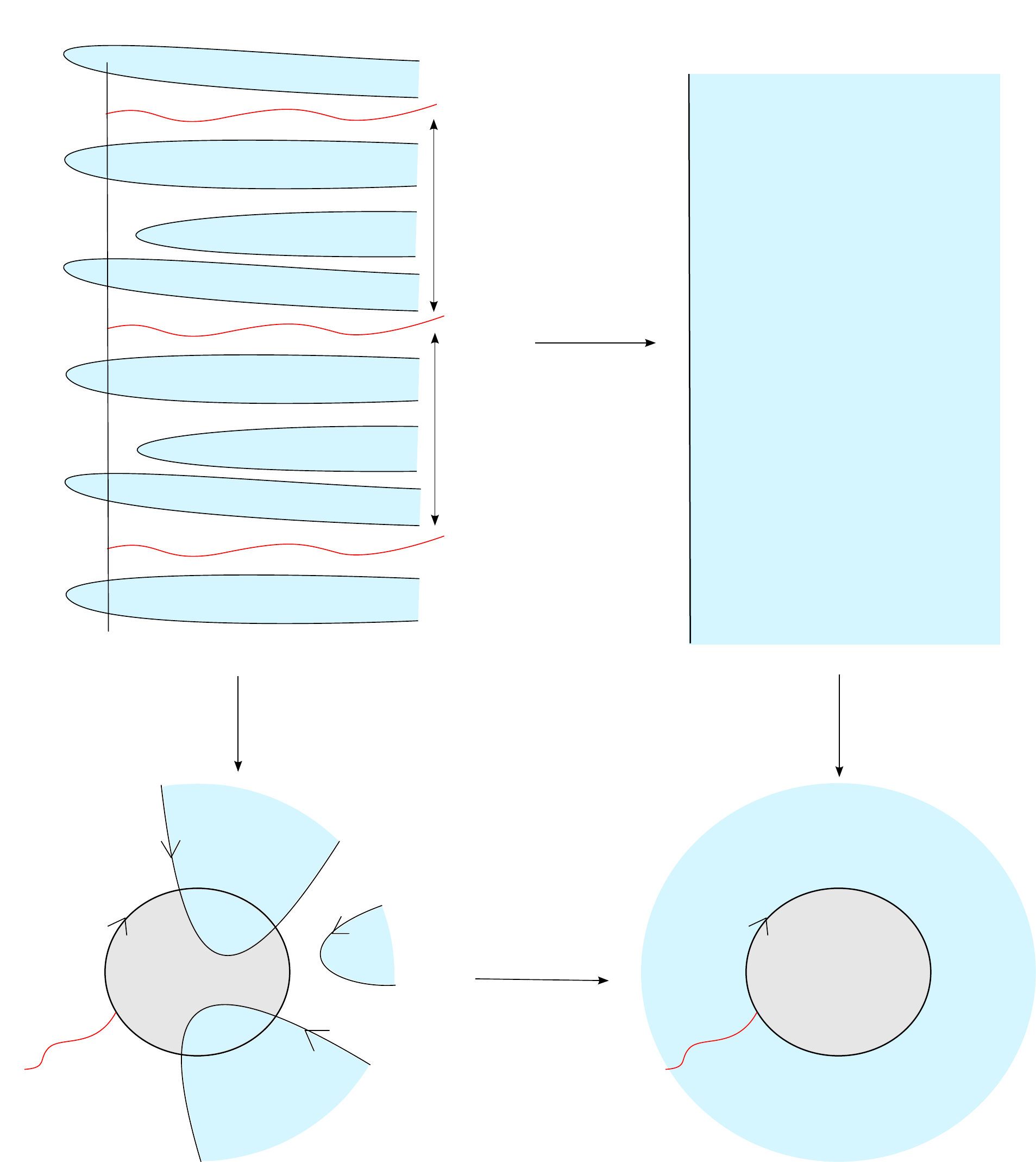}}%
    \put(0.79534245,1.10465697){\color[rgb]{0,0,0}\makebox(0,0)[lb]{\smash{ $\widetilde{H}$}}}%
    \put(0.54785279,0.8050335){\color[rgb]{0,0,0}\makebox(0,0)[lb]{\smash{$\widetilde{f}$}}}%
    \put(0.23935359,1.10797415){\color[rgb]{0,0,0}\makebox(0,0)[lb]{\smash{ $\widetilde{\mathcal{T}}$}}}%
    \put(0.50611888,0.19592467){\color[rgb]{0,0,0}\makebox(0,0)[lb]{\smash{$f$}}}%
    \put(0.15646853,0.42144914){\color[rgb]{0,0,0}\makebox(0,0)[lb]{\smash{\ss $\exp$}}}%
    \put(0.82954545,0.42319739){\color[rgb]{0,0,0}\makebox(0,0)[lb]{\smash{\ss $\exp$}}}%
    \put(0.14073427,0.19417641){\color[rgb]{0,0,0}\makebox(0,0)[lb]{\smash{$D$}}}%
    \put(0.22465035,0.29732327){\color[rgb]{0,0,0}\makebox(0,0)[lb]{\smash{\ss $T_1$}}}%
    \put(0.81381119,0.28333725){\color[rgb]{0,0,0}\makebox(0,0)[lb]{\smash{\small $\C\setminus  \overline{D}$}}}%
    \put(0.3365681,0.20066567){\color[rgb]{0,0,0}\makebox(0,0)[lb]{\smash{\ss $T_2$}}}%
    \put(0.22248726,0.05843503){\color[rgb]{0,0,0}\makebox(0,0)[lb]{\smash{\ss $T_3$}}}%
    \put(0.13887357,0.53694392){\color[rgb]{0,0,0}\makebox(0,0)[lb]{\smash{\scriptsize $\Tt_{1,-1}$}}}%
    \put(0.24233649,0.63193075){\color[rgb]{0,0,0}\makebox(0,0)[lb]{\smash{\scriptsize $\Tt_{3,0}$}}}%
    \put(0.23390294,0.68513884){\color[rgb]{0,0,0}\makebox(0,0)[lb]{\smash{\scriptsize$\Tt_{2,0}$}}}%
    \put(0.14198865,0.75137258){\color[rgb]{0,0,0}\makebox(0,0)[lb]{\smash{\scriptsize $\Tt_{1,0}$}}}%
    \put(0.22748282,0.83755053){\color[rgb]{0,0,0}\makebox(0,0)[lb]{\smash{\scriptsize$\Tt_{3,1}$}}}%
    \put(0.24385604,0.8906401){\color[rgb]{0,0,0}\makebox(0,0)[lb]{\smash{\scriptsize$\Tt_{2,1}$}}}%
    \put(0.14559761,0.95915318){\color[rgb]{0,0,0}\makebox(0,0)[lb]{\smash{\scriptsize$\Tt_{1,1}$}}}%
    \put(0.22624816,1.04837025){\color[rgb]{0,0,0}\makebox(0,0)[lb]{\smash{\ss $\Tt_{3,2}$}}}%
    \put(0.43129354,0.59953073){\color[rgb]{0,0,0}\makebox(0,0)[lb]{\smash{\ss $\delta_0$}}}%
    \put(0.42865332,0.81264801){\color[rgb]{0,0,0}\makebox(0,0)[lb]{\smash{\ss $\delta_1$}}}%
    \put(0.4226708,1.01549079){\color[rgb]{0,0,0}\makebox(0,0)[lb]{\smash{\ss $\delta_2$}}}%
    \put(0.43049579,0.92256981){\color[rgb]{0,0,0}\makebox(0,0)[lb]{\smash{$\widetilde{S}_1$}}}%
    \put(0.43161646,0.71636576){\color[rgb]{0,0,0}\makebox(0,0)[lb]{\smash{$\widetilde{S}_0$}}}%
    \put(0.03153577,0.1190464){\color[rgb]{0,0,0}\makebox(0,0)[lb]{\smash{$\delta$}}}%
  \end{picture}%
\endgroup%
\end{center}
\caption{\small Logarithmic coordinates. The lift $\widetilde{f}$ is only defined on the tracts. Each of them is mapped conformally onto the set $\widetilde{H}$.}\label{logcoordinatesnew2}
\end{figure}

The following properties hold (see \cite{R3S}):
\begin{itemize}
\item $\Ht$ is a $2\pi i$ periodic domain containing a right half plane;
\item every $\Tta$ is an unbounded Jordan domain with real parts bounded from  below, but unbounded from above (i.e., unbounded to the right);
\item the $\Tta$ accumulate only at infinity;
\item $\TTt$ is invariant under translation by $2\pi i $.
\end{itemize}

\pagebreak 

\noindent\emph{Lift of $f$}

Because $f\circ\exp$ is a universal covering on each $\Tta$, it is possible to lift $f$ to a continuous function  $\ft:\TTt\ra\Ht$ which makes the following diagram commute.
\begin{displaymath}
\xymatrix
{
\TTt \ar[r]^{\ft} \ar[d]_\exp &  \Ht \ar[d]^{\exp} \\
\TT \ar[r]^{f} & \C\setminus \Dbar 
}
\end{displaymath}
On each of the infinitely many connected  components $\Tta$ of $\TTt$, the lift $\ft|_\Tta$ is defined up to translation by multiples of $2\pi i$; so $\ft:\TTt\ra\Ht$ is defined up to infinitely many constants.
Let us fix any choice of $\ft$ on each of the tracts $\Tt_{\alpha,0}$, and then choose the remaining constants so that $\ft$ is $2\pi i $ periodic. 

The preimages of $\{\dtj\}$ under  $\widetilde{f}$ divide each $\Tta$ into \emph{fundamental domains} $\Ft_{\alpha, i,j}$. Choose the labeling of fundamental domains so that  $\ft(\Ft_{\alpha,i,j}))=\St_j$ for each $\alpha,i$.
Because we chose the lift $\ft$ to be periodic, the labeling on the fundamental domains for $\ft$ induces a well defined labeling for the fundamental domains for $f$, in such a way that $\exp(\Ft_{\alpha, i, j})= F_{\alpha,j}$ for each $i$ (see Figure~\ref{Label}).

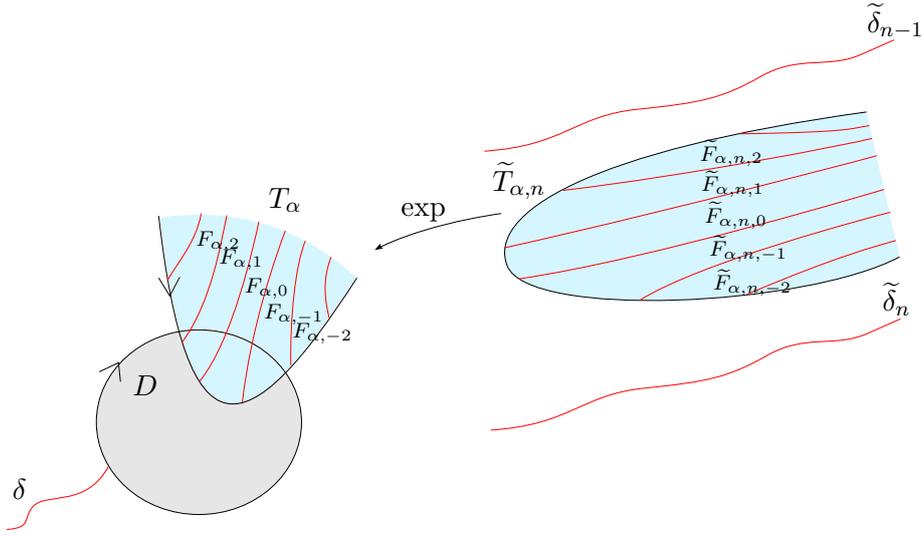
\begin{figure}[hbt!]
\begin{center}
\def\svgwidth{0.8\textwidth}
\begingroup%
  \makeatletter%
  \providecommand\color[2][]{%
    \errmessage{(Inkscape) Color is used for the text in Inkscape, but the package 'color.sty' is not loaded}%
    \renewcommand\color[2][]{}%
  }%
  \providecommand\transparent[1]{%
    \errmessage{(Inkscape) Transparency is used (non-zero) for the text in Inkscape, but the package 'transparent.sty' is not loaded}%
    \renewcommand\transparent[1]{}%
  }%
  \providecommand\rotatebox[2]{#2}%
  \ifx\svgwidth\undefined%
    \setlength{\unitlength}{1124.19941406bp}%
    \ifx\svgscale\undefined%
      \relax%
    \else%
      \setlength{\unitlength}{\unitlength * \real{\svgscale}}%
    \fi%
  \else%
    \setlength{\unitlength}{\svgwidth}%
  \fi%
  \global\let\svgwidth\undefined%
  \global\let\svgscale\undefined%
  \makeatother%
  \begin{picture}(1,0.54305336)%
    \put(0,0){\includegraphics[width=\unitlength]{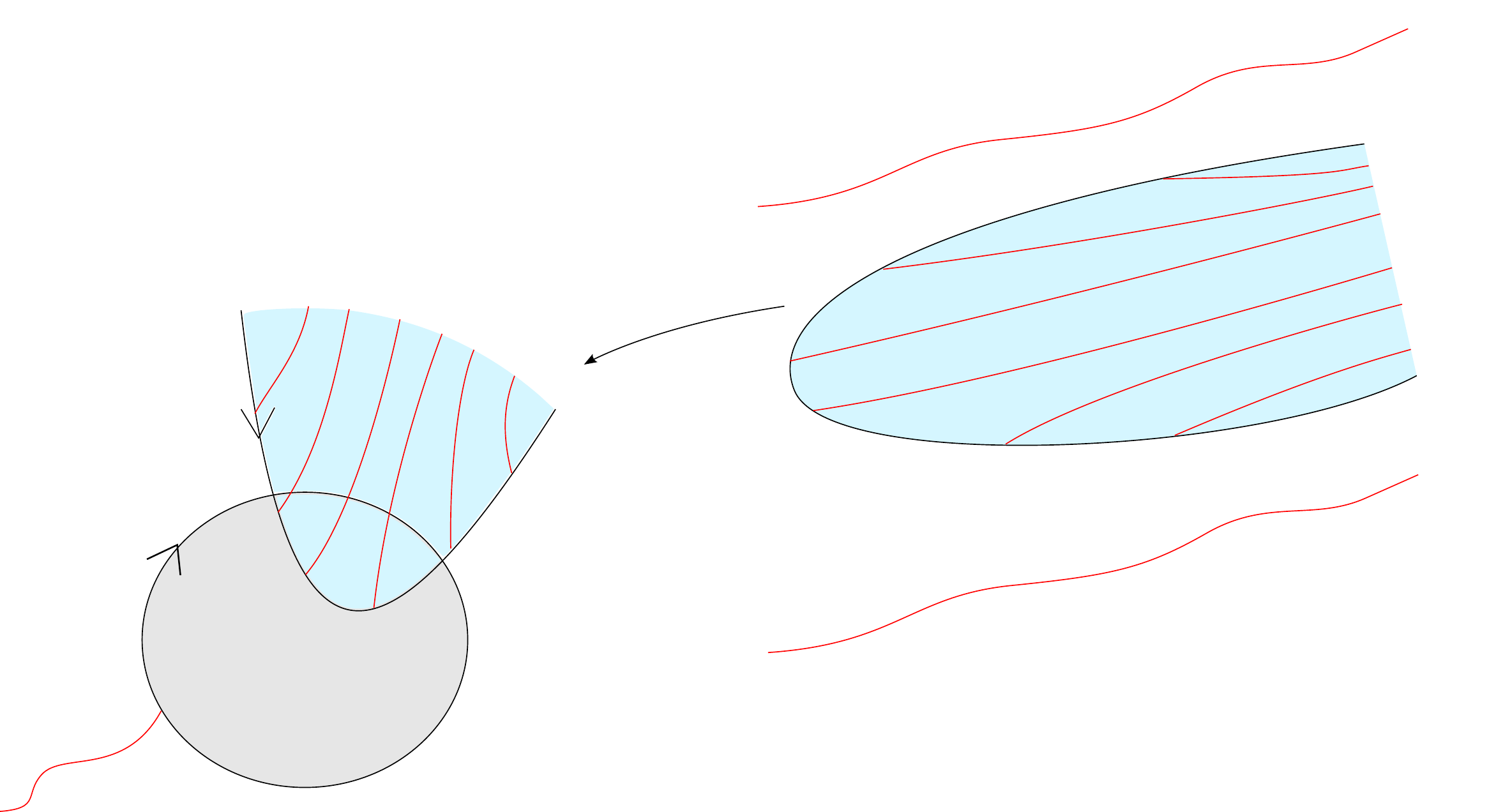}}%
    \put(0.41835568,0.33553515){\color[rgb]{0,0,0}\makebox(0,0)[lb]{\smash{$\exp$}}}%
    \put(0.13357697,0.14375095){\color[rgb]{0,0,0}\makebox(0,0)[lb]{\smash{$D$}}}%
    \put(0.27791875,0.34438555){\color[rgb]{0,0,0}\makebox(0,0)[lb]{\smash{$T_\alpha$}}}%
    \put(0.92935945,0.23394564){\color[rgb]{0,0,0}\makebox(0,0)[lb]{\smash{$\dt_{n}$}}}%
    \put(0.7404581,0.32864089){\color[rgb]{0,0,0}\makebox(0,0)[lb]{\smash{\ss $\Ft_{\alpha,n,0}$}}}%
    \put(0.73701754,0.36195675){\color[rgb]{0,0,0}\makebox(0,0)[lb]{\smash{\ss $\Ft_{\alpha,n,1}$}}}%
    \put(0.7355553,0.39492855){\color[rgb]{0,0,0}\makebox(0,0)[lb]{\smash{\ss $\Ft_{\alpha,n,2}$}}}%
    \put(0.74540385,0.29288799){\color[rgb]{0,0,0}\makebox(0,0)[lb]{\smash{\ss $\Ft_{\alpha,n,-1}$}}}%
    \put(0.74974753,0.25676236){\color[rgb]{0,0,0}\makebox(0,0)[lb]{\smash{\ss $\Ft_{\alpha,n,-2}$}}}%
    \put(0.9127875,0.53224011){\color[rgb]{0,0,0}\makebox(0,0)[lb]{\smash{$\dt_{n-1}$}}}%
    \put(0.51545453,0.36269909){\color[rgb]{0,0,0}\makebox(0,0)[lb]{\smash{$\Tt_{\alpha,n}$}}}%
    \put(0.25232777,0.25337102){\color[rgb]{0,0,0}\makebox(0,0)[lb]{\smash{\ss $F_{\alpha,0}$}}}%
    \put(0.22523118,0.2848556){\color[rgb]{0,0,0}\makebox(0,0)[lb]{\smash{\ss $F_{\alpha,1}$}}}%
    \put(0.27286809,0.22608147){\color[rgb]{0,0,0}\makebox(0,0)[lb]{\smash{\ss $F_{\alpha,-1}$}}}%
    \put(0.00623133,0.03267082){\color[rgb]{0,0,0}\makebox(0,0)[lb]{\smash{$\delta$}}}%
    \put(0.30321648,0.20515154){\color[rgb]{0,0,0}\makebox(0,0)[lb]{\smash{\ss $F_{\alpha,-2}$}}}%
    \put(0.20078503,0.30369253){\color[rgb]{0,0,0}\makebox(0,0)[lb]{\smash{\ss $F_{\alpha,2}$}}}%
  \end{picture}%
\endgroup%
\end{center}
\caption{\small Labeling of the fundamental domains.}\label{Label}
\end{figure}

\noindent\emph{Vertical order}

Because the sets $\Tta$ are unbounded to the right, each $\Tta$ divides any right half plane  sufficiently far to the right, into at least two connected components, one of which contains points with arbitrarily large imaginary part {and bounded real part} (that is  said to be \emph{above} $\Tta$  or $\succ \Tta$) and another one which contains points with arbitrarily small imaginary part {and bounded real part} (that is  said to be \emph{below} $\Tta$ or $\prec \Tta$). This introduces  a natural order $\prec$ on the $\Tta$,  called \emph{vertical order at infinity}. We may assume that $\Ttb\prec\Tta$ if and only if $\beta<\alpha$ (for a fixed $i$). 
Observe also that $\widetilde{T}_{\alpha,i}\prec \widetilde{T}_{\beta,j}$ whenever $i<j$; in this case, the order is just induced by branches of the logarithm. 
The vertical order  on the $\Tta$ induces via the exponential map a  \emph{cyclic order at infinity} (still denoted by $\prec$) on the tracts $\Ta$.   

Observe that for any given tract $\Ta$, the fundamental domains $F_{\alpha,i}$ contained in $\Ta$ also have a well defined cyclic order;  this  induces a cyclic order on  the set of all fundamental domains.

This proves part $(a)$ of the Structural  Lemma~\ref{Technical}.


\subsection{Expansion near infinity}\label{Expansion near infinity}
This Section is devoted to prove that, when restricted to a finite number of fundamental domains and far enough to the right, $\ft$ is uniformly expanding (that is, part $(b)$ of Lemma~\ref{Technical}).  Let $C_R$ be the circle of radius $R$ centered at $0$.

 Let us recall that the density $\rho_\Hyp(z)$ of the hyperbolic metric in the right half plane $\Hyp$ is $\rho_\Hyp(z)=\frac{1}{\Re z}$. Also, it is a standard estimate that the hyperbolic density on any simply connected hyperbolic domain $\Omega$ is comparable with the inverse of the distance to the boundary. More precisely, 
\[\frac{1}{2\dist(z,\partial\Omega)}\leq\rho_\Omega(z)\leq \frac{2}{\dist(z,\partial\Omega)}.\]

{The intersection of any tract $\Tt$ in logarithmic coordinates with a vertical line has connected components with diameter at most  $2\pi$, so} the maximal distance of a point from the boundary is $\pi$ and on the tract we have the inequality
\begin{equation}\label{Basic estimate} \frac{1}{2\pi}\leq\rho_{\Tt}(x).\end{equation}
Therefore for any curve $\gamma$ in $\Tt$, $\ell_{\eucl}(\gamma)\leq 2\pi\ell_{\Tt}(\gamma)$.

\begin{prop}\label{Bounded slope}
If $\ft$ is a lift of a function  $f\in\Bt$, then  the tracts of $\ft$ satisfy the \emph{bounded slope condition}, which means that there exists $M_1, M_2$ such that for any $z, w$ belonging to the same tract,
\begin{equation}\label{Bounded slope equation}
|\Im z-\Im w|\leq M_1 \max\{\Re z, \Re w,0\}+ M_2. 
\end{equation} 
\end{prop}

\begin{proof} 
From  \cite[Corollary 5.8 (a) and (b)]{R3S} we have that if  $\Ht$ is a  right half plane  $\mathbb{H}_r$ with $r$ sufficiently large (i.e. choosing $D=\D_{e^r}$), and after conjugating by the translation $\zt \mapsto \zt - r$,  the tracts of the new map have bounded slope. But after these changes, the new tracts  are translations of subsets of the original ones. If these modified tracts have bounded slope, the original ones must have the same property (maybe after changing the constant $M_2$), keeping in mind  that the bounded slope property is not affected by removing compact parts of the tract (see  \cite[Definition 5.1]{R3S}  and the remark thereafter and see also remarks on page 85 about normalized functions).
\end{proof}

We are now ready to prove the following proposition. For any $R>0$ let $C_R$ and $D_R$ denote respectively the circle and the disk of radius $R$ centered at $0$.

\begin{prop}
\label{Almost straight} Let $f\in\Bt$, 
and $\FF$ be a finite union of fundamental domains for $f$. Then for any $R$ sufficiently large, $f^{-1}(C_R)\cap \FF$  is contained in $D_R$.
\end{prop}

\begin{proof}
Observe that it is enough to prove the claim for one fundamental domain $F$. 
Recall the notation from Section \ref{Logarithmic coordinates}, and let $\ft$ be the lift of $f$ defined on the tracts $\Tt_{\alpha,i}$ with the conventions adopted there. Let $\Ft$ be any fundamental domain for $\ft$ such that $\exp(\Ft)=F$; it is then sufficient to prove that for any $R'$ sufficiently large, and any  $z$ with  $\Re z=R'$, we have that  $\Re \left(\ft^{-1}(z)\cap\Ft \right)<R'$ (the set $\ft^{-1}(z)\cap\Ft$ is a single point).  The proposition would then hold with $R > e^{R'}$. 

 Let $\Tt$ be the tract containing $\Ft$, and let  $\St=\ft(\Ft)$. 
 Since $\ft:\Tt\ra\Ht$ is a biholomorphism, there is a well defined inverse $\ftmT$, which is an isometry between the hyperbolic metric on $\Ht$ and the hyperbolic metric in $\Tt$. Its restriction $\ftmF: \St \ra \Ft$  is hence also an isometry between the hyperbolic metric on $\Ht$ and the hyperbolic metric in $\Tt$. 
Let $\Hyp_r=\{z\in\C \mid \Re z>r\}$ be a right half plane for some $r>0$ such that $\Hyp_r\subset \Ht$; then $\rho_{\Hyp_r}\geq\rho_\Ht$.

Let $\Tt_*$ be a tract contained in $\St$ and fix a  base point $z_0\in\Tt_*\cap\Hyp_r$. 
We first show that  there exists a constant $C$ such that for any $w\in\St$ with $\Re w>\Re z_0$

\begin{equation}\label{Bounds on strips}
|\Im z_0-\Im w|\leq M_1 \Re w+ C. 
\end{equation}\noindent
Let $X$ be the unique unbounded component of $\St\setminus L_{z_0}$, where $L_{z_0}$ is the vertical line passing through $z_0$.  Let $d$ be the diameter of the bounded set $\St\setminus X$. 
Let $\Tt_+$ and $\Tt_{-}$ be the translates of $\Tt_*$ by $+2\pi i $ and  $-2\pi i $ respectively, and let $z_{\pm}=z_0\pm 2\pi i $. 

The set $X$ is contained in the unbounded component of $\mathbb{H}_{\Re z_0}\setminus (\Tt_+\cup\Tt_-)$ which is between $\Tt_+$ and $\Tt_-$, in the sense that for any $w\in X$ there exist $w_\pm$ in $\Tt_\pm$ with $\Re w=\Re w_\pm$ and $\Im w_-<\Im w<\Im w_+$.

By Proposition~\ref{Bounded slope},  $\Tt_+$ and $\Tt_{-}$ satisfy Equation~\ref{Bounded slope equation}. Let us assume that $\Im w>\Im z_0$ (the opposite case is symmetric). Then for all $w\in X$,
\begin{align*}
|\Im z_0-\Im w| & =\Im w-\Im z_0\leq \Im w_+-\Im z_0 
 \leq |\Im w_+-\Im z_+| + |\Im z_+ -\Im z_0| \\
& \leq 2\pi+ M_1 \Re w+ M_2.
\end{align*}
%
%
%
%
Finally if $w \in \St\setminus X$ then $|\Im z_0-\Im w| \leq d$. 
Equation (\ref{Bounds on strips}) then follows with $C=\max\{2\pi+M_2,d\}$.

Now let $z_0'$ be the  unique preimage of $z_0$ in $\Ft$.  We will show the claim by proving that for any $R'>1$, and any point $z\in\St$ with $\Re z=R'+\Re z_0$, we have $|z-z_0|\geq R'$ while $|\widetilde{f}^{-1}(z)\cap \Ft -z_0'| \sim \ln R'$. 
By direct computation, and using (\ref{Bounds on strips}) with $w=z$,
\begin{align*}
\dist_{\Ht}(z_0, z)& \leq \int_{\Re z_0}^{\Re z}\frac{1}{t-r}dt+\left|\int_{\Im z_0}^{\Im z}\frac{1}{\Re z-r}dt\right|=
\ln\left(\frac{ R'+\Re z_0-r}{\Re z_0-r}\right) +\frac{|\Im z_0-\Im z|}{R'+\Re z_0-r}\leq\\
&\leq \ln R'+k,
\end{align*}
where $k$ is a constant. 
Since $\ft$ is an isometry for the hyperbolic metric between $\Ht$ and $\Tt$, we have that 
\[
\dist_{\Tt}(z_0', f^{-1}(z)\cap \Ft)\leq \ln R'+ k, 
\]
 and by Equation~(\ref{Basic estimate}), $|z_0'- (f^{-1}(z)\cap \Ft ) |\leq 2\pi (\ln R'+ k)$.
\end{proof}

This concludes the proof of part (b) in Lemma~\ref{Technical}.

\subsection{Dynamic rays and their properties} \label{raysandprop}
Here we state a result from \cite{R3S} about existence of dynamic rays for $\ft$ (recall that $\ft$ is only defined on $\TTt$). Then, we prove a theorem about existence of dynamic rays for $f$ corresponding to part $(c)$ of the Structural Lemma \ref{Technical}.

We define the  escaping set  \[I(\ft):=\{\zt { \in \TTt}; \Re\ft^n(\zt)\ra\infty \text{ as } n\ra\infty \}.\]

The following definition of ray tail and dynamic ray is  from \cite[Definition 2.2]{R3S}.

\begin{defn}
Let $\ft$ be the lift of a function $f$ in class $\BBt$. A \emph{ray tail} for $\ft$ is an injective curve $\Gt:[t_0,\infty)\ra I(\ft)$ such that $\lim_{t\ra\infty}\Re \ft^n(\Gt(t))=+\infty$ for any $n>0$ and such that $\Re \ft^n(\Gt(t))\ra+\infty$ uniformly in $t$ as $n\ra\infty$.

Likewise ray tails can be defined for $f$. 
Observe that with this definition a  non-broken dynamic ray for  $f$ becomes a maximal injective curve $g(t):(0,\infty)\ra I(f)$ such that $g|_{[t,\infty)}$ is a ray tail for every $t>0$ and such that no iterate of $g(t)$ is a critical point. 
\end{defn}

Let us  define \[ J(\ft):=\{\zt  \in \TTt \mid  \ft^n(\zt) \text{ is defined for all $n$}\}.\]

Observe that $J(\ft)$ includes orbits which escape to infinity but also  any periodic orbit whose projection does not intersect $D$. 

For points $\zt\in J(\ft)$, we can naturally introduce symbolic dynamics using the labeling of the tracts (see \cite{DT}, \cite{R3S},\cite{Ba}): if  for $n\geq 0$ we have  $\ft^n({\zt})\in \Tt_{\alpha_n, i_n} $, we say that $\zt$ has \emph{ address}
\[\st:=\binom{\alpha_0}{i_0}\binom{\alpha_1}{i_1}... \binom{\alpha_n}{i_n} ...\]
and introduce the left-sided shift map $\sigma$ as
 \[
 \sigma:\binom{\alpha_0}{i_0}\binom{\alpha_1}{i_1}...\mapsto \binom{\alpha_1}{i_1}\binom{\alpha_2}{i_2}... .
 \] 
 A ray tail $\Gt$   for $\ft$   is called  a  \emph{\emph{ray tail} of address $\st$ and is denoted by $\Gt_\st$} if it  consists of points with address $\st$. Notice that this is well defined since every ray tail is in $J(\ft)$ and hence belongs entirely to a unique fundamental domain. 
 We define the set 
\[
J_{\st}(\ft):=\{\zt\in J(\ft) \mid \zt\text{ has address $\st$} \}
\] 
From the definition of $J_\st$, it follows that $\ft(J_\st)\subset J_{\sigma\st}$;
 in particular, if an address $\st$ is periodic of period $m$, we have that $\ft^m(J_\st)\subset J_\st$. 

The next theorem is the restatement of one of the main results from  \cite{R3S}.  Let us define the set 
\begin{equation}\label{JK}
 J^K_{\st}(\ft):=\{\zt\in J_\st(\ft) \mid  \Re \ft^n(\zt)\geq K, \text{\ for all $ n\in\N$} \}.
 \end{equation}

\begin{thm}[\bf Ray tails for $\ft$]\label{Transcendental rays}
Let $\ft$ be the lift of a function $f\in\BBt$. Then for any address $\st$  for which $J^K_{\st}(\ft)\neq\emptyset$, the set  $J_\st(\ft)$ contains a unique maximal   unbounded arc $\Gt_\st(t)$, {which consists of escaping points (except for at most its endpoint), where $t\in [0,\infty)$ or $t\in (0,\infty)$.}   For any $t_*>0$ the  restriction of  $\Gt_\st$ to the interval $[t_*,\infty]$  is a  ray tail  of address $\st$ for $\ft$.  
\end{thm}

Theorem~\ref{Transcendental rays} is not stated in this terms  in \cite{R3S}, so we sketch a proof using results and terminology from the source.
\begin{proof}From Theorem 5.6 and Lemma 5.7 in \cite{R3S}, $\ft$ satisfies a uniform linear head start condition (uniform head start and uniform bounded wiggling are equivalent by Proposition 5.4).
By Theorem 3.3 and Proposition 4.4 (b) in \cite{R3S} there exists a constant $K(\ft)$ such that if $K>K(\ft)$ and the set $J^K_{\st}(\ft)$ is non empty, then $J_\st(\ft)$ has a unique unbounded component, which is a closed arc tending to infinity on one side.  From their Corollary 4.5, this arc consists of escaping points  (except for at most the endpoint).   For any $t_*>0$ the  restriction of  $\Gt_\st$ to the interval $[t_*,\infty]$  satisfies the hypothesis of being a {ray tail} by the head start condition, and consists of points of address $\st$ by definition.  
\end{proof} 


We now establish a correspondence between existence of ray tails for $\ft$ and existence of ray tails for $f$. This discussion can been seen as a natural follow up of the results in \cite{R3S} but it is not contained there.

Given a point $z$ in the $f$-plane, whose iterates always belong to $\TT$, and such that $f^n(z)\in F_{\alpha_n,i_n}$, we define its \emph{address} $s$ as

\[
s=\binom{\alpha_0}{i_0}\binom{\alpha_1}{i_1}...\ . 
\]

The left-sided shift map $\sigma$ is well defined, and a point with address $s$ is mapped to a point with address $\sigma s$.

\begin{lem}[\bf Correspondence of addresses] \label{Correspondence of addresses} Let $f\in\Bt$, $\ft$ be a periodic lift of $f $ as described in {Section~\ref{Logarithmic coordinates}}. A point $\zt\in J(\ft)$ has  address $\st=\binom{\alpha_0}{i_0}\binom{\alpha_1}{i_1}...$ if and only if $z=\exp(\zt)$ has  address  $\s=\binom{\alpha_0}{i_1}\binom{\alpha_1}{i_2}...$.
\end{lem}
\begin{proof}
If $\zt\in \Tt_{\alpha_n,i_n}$ and $\ft(\zt)\in \Tt_{\alpha_{n+1},i_{n+1}}\subset \St_{i_{n+1}}$, then $\zt\in \Ft_{\alpha_n,i_n,i_{n+1}}$ by the choice of lift $\ft$. Hence $z=\exp(\zt)\in F_{\alpha_n, i_{n+1}}$ by the choice of labeling for the fundamental domains. 
For the converse, observe that $\exp^{-1}(F_{\alpha_n,i_{n+1}})=\{\Ft_{\alpha_n, m, i_{n+1}}\}_{m\in\Z}$ and $\Ft_{\alpha_n, m, i_{n+1}}\subset \Tt_{\alpha_n,m}$. 
\end{proof}

A ray tail $G$ of $f$ is said to have address $s$ and is denoted by $G_\s$  if the points $G(t)$ have address $\s$ for $t$ large enough (see Lemma \ref{asymptotically contained}). A dynamic ray is said to have address $\s$ 
whenever it contains a tail of address $\s$.

\begin{prop}[\bf Rays for $f$]\label{Rays for f} Let $f\in\Bt$, $\ft$ be a periodic lift of $f $ as described in Section~\ref{Logarithmic coordinates}. If a ray tail $G_{\st_m}$ exists, where $\st_m=\binom{\alpha_0}{m}\binom{\alpha_1}{i_0}...$, and $m\in\Z$, then there exists a unique dynamic  ray $g_\s$ with address  $\s=\binom{\alpha_0}{i_0}\binom{\alpha_1}{i_1}...$, and it contains the projection of $G_{\st_m}$ under the exponential map.
\end{prop}

\begin{proof} The image of any ray tail  $G_{\st_m}$ under the exponential map is a ray tail $G_{s}$ for $f$,  consisting of points of address $s$ by Lemma \ref{Correspondence of addresses}. 
{Then, by Zorn's lemma,  $G_{s}$ (possibly minus its endpoint)  is contained in a dynamic ray $g$ which by definition has address $s$. Hence  $g_s\neq\emptyset$. The ray is unique, by the uniqueness part in Theorem \ref{Transcendental rays}.}
\end{proof}

  The next proposition ensures that rays for $f$ exist for any address $s$ which contains only finitely many symbols, by showing that  the set $J^K_{\st}(\ft)$ defined in Equation~\ref{JK} is non empty. {The proof uses  ideas from \cite{DT}, results from \cite{R3S}, and Lemma~\ref{Technical} (b). This result was previously proven in \cite{DT} for all functions with a finite number of singular values whose tracts satisfy a certain geometric condition.}

\begin{prop}\label{Existence for finitely many symbols}
If $f\in\Bt$ and $s$ is an address which contains   only finitely many different symbols then there exists  a  unique {dynamic} ray $g_s$ with address $s$ for $f$.
\end{prop}
\begin{proof} 
Let $s=\binom{\alpha_0}{i_0}\binom{\alpha_1}{i_1}\ldots$, and let $J_s^K(f)$ be the set of points of address $s$ such that $|f^n(z)|\geq e^K$ for all $n\in\N$. By definition of $\ft$ and by Lemma~\ref{Correspondence of addresses}, $J_s^K(f)=\exp J_\st ^K(\ft)$ (as defined in Equation~\ref{JK}) for any $\st=\binom{\alpha_0}{m}\binom{\alpha_1}{i_0}...$, $m\in\N$. By Theorem 3.3 and Proposition 4.4 (b) in \cite{R3S} (see also the proof of Theorem~\ref{Transcendental rays} in this paper) there exists $K(\ft)$ such that if $K>K(\ft)$ and   $J_{\st}^K(\ft)$  in not empty then a ray tail of address $\st$ exists. 
We start by showing that for some  $K>K(\ft)$ the set  $J_{s}^K(f)$ is not empty provided $s$ contains only finitely many symbols. 
By definition a point $z$ has address $s$ if and only if $f^n(z)\in F_{\alpha_n, i_n}$. To simplify notation let us call $F_n:= F_{\alpha_n, i_n}$. If $s$ contains only finitely many different symbols, all the $F_n$ belong to some finite collection $\FF$ of fundamental domains. Let $R>e^{K(\ft)}$ be  such that Proposition~\ref{Almost straight} holds for the family $\FF$. Let $C_R$ be the circle of radius $R$, $D_R$ be the open disk of radius $R$, both centered at $0$. Note  that $F_n\cap C_R\neq\emptyset $ for all $F_n\in\FF$. For each $F_n\in\FF$ define $\psi_n$ as the unique inverse branch of $f$ mapping $\C\setminus{(\ov{D}\cup\delta)}\ra F_n$. By definition each $\psi_n$ is univalent. All $\psi_n$ also have the following property. Let $Y$ be any unbounded connected {subset of $\C \setminus (\overline{D} \cup \delta)$}; since  $\psi_n$ is continuous  and it is the inverse of an entire function,  its image $\psi_n(Y)$ is also  an  unbounded connected set. Then,  if  $Y\cap C_R\neq\emptyset$, we also have that $\psi_n (Y)\cap C_R \neq \emptyset$   because by Proposition~\ref{Almost straight}  $\psi_{n}(Y)\cap D_R\neq\emptyset$. 
 We refer to this property as \emph{$\psi_n$ preserves unboundedness and intersection with $C_R$.}
 
For each $N\in\N$ define the set 

$$X_N=\{z\in F_0| f^j(z)\in F_j \text{ and } |f^j(z)|\geq R\ \forall\; j=0\ldots N\}.$$

By construction, $\ov{X}_{N+1}\subset X_N$, where the closure is taken in $\C$. 
We now  show that $\bigcap_N X_N\neq\emptyset$ by  additionally showing  that $X_N\cap C_R\neq\emptyset$  for all $N$. 
 Observe that for each $N$ the set $X_N$ is obtained by intersecting $F_N$ with $\C\setminus D_R$, then applying $\psi_{N-1}$, then intersecting again with $\C\setminus D_R$, applying $\psi_{N-2}$ and so on. 
Fix $N$. Since $F_N$ is an unbounded connected set and $C_R$ is compact, $F_N\setminus D_R$ contains at least one unbounded connected component which intersects $C_R$ by definition. So since $\psi_{N-1}$  preserves unboundedness and intersection with $C_R$ {(as defined above)}, the set $\psi_{N-1}(F_N\setminus D_R)$ also contains at least one unbounded connected component intersecting $C_R$. By induction, this property is preserved at each step and $X_N$ also contains an unbounded connected component intersecting $C_R$. Observe that $X_{N+1}\cap C_R$ is compactly contained in  $X_{N}\cap C_R$,  so that $\bigcap_N X_N\neq\emptyset$.

Now let $z\in \cap_N X_N$. Since $R>e^{K}$,  $\bigcap_N X_N \subset J_s^{K}(f) $. By definition $z$ has address $s$, so by Lemma~\ref{Correspondence of addresses} (observe that having an address implies that $f^n(z)\in\TT$ for all $n\in\N$) there exists a point $\zt\in\exp^{-1}z$ with address $\st=\binom{\alpha_0}{0}\binom{\alpha_1}{i_0}...$ such that 
$\Re \ft^n(z)\geq K\geq K(\ft)$ for all $n$. So $J_\st^K(\ft)\neq\emptyset$ and there exists a ray tail for $\ft$ with address $\st$. By Proposition~\ref{Rays for f}, there exists a unique dynamic ray for $f$ with address  $\s$.
\end{proof}

\noindent
Finally we show that each fundamental domain  $F$  contains  asymptotically exactly one fixed ray. 
\begin{prop}[\bf Fixed rays in fundamental domains]\label{Fixed rays in fundamental domains} {Let $f\in\BBt$. }
For each fundamental domain $F$, there is a unique fixed  dynamic  ray $g_{F}$ such that $g_{F}(t)\in F$ for all sufficiently large $t$.
\end{prop} 

\begin{proof}
Existence and uniqueness follow from Proposition \ref{Existence for finitely many symbols}. Due to its address and uniqueness, the ray is invariant and  asymptotically contained in $F$.
\end{proof}
\noindent
This concludes the proof of the Structural Lemma~\ref{Technical}, part (c). 

\section{Location of interior fixed points. Proof of Proposition A}
\label{Where are interior periodic points?}

Let us now fix a function $f\in\BBt$, and restrict our discussion to fixed points and fixed rays. From now on, we assume that all fixed rays land.  Let $D$, $\delta$, etc. be as in Section~\ref{structure} (see Figure \ref{tractsandfd}).

We first prove a proposition about landing of fixed rays for fundamental domains which do not intersect $D$, and then collect some remarks about the structure of the plane in logarithmic coordinates and the distribution of the interior fixed points with respect to the partition of the plane induced by the fixed rays.

As tracts and fundamental domains accumulate only at infinity,  only finitely many fundamental domains can intersect $D$. On those not intersecting $D$, the dynamics is easy to study.  The following proposition {holds without the assumption that all fixed rays land} and is central to prove Proposition A. 
\begin{prop}[\bf Forced landing]\label{Forced landing}
 If a  fundamental domain  $F$ does not intersect $D$, it contains a unique fixed point $w$, which is repelling. Moreover the fixed ray $g_F$  {asymptotically contained in $F$} given by part $(c)$ in the Structural Lemma \ref{Technical} {is fully contained in $F$ and} lands at $w$. 
 \end{prop}
\begin{proof}
As $f:F\ra \C\setminus (\ov{D}\cup\delta)$ is univalent, there is a univalent branch of $f^{-1}$ say $\psi: \C\setminus (\ov{D}\cup\delta) \to F$. Since points in $\partial F$ (where the boundary is taken in $\C$) are mapped to $\C\setminus \overline{F}$, it follows that $\psi(F) \subset F$. Hence $\psi$ maps the topological disk $F$ inside itself. Since it is not an automorphism, it follows from the Denjoy-Wolff Theorem (see  \cite[Theorem 5.4]{Mi}) that there exist  a point $z_0\in \overline{F}$ such that $\psi^n(z)\to z_0$ as $n\to \infty$ uniformly on compact subsets of $ F$. 

Let us reparametrize the piece of fixed ray which is contained in $F$ as $g_F:(-M,\infty) \to \C$ with $M\in\R$ so that $f(g_F(t))=g_F(t+1)$. Since $g_F(t)$ is asymptotically contained in $F$ it follows that it is entirely contained in $F$, as shown by successive iterations of $\psi$ on any compact arc of $g_F$ of the form $[g_F(t),g_F(t+1)]$. In particular $M=\infty$. By the same argument, it follows that $\lim_{t\to -\infty} g_F(t)=z_0$. It remains to prove that $z_0\neq \infty$, and therefore $z_0\in F$ must be a fixed point, unique and repelling for $f$. 

Suppose otherwise and consider  a circle $C_R$ of radius $R$ such that  part $(b)$ of the Structural Lemma \ref{Technical} holds for $F$ and such that $g_F$ intersects $C_R$. Since $|g_F(t)| \to \infty$ as $t\to -\infty$, there exists   $t_*\in \R$ such that $g(t_*) \in C_R\cap F$ and $|g_F(t)|>R$ for all $t< t_*$. But $f(g_F(t_*-1))=g_F(t_*)$   contradicting Lemma \ref{Technical} (b).
\end{proof}

\begin{rem} Observe that the ray $g_F$ does not necessarily land alone; other fixed rays coming from fundamental domains which do intersect the disk might land together with $g_F$ (see Figure \ref{B}). 
\end{rem}

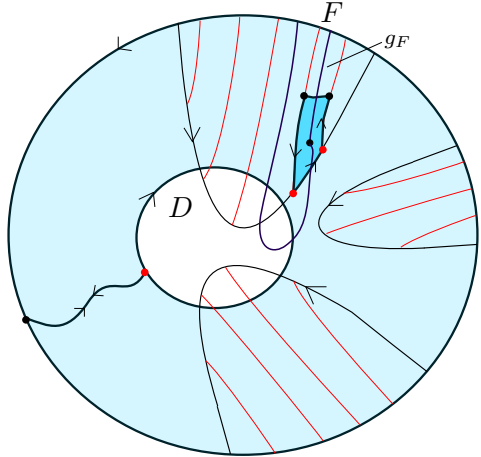
\begin{figure}[hbt!]
\begin{center}
\def\svgwidth{0.4\textwidth}
\begingroup%
  \makeatletter%
  \providecommand\color[2][]{%
    \errmessage{(Inkscape) Color is used for the text in Inkscape, but the package 'color.sty' is not loaded}%
    \renewcommand\color[2][]{}%
  }%
  \providecommand\transparent[1]{%
    \renewcommand\transparent[1]{}%
  }%
  \providecommand\rotatebox[2]{#2}%
  \ifx\svgwidth\undefined%
    \setlength{\unitlength}{492bp}%
    \ifx\svgscale\undefined%
      \relax%
    \else%
      \setlength{\unitlength}{\unitlength * \real{\svgscale}}%
    \fi%
  \else%
    \setlength{\unitlength}{\svgwidth}%
  \fi%
  \global\let\svgwidth\undefined%
  \global\let\svgscale\undefined%
  \makeatother%
  \begin{picture}(1,0.9773498)%
    \put(0,0){\includegraphics[width=\unitlength]{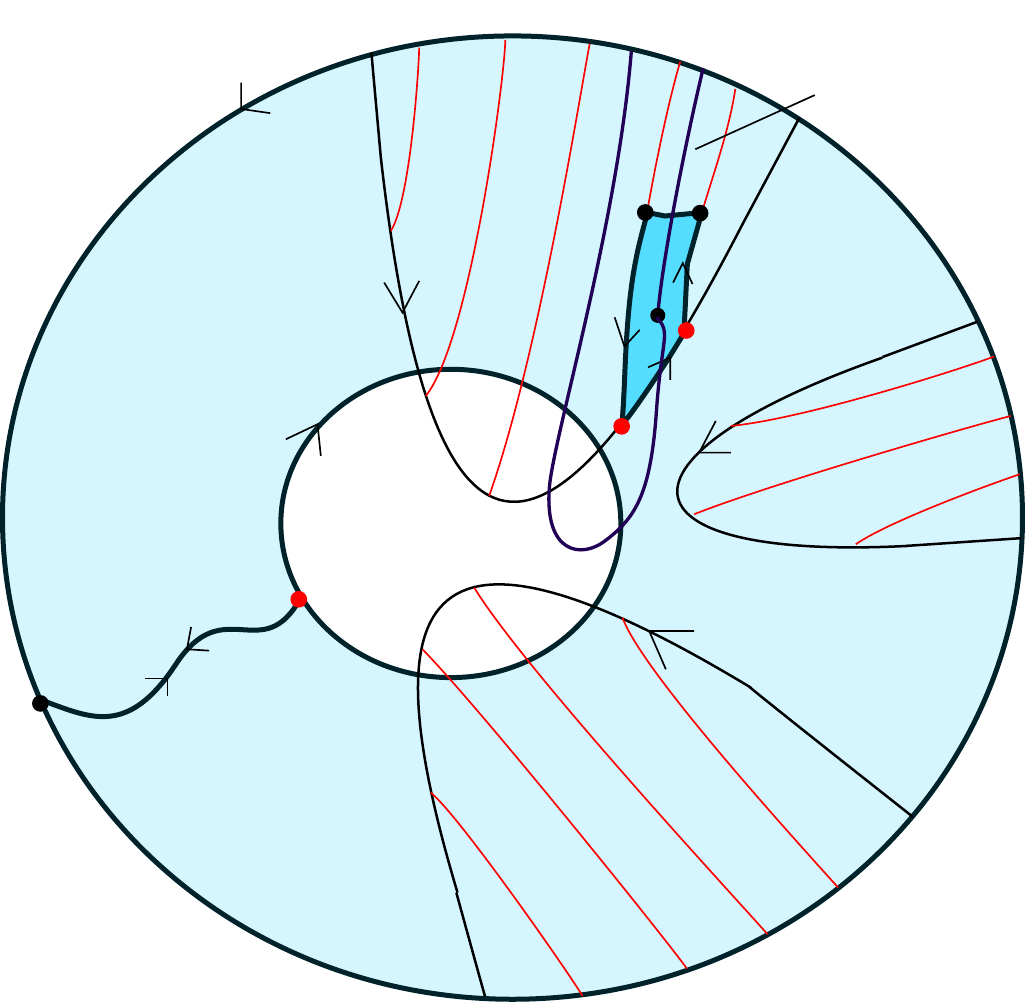}}%
    \put(0.34092141,0.51058995){\color[rgb]{0,0,0}\makebox(0,0)[lb]{\smash{$D$}}}%
    \put(0.66476966,0.92793415){\color[rgb]{0,0,0}\makebox(0,0)[lb]{\smash{$F$}}}%
    \put(0.80027099,0.88186368){\color[rgb]{0,0,0}\makebox(0,0)[lb]{\smash{\ss $g_F$}}}%
  \end{picture}%
\endgroup
\end{center}
\caption{\small If $F$ does not intersect $D$, the fixed ray $g_F$ lands at the unique fixed point in $F$ which must be repelling. However fixed rays from other fundamental domains which do intersect $D$, can also land at the same fixed point. There are only a finite number of those freelanders though. Light shaded we see the image of the dark shaded domain, {in the special case when $\psi(F\cap D_R) \subset D_R$.}}
\label{B}
\end{figure}

Recall that  a fixed point  is  an \emph{interior fixed point} if there are no fixed rays  landing at it.  Observe that interior fixed points cannot be in $\C\setminus(\ov{D}\cup \TT)$, because $f$ maps $\C\setminus(\ov{D}\cup \TT)$ to $\ov{D}$.
Also, they cannot be in a fundamental domain not intersecting the disk by Proposition~\ref{Forced landing}, {nor in $D\cap \TT$}.
So interior fixed points are contained either in $\ov{D}\setminus \TT$ or, {if they are outside $D$}, they lie in the finitely many fundamental domains which intersect $D$ (see Figure \ref{C}).

\begin{figure}[hbt!]
\begin{center}
\def\svgwidth{0.4\textwidth}
\begingroup%
  \makeatletter%
  \providecommand\color[2][]{%
    \errmessage{(Inkscape) Color is used for the text in Inkscape, but the package 'color.sty' is not loaded}%
    \renewcommand\color[2][]{}%
  }%
  \providecommand\transparent[1]{%
    \errmessage{(Inkscape) Transparency is used (non-zero) for the text in Inkscape, but the package 'transparent.sty' is not loaded}%
    \renewcommand\transparent[1]{}%
  }%
  \providecommand\rotatebox[2]{#2}%
  \ifx\svgwidth\undefined%
    \setlength{\unitlength}{453.975bp}%
    \ifx\svgscale\undefined%
      \relax%
    \else%
      \setlength{\unitlength}{\unitlength * \real{\svgscale}}%
    \fi%
  \else%
    \setlength{\unitlength}{\svgwidth}%
  \fi%
  \global\let\svgwidth\undefined%
  \global\let\svgscale\undefined%
  \makeatother%
  \begin{picture}(1,0.7799989)%
    \put(0,0){\includegraphics[width=\unitlength]{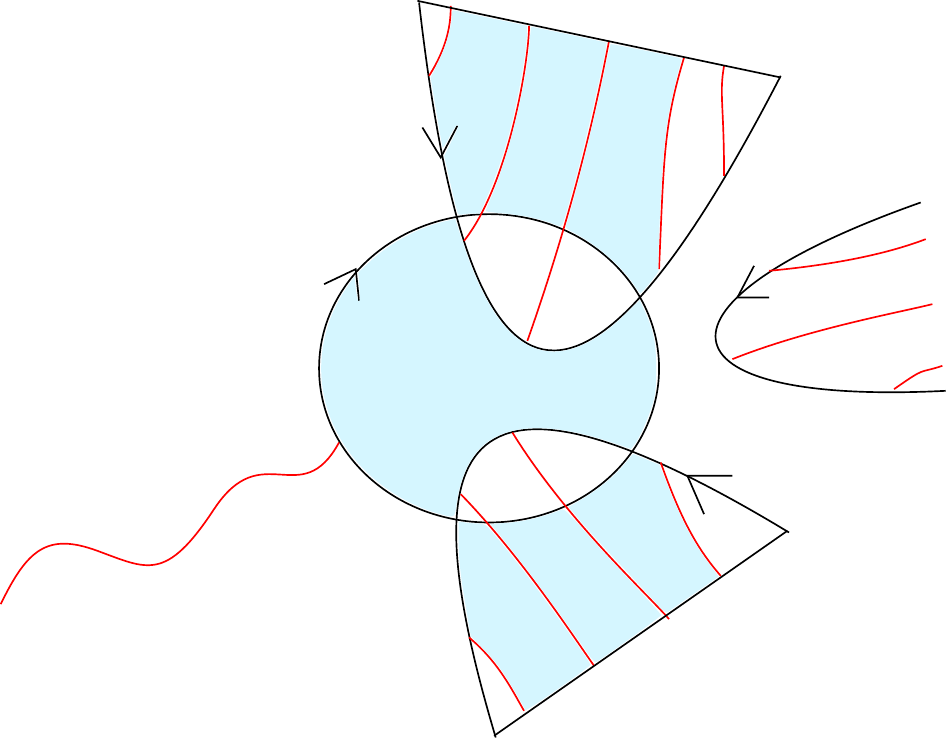}}%
    \put(0.40607963,0.4478446){\color[rgb]{0,0,0}\makebox(0,0)[lb]{\smash{$D$}}}%
    \put(0.03072856,0.22228152){\color[rgb]{0,0,0}\makebox(0,0)[lb]{\smash{$\delta$}}}%
  \end{picture}%
\endgroup%
\end{center}
\caption{\small Interior fixed points can only lie in the finite number of fundamental domains which intersect $D$, or in $\ov{D} \setminus\TT$.}\label{C}
\end{figure}
We now assume that all fixed rays land.  Recall that  $\Gamma$ is the graph formed by fixed rays together with their endpoints, and that the connected components of $\C\setminus \Gamma$ are called \emph{basic regions} for $f$. The following is a corollary of Proposition \ref{Forced landing}.
\begin{cor}
There are finitely many basic regions for fixed rays.
\end{cor}
\begin{proof}
As tracts do not accumulate in any compact set,  there are only finitely many tracts intersecting $D$. As preimages of  $\delta$ do not accumulate in any compact set either, there are finitely many fundamental domains intersecting $D$, hence by Proposition \ref{Forced landing} all but finitely many fixed rays must land at repelling fixed points in the interior of their respective fundamental domains. Therefore, only the finitely many remaining ones are free to land together with other rays. 
\end{proof}

This concludes the proof of Proposition A. The next two lemmas close this section.

\begin{lem}[\bf Crossing tracts]\label{Crossing tracts}
A  basic region cannot be fully contained inside a tract.
\end{lem}

\begin{proof}
This proof uses notation from Subsections~\ref{Logarithmic coordinates} and \ref{raysandprop}.
If a basic region is fully contained inside some tract $T_\alpha$, there exist at least two fixed rays $g^1$ and $g^2$ landing at  a common fixed point $z$ which are completely contained in $T_\alpha$. Let $\{z^k\}$ be the fiber of $z$ under the exponential map, with the convention that $z^k\in \St_k$. Because $z$ is fixed, {and because $\ft$ has been chosen to be $2\pi i -$periodic, } there is a unique $z^*\in \{z^k\}$ which is fixed. Let $\tilde{g}^1$ and  $\tilde{g}^2$ be the {unique} lifts of $g^1, g^2$ {which are fixed, and note that they are forced to}  land at the unique fixed point. This gives two fixed ray tails for $\ft$ which are contained in the same tract, {that is two unbounded arcs in the escaping set with the same address}, contradicting the  unicity in Theorem~\ref{Transcendental rays}.
\end{proof}

\begin{lem}
Every fixed ray pair separates the set of singular values.
\end{lem}
\begin{proof}
If a ray pair does not separate the set of singular values, the disk $D$ can be redefined so  as to  not intersect the ray pair. Then, the original ray pair is contained in $\C\setminus \Dbar$, hence its preimages (which includes the rays themselves) are fully contained inside tracts, contradicting Lemma~\ref{Crossing tracts}.
\end{proof}


\section{Tools: index and homotopies}\label{Index}

To count the number of fixed points of $f$ inside a basic region we will use the Argument Principle. The following are preliminary concepts about index of arcs and its invariance under homotopy. For a general theory about index see for example \cite{Why}.

In this section a {\em curve}  $\gamma$ is a continuous map $\gamma:[a,b]\ra\C$, with $a,b\in\R$, $a\leq b$. A curve is closed if $\gamma(a)=\gamma(b)$. An {\em arc} is a non-closed curve.  Abusing notation, the symbol $\gamma$ may denote both the function $\gamma(t)$ or the set $\gamma[a,b]$.  If $\gamma$ is closed and injective it is called a  {\em Jordan curve}.   For a curve $\gamma$ as above and a point $P\in\C\setminus\gamma$   there exists a continuous branch  $u(t)$ of the argument of $u(t)-P$. The {\em index} of  $\gamma$ with respect to  $P$  is then defined as 
$$
\ind(\gamma,P):= \frac{1}{2\pi}(u(b)-u(a)),
$$
a real number independent of the chosen branch of the argument. If $\gamma$ is rectifiable then the index can be expressed as the integral 
\[
 \ind(\gamma,P)=\Re\frac{1}{2\pi i}\underset{\gamma}\int\frac{1}{z-P}dz.  
 \]
By definition, if $\gamma$ is the union of consecutive arcs $\gamma_i$, then
$
\ind(\gamma,P)=\underset{i}\sum\ind(\gamma_i,P).
$ 

If  $\gamma$ is closed, $\ind(\gamma,P)$ is the winding number of $\gamma$ with respect to $P$ and it  is an integer. In this case it can be used to count the number of zeros of a holomorphic map inside a Jordan curve. 

\begin{thm}[\bf Argument principle]
 Let $\gamma$ be a Jordan curve bounding a region $\Omega$. Let $f$ be holomorphic in a neighborhood of $\overline{\Omega}$  such that $f(z)\neq 0$ for all points $z\in \gamma$. Let $Z(f)$ be the set of zeros of $f$. Then 
\[ 
\ind(f(\gamma),0) = \# (Z(f) \cap \Omega)
\]
counted with multiplicity.
\end{thm}
\begin{rem}
The Argument principle is most often stated for rectifiable curves to be able to use the integral expression of the index. But every Jordan curve $\gamma$ can be approximated by a piecewise linear Jordan curve which bounds the same zeros of $f$ and such that its image has the same index as $f(\gamma)$ {(see e.g. \cite[Sect V.3]{Why}).}
\end{rem} 

Given two curves  $\gamma, \sigma:[a,b]\to \C$, we denote by $\sigma-\gamma:[a,b]\to \C$ the curve defined by   $\sigma(t)-\gamma(t)$.
The following is a corollary of the Argument Principle, and will be our main tool.

\begin{cor}[\bf Counting fixed points]
Let $\gamma$ be a Jordan curve bounding a region $\Omega$. Let $f$ be holomorphic in a neighborhood of $\overline{\Omega}$  such that $f(z)\neq z$ for all points $z\in \gamma$. Let ${\rm Fixed}(f)$ be the set of fixed points of $f$. Then 
\[ 
\ind(f(\gamma)-\gamma,0) = \# (\rm{Fixed}(f) \cap \Omega)
\]
counted with multiplicity.
\end{cor}

In the remaining sections, we shall apply this corollary  repeatedly, and we will do it piecewise, i.e. breaking the Jordan curve in a finite number of arcs and computing the contribution of each piece to the total index. Hence given two curves $\gamma(t)$ and $\sigma(t):=f(\gamma(t))$, not necessarily closed, we will need to compute the index of the {\em subtraction curve} $\Gamma(t):=\sigma(t) -\gamma(t)$ with respect to 0, in a variety of different situations. This is a more general problem in which the role of $f$ is substituted by the relation between the respective parametrizations of $\sigma$ and  $\gamma$. 
In what follows, we collect some lemmas which give this index in terms of the relative positions and parametrizations of $\sigma$ and $\gamma$. They are based on the invariance of the index under homotopies satisfying certain conditions. 
We first recall the definition of (relative) homotopy.
\begin{defn}[\bf (Relative) homotopies]
Let $\gamma, \gammah:[a,b] \to X\subset \C$ be two curves. We say that $\gamma$ and $\gammah$ are {\em homotopic in $X$} if there exists a continuous map 
\[
\begin{array}{rccl}
H:& [a,b]\times [0,1] &\longrightarrow &X \\
 & (t,s) & \longmapsto & H(t,s)=:\gamma_s(t)
 \end{array}
 \]
such that $\gamma_0(t)=\gamma(t)$ and $\gamma_1(t)=\gammah(t)$ for all $t\in[a,b]$. In particular, a curve is {\em homotopic in $X$ to a point $P_0$} if it can be continuously deformed {\em in $X$} to the constant curve equal to $P_0$. 
{We say that $\gamma$ and $\gammah$ are {\em homotopic  in $X$ relative to $T\subset [a,b]$} if for each $t\in T$, $\gamma_s(t)$ is constant for all $s\in [0,1]$. We then write $\gamma \sim \gammah$ rel $T$ in $X$. }

\end{defn}
\begin{lem}[\bf First Homotopy Lemma]\label{First Homotopy Lemma} \ \ 
\begin{itemize}
\item[\rm (a)]
Fix $P\in \C$ and  let $\gamma,\widehat{\gamma}:[0,1]\ra\C\setminus \{P\}$ be two curves homotopic in $\C\setminus \{P\}$, relative to  $t_0=0$ and $t_1=1$. 
Then
\[
\ind(\gamma,P)=\ind(\widehat{\gamma},P).
\]
\item[\rm (b)]
Let $\gamma,\widehat{\gamma},\sigma,\widehat{\sigma}:[0,1]\ra\C$ be curves such that $\gamma \sim \gammah$  rel $\{0,1\}$ and $\sigma \sim \sigmah$ rel  $\{0,1\}$, both homotopies  in $\C$. Assume further that $\sigma_s(t)-\gamma_s(t) \neq 0$ for all $s,t\in [0,1]$. Then 
\[
\ind(\sigma-\gamma,0)=\ind(\sigmah-\gammah,0).
\]\end{itemize}
\end{lem}
\noindent See Figure \ref{Hom}.  The proof is left as an exercise. 
Part (b) follows from  (a) using  that the subtraction of the two homotopies gives a homotopy between $\Gamma(t):=\sigma(t) - \gamma(t)$ and $\Gammah(t):=\gammah(t) -\sigmah(t)$ in $\C\setminus \{0\}$ relative to their common endpoints (i.e., to $\{0,1\}$). 

 We shall use part (b) repeatedly throughout this section and inside the proof of the Main Theorem.   Informally, it says that we can deform each curve separately as long as we do not create zeros of the subtraction curves $\Gamma_s$   at any step. Observe that a continuous series of  reparametrizations of a curve is a homotopy. For single curves, this has no effect in the index. But in the setup above, reparametrizing one of the curves does affect $\Gamma$, since it changes how the points in $\gamma$ and $\sigma$ correspond. 

\begin{figure}[htb!]
\centering
\includegraphics[width=0.35\textwidth]{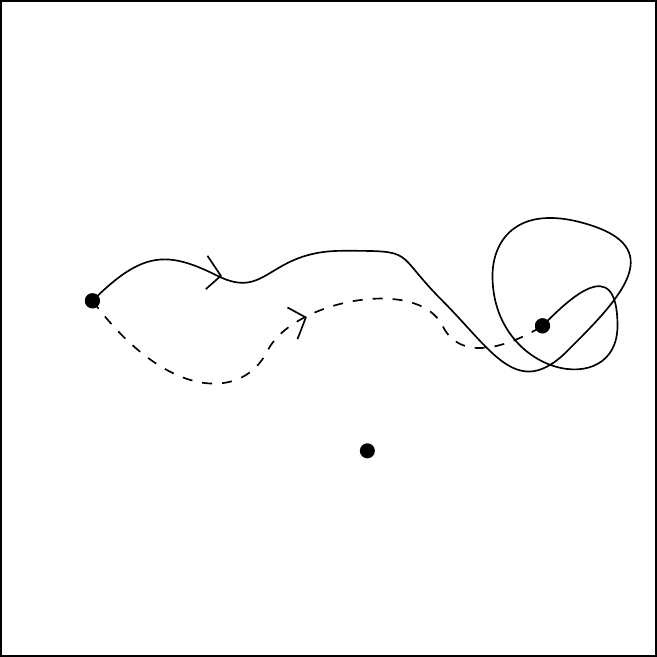} \hfil
\includegraphics[width= 0.35 \textwidth]{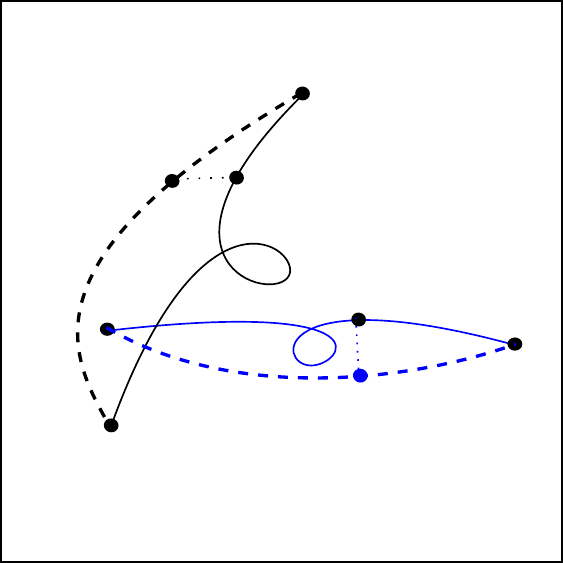}
\setlength{\unitlength}{0.83\textwidth}
\put(-0.74,0.1){$P$}
\put(-0.94,0.2){$P_0$}
\put(-0.64,0.23){$P_1$}
\put(-0.36,0.08){\scriptsize $\sigma(0)$}
\put(-0.21,0.37){\scriptsize $\sigma(1)$}
\put(-0.35,0.19){\scriptsize $\gamma(1)$}
\put(-0.05,0.18){\scriptsize $\gamma(0)$}
\put(-0.17,0.19){\scriptsize $\gamma(0.5)$}
\put(-0.17,0.115){\scriptsize $\gammah(0.5)$}
\put(-0.235,0.285){\scriptsize $\sigma(0.5)$}
\put(-0.37,0.285){\scriptsize $\sigmah(0.5)$}
\caption{\small Setup of the First Homotopy Lemma part (a) (left), {with $P_0:=\gamma(0)=\gammah(0)$},  and part (b) (right).  In this particular example, because of the way that $\gamma$ and $\sigma$ correspond, we may deform $\sigma$ and $\gamma$ (continuous)  into $\sigmah$ and $\gammah$ (dashed) without creating any zero of $\sigma_s(t)-\gamma_s(t)$ in the way. This means that $\ind(\sigma_s-\gamma_s,0)$ is constant for all $s$. The dotted curves denote $\gamma_s(0.5)$ and $\sigma_s(0.5)$ respectively.}
\label{Hom}
\end{figure}
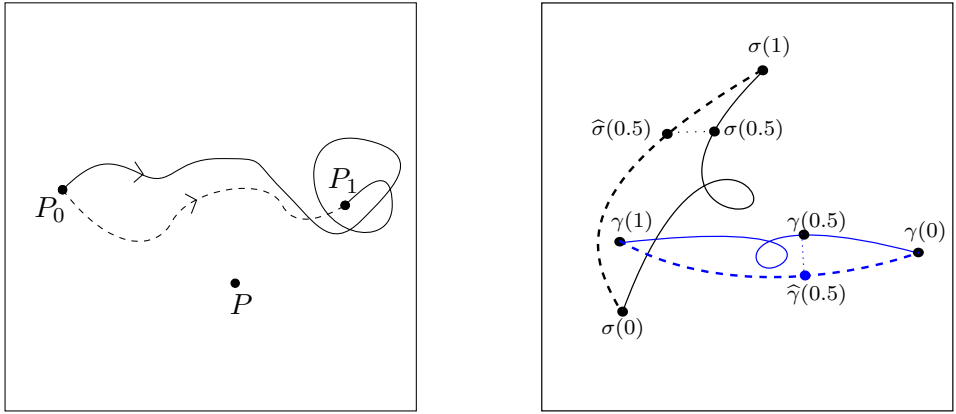

With the first Lemma as a tool we prove the following. See Figures \ref{Hom2a} and \ref{Hom2b}.

\begin{lem}[\bf Second Homotopy Lemma] \label{Third homotopy lemma}
Let  $\gamma, \sigma:[0,1] \to \C$ be  two curves such that $\sigma(0)=\sigma(1)=P\in\C$. Suppose one of the following occurs: 
\begin{itemize}
\item[\rm (a)] $\gamma \cap \sigma=\emptyset$; or
\item[\rm (b)] $\gamma \cap \sigma\neq \emptyset$ but 
\begin{itemize}
\item[\rm (i)] $\gamma(t)\neq \sigma(t)$ for all $t\in [0,1]$,  and $P\notin \gamma$;
\item[\rm (ii)] there exists a connected component $U$ of $\C\setminus \sigma$, such that $\gamma(0),\gamma(1) \in U$ and $\gamma(0,1) \subset \overline{U}$;
\item[\rm (iii)] if $z\in \partial U$ has more than one access from $U$, then $z\neq \gamma(t)$ for any $t\in[0,1]$.
\end{itemize}
\end{itemize}
Then 
 \[
 \ind(\sigma-\gamma,0)=\ind(\gamma,P) + N,
 \]
 where $N=\ind(\sigma,z)$ for all $z\in\gamma \setminus \sigma$.
\end{lem}

\begin{figure}[htb!]
\centering
\includegraphics[width=0.3\textwidth]{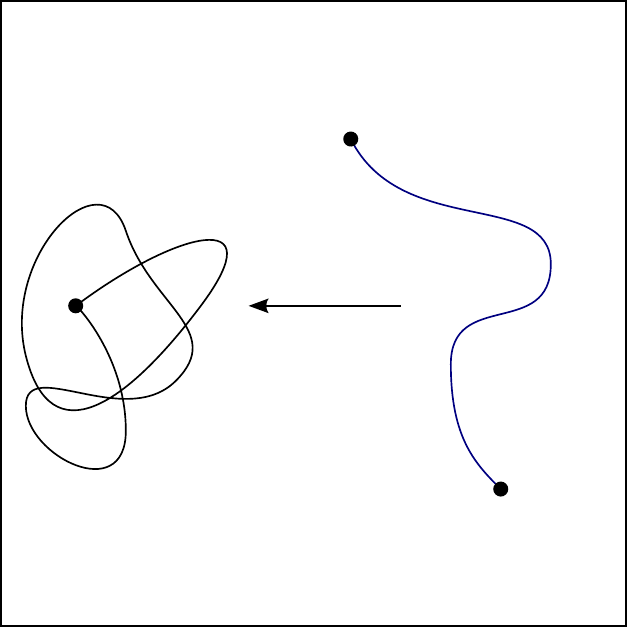} \hfil
\includegraphics[width= 0.3\textwidth]{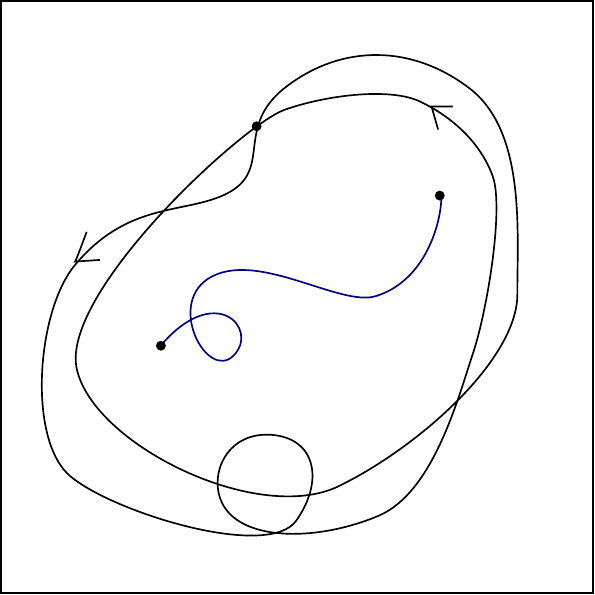}
\setlength{\unitlength}{0.83\textwidth}
\put(-0.82,0.25){$\sigma$}
\put(-0.86,0.2){\scriptsize $P$}
\put(-0.63,0.27){\b $\gamma$}
\put(-0.1,0.33){$\sigma$}
\put(-0.19,0.27){\scriptsize $P$}
\put(-0.11,0.17){\b $\gamma$}
\caption{\small Setup of the Second Homotopy Lemma part (a) with $N=0$ (left) and $N=2$ (right).}
\label{Hom2a}
\end{figure}

\begin{figure}[htb!]
\centering
\includegraphics[width=0.3\textwidth]{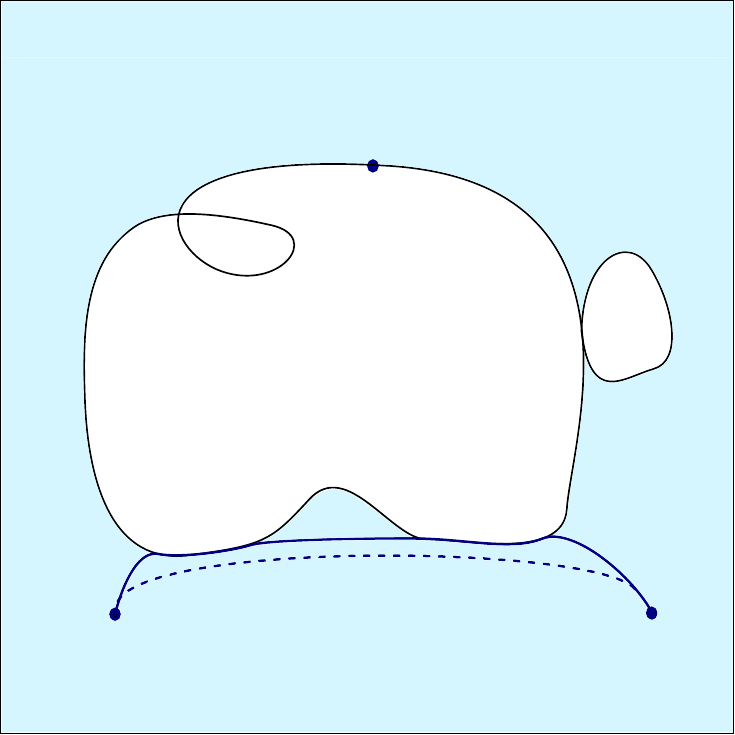} \hfil
\includegraphics[width= 0.3\textwidth]{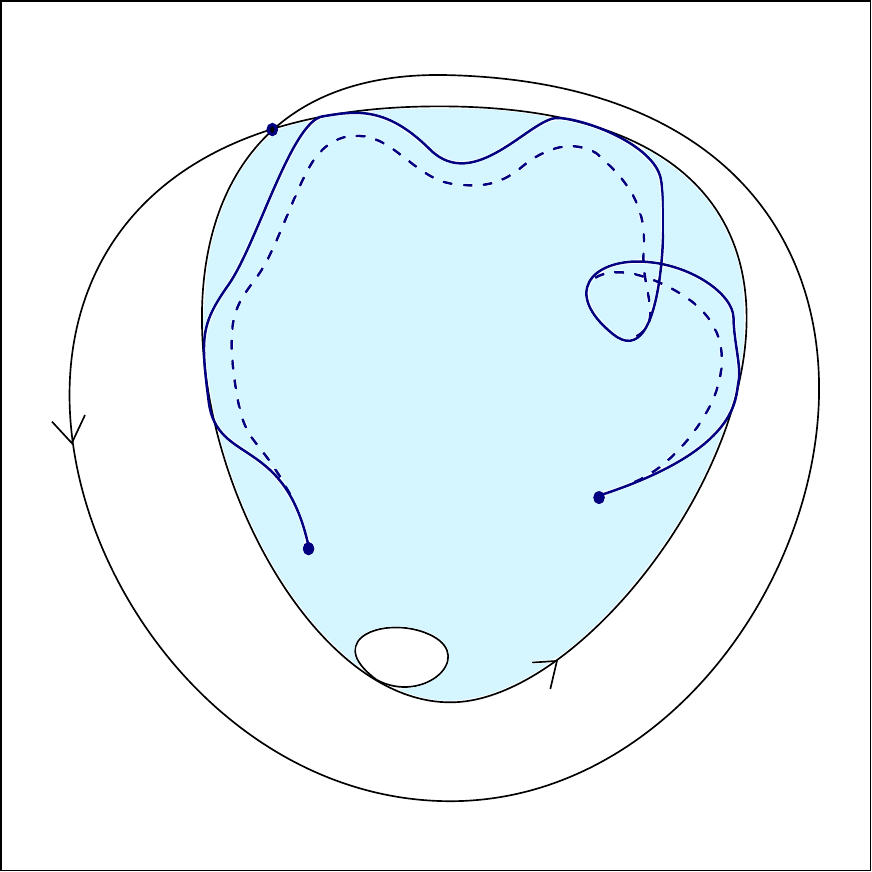}
\setlength{\unitlength}{0.83\textwidth}
\put(-0.83,0.26){$\sigma$}
\put(-0.72,0.29){\scriptsize $P$}
\put(-0.6,0.1){\b $\gamma$}
\put(-0.72,0.07){\b $\gamma_1$}
\put(-0.1,0.32){$\sigma$}
\put(-0.26,0.32){\scriptsize $P$}
\put(-0.105,0.14){\b $\gamma$}
\put(-0.18,0.26){\b $\gamma_1$}
\caption{\small Setup of the Second Homotopy Lemma part (b) with $N=0$ (left) and $N=2$ (right). The component $U$ is shadowed. The broken lines show the curve $\gamma_1$ (to be constructed in the proof) homotopic to $\gamma$ and not intersecting $\sigma$.}
\label{Hom2b}
\end{figure}
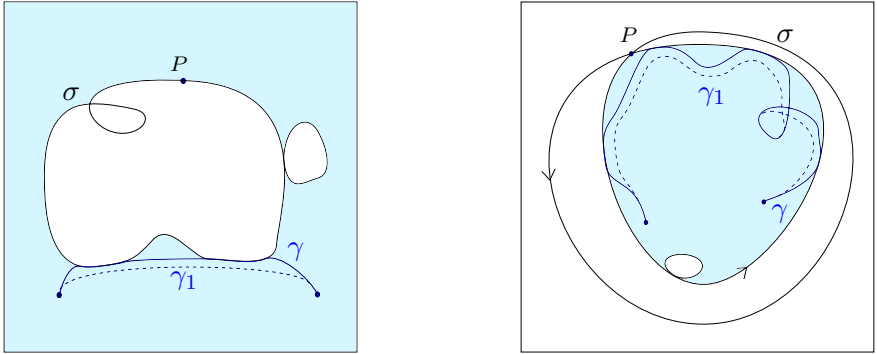

\begin{proof} 
{\bf Case (a):}  Since $\gamma$ and $\sigma$ do not intersect, all points in $\gamma$ must be in the same component $U$ of $\C\setminus \sigma$. Hence  the index $N$ is well defined for all $z\in\gamma$ and does  not depend on $z$. 

Observe that the curve $\sigma(t)-\gamma(t)$ starts with the vector $P-\gamma(0)$ and ends with the vector $P-\gamma(1)$. Up to an integer, the difference between the arguments of these two vectors is actually the index of $\gamma$ with respect to $P$, but also the index of $\sigma-\gamma$ with respect to $0$. Hence $ \ind(\sigma-\gamma,0) - \ind(\gamma,P) \in \Z$.  We will prove that this integer equals $N$  by proving that for every $\delta >0$ we can find $\epsilon>0$ and a curve  $\Gamma_\epsilon$ such that $\ind(\sigma-\gamma,0)=\ind(\Gamma_\epsilon,0)$ and 
\[
\left|\ind(\Gamma_\epsilon,0)-\ind(\gamma,P)-N \right|<\delta.
\]
The result then follows since the index must be an integer.

For every $\epsilon>0$, we first make a continuous series of reparametrizations $\gamma_s^\epsilon$ of $\gamma$, $s\in [0,1]$,  so that the final curve $\gamma_1^\epsilon$ is constant and equal to  $\gamma(0)$, for all $t\in [0, 1-\epsilon]$ and then ``runs'' to $\gamma(1)$ in the remaining $\epsilon$ time. An explicit  homotopy is for example 
\[
\gamma_s^\epsilon(t):=
\begin{cases}
\gamma(0) & \text{\ if $t\in [0,s(1-\epsilon)]$}\\
\gamma\left(\frac{t -s(1-\epsilon)}{1-s(1-\epsilon)}\right) &
\text{\ if $t\in [s(1-\epsilon),1]$}.
\end{cases}
\]

Now define  $\Gamma_\epsilon:=\sigma-\gamma_1^\epsilon$. Since the curves are disjoint at every step, $\Gamma_\epsilon$ is homotopic to $\sigma-\gamma$ (relative to  $\{0,1\}$)  in $\C\setminus \{0\}$. Observe that  $\Gamma_\epsilon(t)=\sigma(t)-\gamma(0)$ for all $t\in [0,1-\epsilon]$, hence the contribution to the index during this parameter interval is arbitrarily close to $N$ if $\epsilon$ is small enough. On the other hand, for $t\in [1-\epsilon,1]$ we have $\Gamma_\epsilon(t) \simeq P-\gamma(t)$, so the index contribution is  arbitrarily close to $\ind(\gamma,P)$ for $\epsilon$ sufficiently small. It follows that {$\ind(\Gamma_\epsilon,0)$}  is arbitrarily close to $\ind(\gamma,P)+N$ as we wanted to show.

\noindent {\bf Case (b):}  Observe that $N$ is also well defined in this case for the  points $z\in\gamma \setminus \sigma$: in fact $N(z)=\ind(\sigma,z)$ is constant in each component of $\C\setminus \sigma$, hence it is constant in $U$ and consequently, by (ii), does not depend on $z \in \gamma\setminus\sigma$.

In this setting the  trick of part (a) does not necessarily work since it is not clear that by reparametrizing $\gamma$ we do not create zeros of the subtraction curve. The idea is to make a homotopy which moves $\gamma$ away from $\sigma$. More precisely we will define a homotopy from $\gamma$ to a curve $\gamma_1$ in such a way that  $\ind(\sigma-\gamma,0)=\ind(\sigma-\gamma_1,0)$ and $\gamma_1 \cap \sigma=\emptyset$  and then  the result will follow from case (a).

To that end, observe that $U$ is a topological disk (in $\chat$) with  locally connected boundary, since $\sigma$ is connected and locally connected  (the image of a closed interval is locally connected by Hahn-Mazurkiewicz Theorem, and {so is  the boundary of} each component of $\C\setminus \sigma$   by Torhorst's Theorem  \cite{Why}). Let  $\phi:\ov{\D}\to \ov{U}$ be a Riemann map extended to the boundary and set $\gammah(t):=\phi^{-1}(\gamma(t)) \in \ov{\D}$, {which is a continuous curve by (iii)}. Observe that  $\phi^{-1}(\sigma\cap\gamma)$ is contained in  $\{|z|=1\}$.
Fix $0<\rho<1$ such that $\gammah(0),\gammah(1) \in \D_\rho$ and let $\epsilon <1-\rho$. For $r,s\in [0,1]$ let $l_s$ be the piecewise linear map which is {the identity}  on $[0,\rho]$ and affine from $[\rho,1]$ onto $[\rho, 1-s\epsilon]$. Observe that $l_s$  depends continuously on $s$. Now define $h_s : \ov{\D} \to \ov{\D}_{1- s \epsilon}$ as $h_s(r e^{i\theta})=l_s(r) e^{i\theta}$. Then $h_s$ is an angle-preserving  homeomorphism which is the identity on $\ov{\D}_\rho$, and  sends the annulus $\ov{\D}\setminus \D_\rho$ to the annulus $\ov{\D}_{1-s\epsilon} \setminus \D_\rho$, satisfying that $h_0={\rm Id}$ on $\D$.
Hence the map 
\[
\widehat{H}(t,s):=\gammah_s(t):=h_s(\gammah(t))
\]
is continuous in both variables and a homotopy between $\gammah$ and $\gammah_1$ relative  to $\{0,1\}$. It follows from {the continuity of $\varphi$ and $\gammah$}  that 
\[
 H(t,s):=\gamma_s(t):=\phi (\gammah_s(t))
 \]
defines a homotopy between $\gamma$ and $\gamma_1$ relative to its endpoints. 
It is clear that $\sigma(t)-\gamma_s(t)\neq 0$ for all $s,t\in [0,1]$ hence, by Lemma~\ref{First Homotopy Lemma}, we deduce that  $\ind(\sigma-\gamma,0)=\ind(\sigma-\gamma_1,0)$.   Since $\sigma$ is disjoint from $\gamma_1$, we obtain from case (a) that $\ind(\sigma-\gamma_1,0)=\ind(\gamma_1,P) + N'$ where $N'=\ind(\sigma,z)$ for all $z\in\gamma_1$. Clearly we have that $N=N'$ because $\gamma_1 \subset U$. Finally observe that $  \ind(\gamma_1,P) =\ind(\gamma,P) $ because the homotopy avoids $P$ and the endpoints remain fixed.

\end{proof}


\section{Global counting:  proof of Theorem B}
\label{Global counting}

Recall that, given a fundamental domain $F$ and its fixed dynamic  ray $g_F$ as given by the Structural Lemma \ref{Technical}, we say that $g_F$ \emph{lands alone} if it lands, and no other fixed ray has the same landing point.  {By the snail lemma,} a {fixed} ray may land alone only at a repelling or  parabolic fixed point. By Proposition~\ref{Forced landing}, there are only finitely many fundamental domains whose fixed rays do not land alone  at repelling fixed points. 

\begin{defn}[\bf Full and complete collections]\label{Full and complete collections}
A collection $\FFa$ of fundamental domains in $\Ta$ is  {\em full} if for any $F,F' \in \FFa$, the collection contains any other  fundamental domain in $\Ta$ which is in between $F$ and $F'$ with respect to the order of the fundamental domains in $\Ta$. A collection $\FF$ of fundamental domains in $\TT$ is {\em full} if its restriction to each one of the tracts is full. We say that $\FF$ is {\em complete} if it contains all fundamental domains whose fixed ray does not land alone at a repelling fixed point, all fundamental domains which intersect $D$  and all  {fundamental domains which intersect any fixed ray which is not fully contained inside its fundamental domain.} 
\end{defn}
Informally, a collection is full if it has no gaps (in each tract). {Also, since each fixed ray lands, it only intersects finitely many fundamental domains before landing.} Theorem B is an easy  consequence of the following theorem.

\begin{thm}[\bf Global counting]
\label{Global counting theorem} 
Let $f\in \BBt$ and assume that all fixed rays land. Let  $\FF$ be a finite collection of fundamental domains which is full and complete.  Let $N=\#\FF$ and $\GG$ be the collection of {the} $N$ fixed rays which are asymptotically contained in some $F\in \FF$.
Then there are exactly $N+1$ fixed points counted with multiplicity, which are either landing points of some $g\in \GG$, or interior fixed points. 
\end{thm}

Let us first show how Theorem B follows from Theorem \ref{Global counting theorem}. 

\begin{proof}[Proof of Theorem B]
Let $G$ be the finite collection of $L$  fixed rays given in Theorem B, which includes all rays which do not land alone at repelling fixed points.  Let $\FF'$ be the finite collection of $L$  fundamental domains in which the rays in $G$ are asymptotically contained.   By adding a finite number of fundamental domains to $\FF'$, (whose fixed rays must land alone at repelling fixed points),   we obtain a new finite collection $\FF \supset \FF'$ which is complete and full (informally, we first add all fundamental domains which intersect $D$ and are not already in $\FF'$, {then add those which are "temporarily" visited by some fixed ray} and then "fill in" the gaps in between non-consecutive domains). {Since all fixed rays land, we have added only a finite number of new domains.} Set $N=\#\FF$, and let $\GG \supset G$ be the collection of $N \geq L$ rays associated to $\FF$ . By Theorem \ref{Global counting theorem} there exist $N+1$ fixed points counted by multiplicity which are either landing points of some $g\in \GG$, or interior fixed points. Now note that  the $N-L$ rays in $\GG \setminus G$ are necessarily
fixed rays which land alone at repelling (hence simple) fixed points. Therefore there are exactly $N-L$ simple fixed points which are landing points of these $N-L$ rays. Thus the remaining $L+1$ fixed points (counted with multiplicity) are either landing points of a ray $g\in G$ or interior fixed points. 
\end{proof}

Before proving Theorem \ref{Global counting theorem} we would like to emphasize the following simple but rather surprising corollary.

\begin{cor} \label{onlyoneregion}
Let $f\in \BBt$ and assume that all fixed rays land.  If there is only one basic region (with respect to the set of fixed rays), then there is  exactly one interior fixed point or virtual fixed point.
\end{cor}

\begin{proof}
If there is only one basic region then there are no fixed ray pairs, that is, all fixed rays land alone either at repelling fixed points or parabolic ones. There are only a finite number of the latter by Proposition A. Let $G$ be {\em any} collection of $N$ rays which includes them. By Theorem $B$ there are $N+1$ fixed points counted with multiplicity which are either landing points of some $g\in G$ or interior fixed points. But all rays land alone; so, either all the landing fixed points are repelling and there is exactly one fixed point left which must be  
 an interior and non-multiple  fixed point, or one (and only one) of the landing points is parabolic of multiplicity two and there is one (and only one) virtual fixed point. 
\end{proof}
This means for example, that even if $f$ may have infinitely many singular values, if all fixed rays land alone, there is at most one non-repelling fixed point. 

\noindent The remaining of the section is dedicated to prove Theorem \ref{Global counting theorem}. 

\begin{proof}[Proof of Theorem \ref{Global counting theorem}]
Let $\FF$, $\GG$ and $N$ be as in the statement. Note that, since $\FF$ is complete,  rays in $\GG$ cannot intersect any fundamental domain which is not in $\FF$. Since the collections are finite, their members intersect a finite number of tracts say, $\{\Ta\}_{1\leq \alpha \leq M}$, where the tracts are labeled consecutively respecting the cyclic order. From now on  the index $\alpha$ will be taken modulo $M$.  Let $\FFa:=\FF\cap \Ta$, and $N_\alpha:=\#\FFa$. 

\begin{claim} \label{wearesaved}
Without loss of generality we may assume that 
\begin{itemize}
\item[(a)] $\delta$ does not intersect  the set of fixed  rays and  
\item[(b)]  $\Ta$ intersects $D$ for every $1\leq \alpha \leq M$.
\end{itemize}
\end{claim}
\begin{proof}
{
In fact we will show that $(a)$ and $(b)$ are satisfied up to considering {another}  full and complete finite family $\FF'$ of fundamental domains, which is obtained by {first erasing some compact parts of the original fundamental domains and then} adding to $\FF$ finitely many extra ones (whose fixed ray must land alone at a repelling fixed point). Proving Theorem~\ref{Global counting theorem} for $\FF'$ then implies Theorem~\ref{Global counting theorem} for $\FF$.} {This new family will be constructed by enlarging $D$ and shortening $\delta$}.

Observe that rays are intrinsic to the function $f$, i.e., they do not depend on the definitions of $D$ or $\delta$. Tracts and fundamental domains  do. {By Lemma~\ref{Forced landing} there are only finitely many rays which are not fully contained in their fundamental domains; since any such ray is still asymptotically contained in a fundamental domain by Lemma \ref{asymptotically contained}}, there exists a sufficiently large disk $D'\supset D$ such that the set of fixed rays is contained in $D'\cup \TT$, where $\TT$ are the tracts defined by $D$. 
Also, for any tract   $\Ta$   which contains at least one fundamental domain  {$F_\alpha\in\FF$,  let $z_\alpha$ be the landing point of the fixed ray associated to $F_\alpha$; up to  enlarging $D'$, we can assume that it contains all such  $z_\alpha$.}
 Let $\delta'$ be the subset of $\delta$ in $\C\setminus D'$ connecting $D'$ with infinity. By construction $\delta'$ does not  intersect any  fixed ray. The tracts defined as preimages of $\C\setminus D'$ are smaller than the previous ones,  however, the  set of  fixed rays does not change. Also, the new fundamental domains are {the same} as the old fundamental domains {outside a compact set}, since $\delta'=\delta$ outside  a compact set. By Proposition \ref{Forced landing}, and for all $\alpha=1\ldots M$,  the new  fundamental domains $F_\alpha'$ which are contained in $F_\alpha$ all intersect $D'$,  since their fixed ray lands in $D'$.
 
The enlarged disk $D'$  intersects possibly more fundamental domains than the ones which were originally in $\FF$, so $\FF$ is no longer a  complete collection with respect to the new disk. However it can be turned into a full and complete collection $\FF'$ (which satisfies property $(b)$ by construction), by adding finitely many fundamental domains, whose fixed rays all land alone at repelling fixed  points since $\FF$ was originally complete.
\end{proof}
 
Let $C_R$ denote  the circle of radius $R$ and $D_R$ the closed disk bounded by $C_R$.  Choose $R$  large enough so that   
\begin{itemize}
\item  $D\subset D_R$;
\item $f^{-1}(C_R) \cap \FF$ is contained in $D_R$ (see part $(b)$ of  the Structural Lemma~\ref{Technical}), and 
\item all fixed rays associated to fundamental domains in $\FF$ land inside the disk $D_R$.  
\end{itemize}

Let $P$ be the first intersection of $\delta$ with  $C_R$ when moving from $D$ towards infinity. Call $\delta_P$ the arc in $\delta$ connecting $P$ to $D$. Let $\dapm$ be the two  preimages of $\delta_P$ which are contained in $\Ta$, so that each  $\dapm$ is on the boundary between a fundamental domain in $\FFa$ and a fundamental domain which is not in $\FFa$ (these curves are well defined because $\FF$ is full). Observe that by construction the curves $\dapm$ are in the complement of $D$. Label the $\dapm$ so as to respect the cyclic order in the sense that $\dam$ comes before $\dap$ and after $\damp$. 
Each $\dapm$ intersects the preimage of $C_R$ in a point 
$P^\pm_\alpha$ which is a preimage of $P$ (and hence it is inside $D_R$), and ends at the boundary of $\Ta$.  Up to increasing $R$, we may assume that these endpoints are also in $D_R$.
Let $r_\alpha$ be the arc in the preimage of $C_R$ connecting $\Pam$ and $\Pap$ inside $\Ta$ and such that $f(r_\alpha)$ covers $C_R$  $N_\alpha$ times.  As usual, give  $C_R$ and any other Jordan curve in this section a counterclockwise orientation; this induces an orientation on the arcs $r_\alpha$ (see Figure~\ref{G} for an illustration of a possible setup).

Finally, we claim that there exists $\gamma_\alpha$ Jordan arcs in $\C\setminus (\ov{\TT\cup D})$ connecting   $\damp$ with $\dam$, which do not intersect the set of fixed rays.  Indeed, $\TT \cup D  \cup \GG$ divides the plane into finitely many components. Some of them may be bounded (the tracts in $\Ta$ can intersect $\partial D$ several times -- the others do not intersect $D$ at all -- {and the rays could connect two different tracts through the complement of $D\cup \TT$}), but  there are exactly $M$  unbounded components, since the boundaries of the tracts consist of single unbounded curves, there are only $M$ tracts intersecting $D$ {and the relevant pieces of rays form a compact set.}  Let us denote by $E_\alpha$ the unbounded component in between $T_{\alpha-1}$ and $T_\alpha$ with respect to the cyclic order. By definition of $\damp$, being the boundary of the ``last'' fundamental domain of $T_{\alpha-1}$, its unique point of intersection with $\partial T_{\alpha-1}$ belongs to the boundary of $E_\alpha$. This is also the case for  $\dam$, so both endpoints are on the boundary of  the same open, connected and simply connected set $E_\alpha$ and can therefore be joined by a curve in its interior. Since only finitely many fixed rays intersect $\C\setminus \ov{\TT}$ by Proposition \ref{Forced landing}, the arc $\gamma_\alpha$ can be chosen to avoid them.

In the notation above, the curve $\delta$ belongs to a unique unbounded complementary component $E_\alpha$, since $\delta$ is connected and unbounded.  So  one and exactly one of the curves $\gamma_\alpha$ intersects $\delta$.
Choose that specific $\gamma_\alpha$ so as to intersect $\delta $ only once, and to do so in the set $\{|z|>R\}$. 
  Note that the image of each $\gamma_\alpha$ is fully contained in $D$.  

Call $\Gamma_\alpha$ the arcs  $\damp\cup\gamma_\alpha\cup\dam$ connecting $\Pamp$ to $\Pam$.
We will count the number of fixed points of $f$ inside the region  $\widehat{U}$ enclosed by the Jordan curve 
\begin{displaymath} \Gamma=\underset{\alpha\leq M}\bigcup (\Gamma_\alpha\cup r_\alpha).\end{displaymath}

Observe that, for $R$ large enough, all fixed points which are landing points of the fixed rays associated to the fundamental domains in ${\FF}_\alpha$ have to be contained in $\widehat{U}$. Also, because there are no fixed points in $\C\setminus(\ov{D\cup\TT})$, the  final count does not depend on the choice of $\gamma_\alpha$. We will also see that the result does not depend on $C_R$ (and hence on the $r_\alpha$) as long as $C_R$ satisfies the Structural Lemma~\ref{Technical} {and the condition above} (see Figure~\ref{G}).

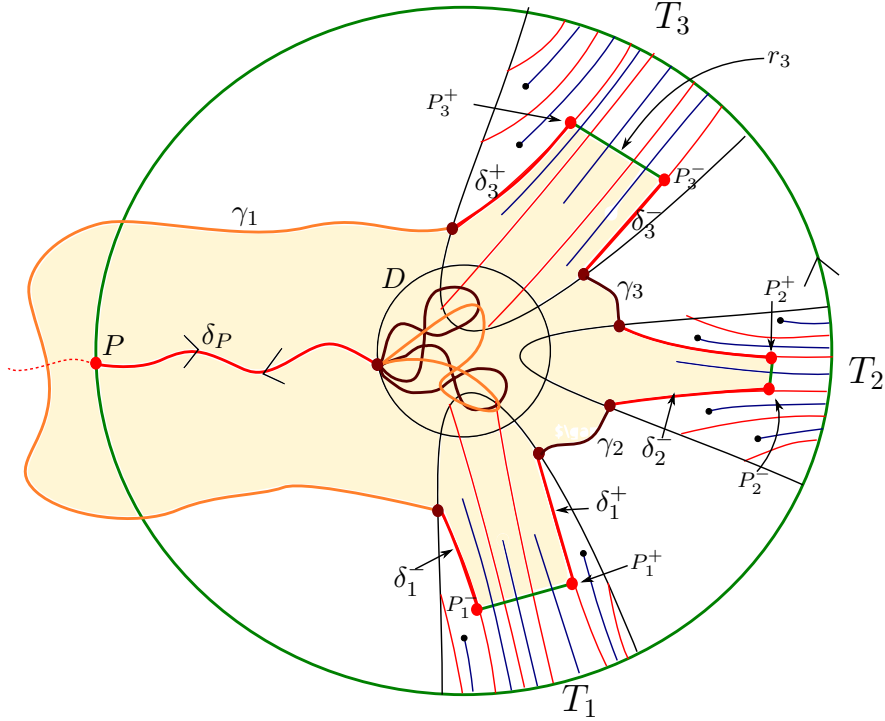
\begin{figure}[hbt!]
\begin{center}
\def\svgwidth{0.8\textwidth}
\begingroup%
  \makeatletter%
  \providecommand\color[2][]{%
    \errmessage{(Inkscape) Color is used for the text in Inkscape, but the package 'color.sty' is not loaded}%
    \renewcommand\color[2][]{}%
  }%
  \providecommand\transparent[1]{%
    \errmessage{(Inkscape) Transparency is used (non-zero) for the text in Inkscape, but the package 'transparent.sty' is not loaded}%
    \renewcommand\transparent[1]{}%
  }%
  \providecommand\rotatebox[2]{#2}%
  \ifx\svgwidth\undefined%
    \setlength{\unitlength}{522.46962891bp}%
    \ifx\svgscale\undefined%
      \relax%
    \else%
      \setlength{\unitlength}{\unitlength * \real{\svgscale}}%
    \fi%
  \else%
    \setlength{\unitlength}{\svgwidth}%
  \fi%
  \global\let\svgwidth\undefined%
  \global\let\svgscale\undefined%
  \makeatother%
  \begin{picture}(1,0.75807552)%
    \put(0,0){\includegraphics[width=\unitlength]{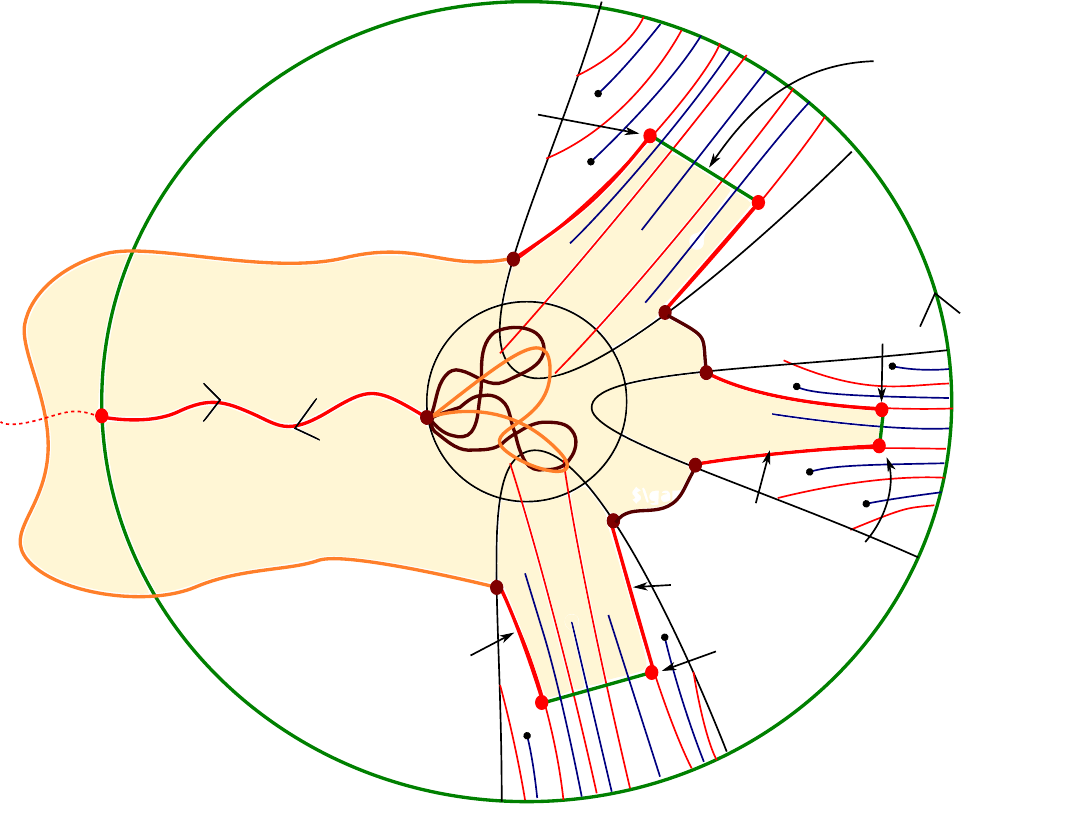}}%
    \put(0.64646345,0.44985639){\color[rgb]{0,0,0}\makebox(0,0)[lb]{\smash{$\gamma_3$}}}%
    \put(0.09961681,0.38449267){\color[rgb]{0,0,0}\makebox(0,0)[lb]{\smash{$P$}}}%
    \put(0.20540572,0.39676265){\color[rgb]{0,0,0}\makebox(0,0)[lb]{\smash{$\delta_P$}}}%
    \put(0.23783312,0.52876429){\color[rgb]{0,0,0}\makebox(0,0)[lb]{\smash{$\gamma_1$}}}%
    \put(0.58790489,0.00397227){\color[rgb]{0,0,0}\makebox(0,0)[lb]{\smash{\Large $T_1$}}}%
    \put(0.46467463,0.11074021){\color[rgb]{0,0,0}\makebox(0,0)[lb]{\smash{\ss $P_1^-$}}}%
    \put(0.66055415,0.1558304){\color[rgb]{0,0,0}\makebox(0,0)[lb]{\smash{\ss $P_1^+$}}}%
    \put(0.40934513,0.13651679){\color[rgb]{0,0,0}\makebox(0,0)[lb]{\smash{$\delta_1^-$}}}%
    \put(0.62391831,0.21782946){\color[rgb]{0,0,0}\makebox(0,0)[lb]{\smash{$\delta_1^+$}}}%
    \put(0.62549651,0.28236493){\color[rgb]{0,0,0}\makebox(0,0)[lb]{\smash{$\gamma_2$}}}%
    \put(0.67237153,0.27554504){\color[rgb]{0,0,0}\makebox(0,0)[lb]{\smash{$\delta_2^-$}}}%
    \put(0.77421913,0.24319278){\color[rgb]{0,0,0}\makebox(0,0)[lb]{\smash{\ss $P_2^-$}}}%
    \put(0.80146115,0.44891697){\color[rgb]{0,0,0}\makebox(0,0)[lb]{\smash{\ss $P_2^+$}}}%
    \put(0.66149353,0.51843114){\color[rgb]{0,0,0}\makebox(0,0)[lb]{\smash{$\delta_3^-$}}}%
    \put(0.70564438,0.56915763){\color[rgb]{0,0,0}\makebox(0,0)[lb]{\smash{\ss $P_3^-$}}}%
    \put(0.49522326,0.56164262){\color[rgb]{0,0,0}\makebox(0,0)[lb]{\smash{$\delta_3^+$}}}%
    \put(0.44355734,0.64806557){\color[rgb]{0,0,0}\makebox(0,0)[lb]{\smash{\ss $P_3^+$}}}%
    \put(0.39630912,0.456486){\color[rgb]{0,0,0}\makebox(0,0)[lb]{\smash{$D$}}}%
    \put(0.89205787,0.35447681){\color[rgb]{0,0,0}\makebox(0,0)[lb]{\smash{\Large $T_2$}}}%
    \put(0.6860541,0.73391457){\color[rgb]{0,0,0}\makebox(0,0)[lb]{\smash{\Large $T_3$}}}%
    \put(0.80627346,0.69740714){\color[rgb]{0,0,0}\makebox(0,0)[lb]{\smash{\small $r_3$}}}%
  \end{picture}%
\endgroup%
\end{center}
\caption{\small Sketch of a setup for the global counting. The curve $\Gamma$ is the boundary of the shaded region $\hat{U}$. Its image consists of the big circle $C_R$, together with $\delta_P$ and the images of the curves $\gamma_i$ which are inside $D$. In this example the Argument Principle gives 8 fixed points inside the shaded region, counted with multiplicity. Since there are only 7 fixed rays landing in $\hat{U}$, at least one of the fixed points must be an interior or virtual fixed point.}
\label{G}
\end{figure}

We calculate $\ind(f(\Gamma)-\Gamma, 0)$, by calculating it piece by piece on the $r_\alpha$ and on the $\Gamma_\alpha$.
Observe that by construction there are no fixed points on $\Gamma$.
Let us first calculate $\ind(f(r_\alpha)-r_\alpha,0)$. 
Observe that $f(r_\alpha)$ is the circle $C_R$ covered $N_\alpha$ times, starting and ending at $P$.
Then by the Second homotopy Lemma \ref{Third homotopy lemma},  part (a),

$$\ind(f(r_\alpha)-r_\alpha,0)=N_\alpha+ \ind(r_\alpha,P).$$

 Now let us calculate $\ind(f(\Gamma_\alpha)- \Gamma_\alpha,0)$. 
 The image of each $\Gamma_\alpha$ is a curve starting at $P$, moving along $\delta_P$, staying inside $\ov{D}$ and then moving back to $P$ along $\delta_P$. 

 Observe that $\Gamma_\alpha \cap f(\Gamma_\alpha) =\emptyset$. Then, by the Second Homotopy Lemma   \ref{Third homotopy lemma}, part (a), with $N=0$ we have
\[
\ind(f(\Gamma_\alpha)- \Gamma_\alpha,0)=\ind(\Gamma_\alpha,P).
\]
 
Summing up,  we obtain 
\[
\ind(f(\Gamma)-\Gamma, 0)=\underset{\alpha}\sum \Big(\ind(r_\alpha, P)+ N_\alpha +\ind(\Gamma_\alpha, P)\Big)=N+ \ind(\Gamma, P)= N+1.
\]

The last equality follows from the fact that $\Gamma$ is a Jordan curve (oriented counterclockwise) and $P$ is contained in the bounded connected component of $\C \setminus \Gamma$. Hence, there are exactly $N+1$ fixed points in $\widehat{U}$ counted with multiplicity.
\end{proof}

\section{Local Counting. Proof of the Main Theorem}\label{Local Counting}

If there exists only one basic region, the Main Theorem  reduces to  Corollary \ref{onlyoneregion}. So from now on we assume that there are at least two basic regions. 

Let us  consider a basic region $V$ and show that it contains exactly one interior fixed point or,  instead,  exactly one virtual fixed point. For the purpose of the proof in this section, we redefine   $V$  to contain also the fixed rays landing alone (at repelling or parabolic fixed points) that  were part of $\partial V$. In other words, $\partial V$ now consists exclusively of ray pairs.

Because of the presence of fixed rays and fixed points in $\partial V$, the counting is more involved than it was  in the last section. Notice however that the only fixed points in $\partial V$ are landing points of fixed rays, which by the Snail lemma \cite[Lemma 16.2]{Mi}  must be repelling or parabolic.  

The strategy of the proof is as follows. We first consider the case where $V$ contains no virtual fixed point. This implies that every boundary fixed point is repelling or, if it is parabolic,  both boundary rays land at it  from the same repelling petal (see the precise  definition below). We then define a new region $\Vhat$,  a bounded modification of $V$,  which includes all the repelling fixed points which were previously on $\partial V$ and excludes all the parabolic ones.
After defining $\Gamma_{\Vhat}:=\partial \Vhat$, now a fixed point free Jordan curve,  we will apply the argument principle to $f(\Gamma_{\Vhat}) - \Gamma_{\Vhat}$ to count the number of fixed points in $\Vhat$. We shall show that $\Vhat$ contains as many fixed points as the number of rays in $\Vhat$ landing alone, plus the number of repelling fixed points which where previously  in $\partial V$, plus one extra fixed point which must therefore be interior. Note that this extra fixed point cannot be attributed to a fixed point of multiplicity two since we assumed that $V$ had no virtual fixed points to start with  and we left out all parabolic points which were previously on $\partial V$.

To that end,  we shall compute $\ind(f(\Gamma_{\Vhat}) - \Gamma_{\Vhat},0)$  piecewise, as in the section above, using the Homotopy Lemmas in Section \ref{Index}, but also doing some homotopies {\em in situ}. It is worth noticing that homotopies are used exclusively to compute the indexes of certain arcs. They do not further modify the region $\Vhat$ or the map $f$. 
 
Finally we deal with the remaining regions which contain virtual fixed points. This will be done by a global counting argument.

\subsection*{Modification of $V$ near the fixed points on the boundary}\label{Fixed points on the boundary}
By definition basic regions are open. Suppose $z_0\in \partial V$ is a repelling or parabolic fixed point. We slightly modify the boundary of $V$ near $z_0$ making use of the local dynamics around $z_0$.  Here we describe the modification near every fixed point on $\partial V$; in the proof of the Main Theorem we will apply all the modifications together. In these lemmas $V'$ denotes the modification of $V$.

\begin{lem}[\bf Repelling fixed points in $\partial V$]\label{Enlarging rfp} Let $z_0$ be a repelling fixed point on  $\partial V$. Then, a small arc on $\partial V$ containing $z_0$ can be replaced by an  arc $\zeta$ with the same endpoints such that: 
\begin{itemize}
\item the resulting region  $V' \supset V$ satisfies  ${\rm Fixed}(f) \cap V' =\left( {\rm Fixed}(f) \cap V\right) \cup \{z_0\}$; 
\item $f$ has no fixed points on $\zeta$;
\item $f(\zeta)$ is contained in  $\C\setminus V'$.
\end{itemize}
\end{lem}

\begin{proof}
Let $D_\epsilon$ be a small disk centered at $z_0$, with $\epsilon$ small enough so that $f(D_\epsilon)$ is a topological disk containing $\overline{D_\epsilon}$. See Figure \ref{modificationrfp}. We define $V'$ as the ``filled'' region $V\cup D_\epsilon$, that is, the region $V\cup D_\epsilon$ union the (finite number of) bounded connected components of the complement of {$\overline{V\cup D_\epsilon}$}  (which are due to the possible wiggling of the rays in and out $D_\epsilon$).  By choosing $\epsilon$ small enough, we can make sure that we added no extra fixed points to $V'$ other than $z_0$. 
 We then define $\zeta$ as the arc in $\partial V'$ starting an ending at $\partial D_\epsilon$ such that $\partial V' \setminus \zeta \subset \partial V$. More precisely, $\zeta$ joins the two points $x_1$ and $x_2$ in $\partial D_\epsilon$ which correspond to the first intersection (when coming from infinity) of the fixed rays in $\partial V$ landing at $z_0$, with $\partial D_\epsilon$. Observe that if these rays do not wiggle outside of $D_\epsilon$ after entering it for the first time, then $\zeta \subset \partial D_\epsilon$. Otherwise, $\zeta$ also contains parts of the rays. 

By construction  $f(\zeta)$ has no fixed points and is contained in $\C\setminus{V'}$. Indeed, {$V'$ can be characterized as the union of all bounded domains whose boundary is contained in  $\overline{V\cup D_\epsilon}$. Locally, the preimage of such domain has the same property. Thus locally $V'$ is backward invariant.}
\end{proof}

We now consider the case of parabolic fixed points.  For the following description in more detail see for example \cite[\textsection 10]{Mi}.  Let  $z_0$ be a parabolic fixed  point of multiplicity $m+1$. On any {simply connected } neighborhood of $z_0$ disjoint from the set of singular values, there is a well defined branch $\phi$ of $f^{-1}$ which fixes $z_0$.  In this setup, there exist $m$ {\em attracting vectors}   such that if an orbit converges to $z_0$ it does so in a direction asymptotically tangent to one of them.  Likewise, there exist $m$ {\em repelling vectors}, defined as attracting vectors for $\phi$.  Attracting and repelling vectors alternate.  An \emph{attracting petal} for $z_0$  is defined as any open simply connected set $W$ containing $z_0$ on its boundary, such that $f(\overline{W})\subset W \cup \{z_0\}$. Likewise, a {\em repelling petal} is an attracting petal for $\phi$. Every attracting (resp.~repelling) petal is associated to an attracting (resp.~repelling) vector. Every virtual fixed point (or immediate attracting basin) attached to $z_0$ is associated to a unique attracting vector and contains all attracting petals also associated to the same vector.

\begin{lem}[\bf Parabolic fixed points in $\partial V$]\label{Enlarging pfp} 
Let $z_0$ be a parabolic fixed point on  $\partial V$, such that no virtual fixed point associated to $z_0$ is contained in $V$. Then, a small arc on $\partial V$ containing $z_0$ can be replaced by an  arc $\zeta$ with the same endpoints such that: 
\begin{itemize}
\item the resulting new region $V' \subset V$  satisfies ${\rm Fixed}(f) \cap V' ={\rm Fixed}(f) \cap V$;
\item $f$ has no fixed points on $\zeta$;
\item $f(\zeta)$ is contained in  $\ov{V'}$.
\end{itemize}
\end{lem}
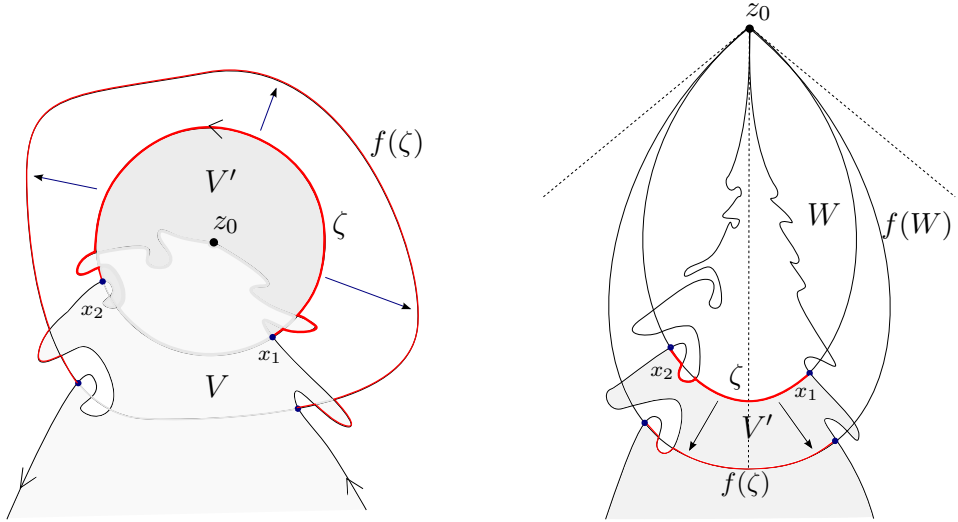
\begin{figure}[htb!]
\centering
\begin{minipage}[b]{0.35\textwidth}
\includegraphics[width= \textwidth]{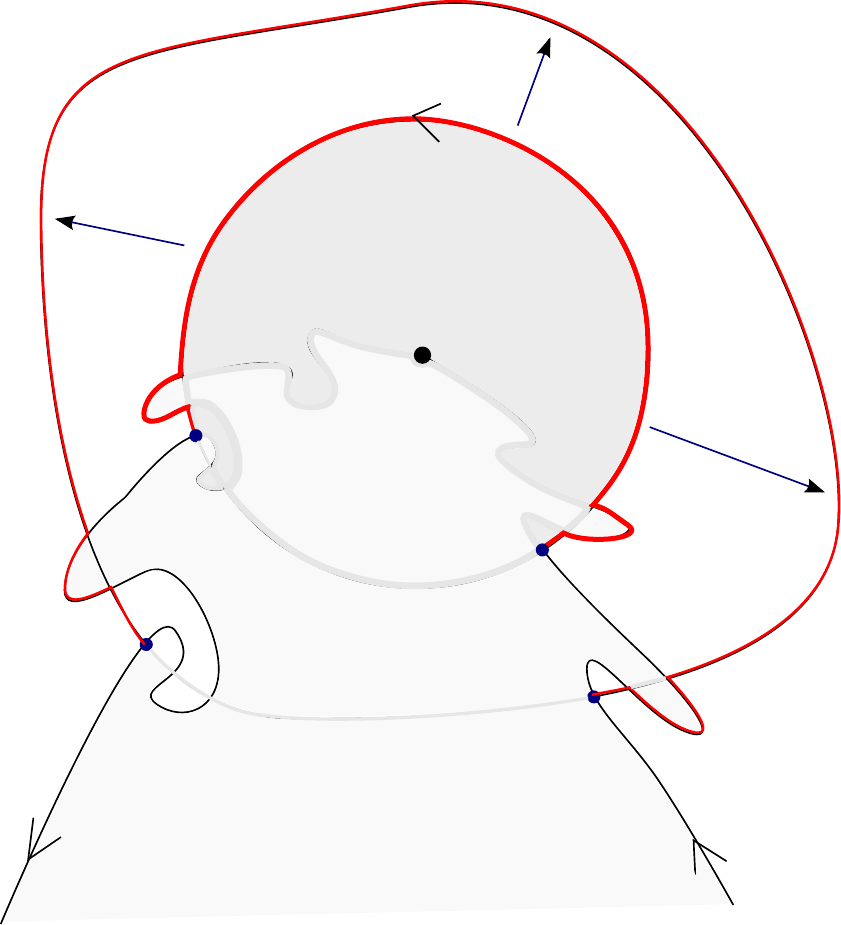}
\setlength{\unitlength}{\textwidth}
\put(-0.12,0.9){$f(\zeta)$}
\put(-0.21,0.7){\large $\zeta$}
\put(-0.5,0.71){$z_0$}
\put(-0.39,0.39){\scriptsize $x_1$}
\put(-0.82,0.5){\scriptsize $x_2$}
\put(-0.52,0.3){\large$V$}
\put(-0.52,0.8){\large$V'$}
\end{minipage}
\hfil
\begin{minipage}[b]{0.35\textwidth}
\includegraphics[width= \textwidth]{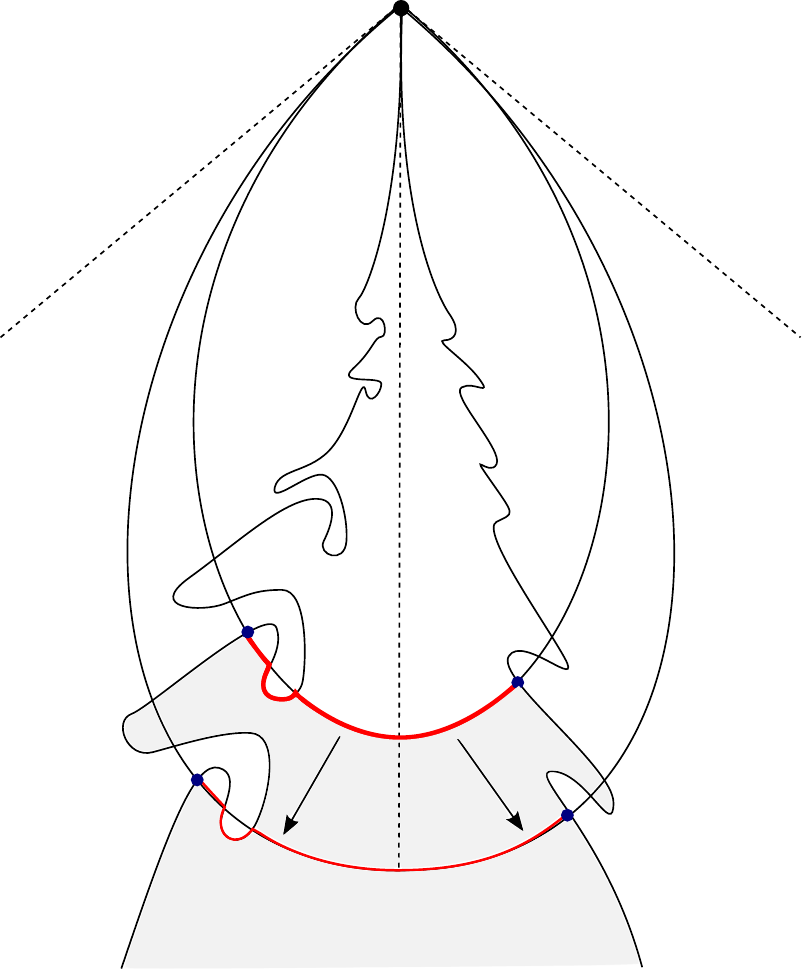}
\setlength{\unitlength}{\textwidth}
\put(-0.18,0.7){$f(W)$}
\put(-0.36,0.72){\large $W$}
\put(-0.51,1.23){$z_0$}
\put(-0.39,0.3){\scriptsize $x_1$}
\put(-0.74,0.35){\scriptsize $x_2$}
\put(-0.52,0.2){\large$V'$}
\put(-0.55,0.33){$\zeta$}
\put(-0.57,0.07){\small $f(\zeta)$}
\end{minipage}
\caption{\small Left: The enlargement of $V$ near a repelling fixed point on its boundary.  The region $V'$ includes the lighter shaded region $V$ as well as the darker shaded parts. Right: The shrinking of $V$ near a parabolic fixed point in its boundary.}
\label{modificationrfp}
\end{figure}

\begin{proof}
Since there is no virtual fixed point associated to $z_0$ in $V$, there exists an  arbitrarily  small repelling petal $W$ attached to $z_0$ which eventually contains both rays landing at $z_0$ on $\partial V$. See Figure \ref{modificationrfp}. Consider the region $V'$ to be the unbounded connected component of $V\setminus \ov{W}$.  We define $\zeta \subset \partial V'$ as the curve starting and ending in $\partial W$, not going through $z_0$, such that $\partial V'\setminus \zeta\subset \partial V$. Equivalently $\zeta$ joins the first intersection points between the rays and  $\partial W$. If the rays do not wiggle outside of $W$ once they enter it,  then $\zeta\subset \partial W$, but otherwise $\zeta$ contains part of the rays. By construction, the new region $V'$  does not contain $z_0$ and $f$ has no fixed points on $\zeta$. Finally,  $f(\zeta)$ is contained in  $\overline{V'}$   by similar arguments to those used in Lemma \ref{Enlarging rfp}.  Again, by shrinking $W$ if necessary, we can make sure we did not remove any fixed point from $V$ by cutting away the bounded components of $V' \setminus W$. 
 \end{proof}

\subsection*{Proof of the Main Theorem}
The setup for the proof of the Main Theorem is very similar to the setup in Section  \ref{Global counting}, so we will use the notation introduced there.   Let $\FF$ be a finite full and complete collection of fundamental domains (see Definition~\ref{Full and complete collections}). Let $P, \Gamma_\alpha, r_\alpha$  be as in the proof of Theorem \ref{Global counting theorem}, and $\hat{U}$ be the the region bounded by $\bigcup(r_\alpha\cup \Gamma_\alpha)$. 
Let $V$ be a basic region for $f$. For simplicity of exposition, we assume that  rays in $\partial V$  intersect the circle $C_R$ in a single point. At the very end of the proof we will indicate how to slightly modify the construction to deal with the case of multiple intersections.

Assume first that there are no virtual fixed points in $V$.   Let  $\Vhat=V'\cap\widehat{U}$, where $V'$ is $V$ after being modified  near the boundary fixed points as in Lemmas \ref{Enlarging rfp} and \ref{Enlarging pfp}. Recall that when modifying the region,  repelling fixed points on $\partial V$ end up inside $V'$, while parabolic fixed points do not.  Without loss of generality,  $R$ is large enough so that all fixed rays associated to the fundamental domains in  $\mathcal{F}\cap V$ have their landing points in $\Vhat$. Let $\GV$ be the boundary of $\Vhat$.
 
The difference with respect to the global counting in Section \ref{Global counting} is that we additionally  have to calculate the contributions to $\ind (f(\GV)-\GV,0)$ made by the (pieces of) fixed ray pairs in $\GV$. 
Observe that, {if $\delta \cap V =\emptyset$},  there is one and exactly one fixed ray pair on $\GV$ separating $\delta$ from $\Vhat$.
Let $\{P_i\}_{i=1}^k$ be the collection of $k$ preimages of $P$ which are contained in $\GV$, labeled so as to respect the order in which they are encountered when moving along $\GV$ counterclockwise. Also, let $s_i$ be the arc along $\GV$ joining two consecutive points $P_i$ and $P_{i+1}$, where indices are taken mod $k$ (see Figure~\ref{J}).

The arcs $s_i$ can be of {five}  types, {two of which are mutually exclusive}:

\begin{enumerate}
\item[(0)] $s_i$ intersects the ray pair separating $\delta$ from $\Vhat$ {(only present if $\delta\cap V=\emptyset$)};
\item[(0')] $s_i$ is the unique $\Gamma_\alpha$ intersecting $\delta$ (only present if $\delta\cap V\neq \emptyset$);
\item[(1)]  $s_i$ intersects any other  ray pair; 
\item[(2)]  $s_i\subset r_\alpha$  for some $r_\alpha$;
\item[(3)]  $s_i= \Gamma_\alpha$ for some $\alpha$ such that $\Gamma_\alpha\cap\delta=\emptyset$.
\end{enumerate} 

\begin{figure}[hbt!]
\begin{center}
\def\svgwidth{0.5\textwidth}
\begingroup%
  \makeatletter%
  \providecommand\color[2][]{%
    \errmessage{(Inkscape) Color is used for the text in Inkscape, but the package 'color.sty' is not loaded}%
    \renewcommand\color[2][]{}%
  }%
  \providecommand\transparent[1]{%
    \errmessage{(Inkscape) Transparency is used (non-zero) for the text in Inkscape, but the package 'transparent.sty' is not loaded}%
    \renewcommand\transparent[1]{}%
  }%
  \providecommand\rotatebox[2]{#2}%
  \ifx\svgwidth\undefined%
    \setlength{\unitlength}{412bp}%
    \ifx\svgscale\undefined%
      \relax%
    \else%
      \setlength{\unitlength}{\unitlength * \real{\svgscale}}%
    \fi%
  \else%
    \setlength{\unitlength}{\svgwidth}%
  \fi%
  \global\let\svgwidth\undefined%
  \global\let\svgscale\undefined%
  \makeatother%
  \begin{picture}(1,0.93592233)%
    \put(0,0){\includegraphics[width=\unitlength]{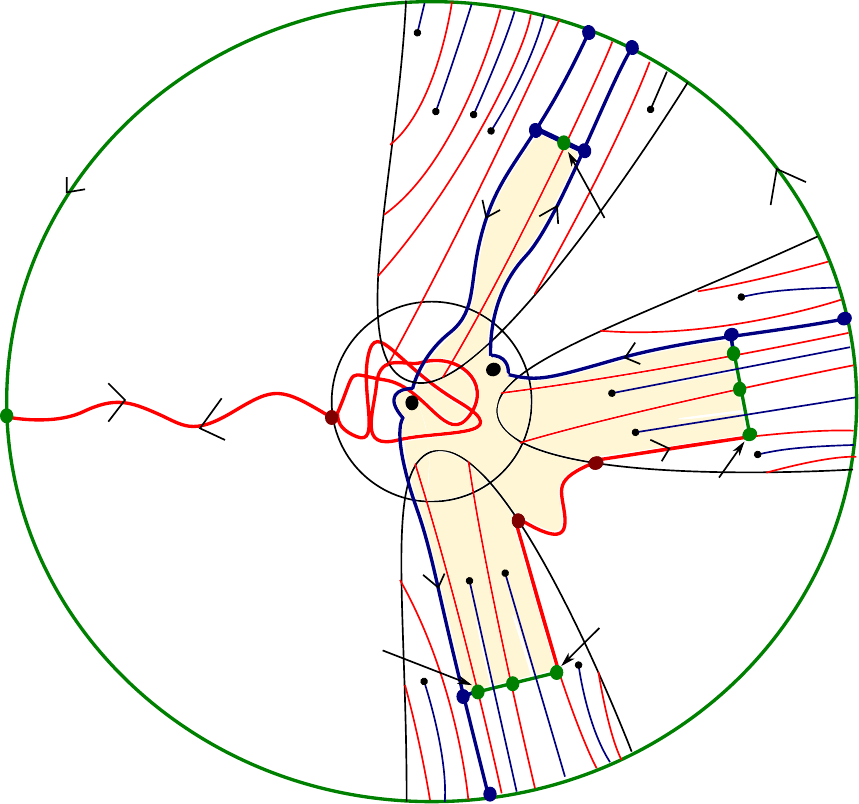}}%
    \put(0.41243671,0.19244847){\color[rgb]{0,0,0}\makebox(0,0)[lb]{\smash{\ss $P_1$}}}%
    \put(0.70072069,0.19959598){\color[rgb]{0,0,0}\makebox(0,0)[lb]{\smash{\ss $P_3$}}}%
    \put(0.82222883,0.35565055){\color[rgb]{0,0,0}\makebox(0,0)[lb]{\smash{\ss $P_4$}}}%
    \put(0.8770266,0.47477624){\color[rgb]{0,0,0}\makebox(0,0)[lb]{\smash{\ss $P_5$}}}%
    \put(0.70191194,0.66299468){\color[rgb]{0,0,0}\makebox(0,0)[lb]{\smash{\ss $P_7$}}}%
    \put(0.24804336,0.90124595){\color[rgb]{0,0,0}\makebox(0,0)[lb]{\smash{$C_R$}}}%
    \put(0.17299422,0.47596744){\color[rgb]{0,0,0}\makebox(0,0)[lb]{\smash{$\delta_P$}}}%
    \put(0.01932217,0.46643742){\color[rgb]{0,0,0}\makebox(0,0)[lb]{\smash{$P$}}}%
    \put(0.60184643,0.36875439){\color[rgb]{0,0,0}\makebox(0,0)[lb]{\smash{$\widehat{V}$}}}%
  \end{picture}%
\endgroup%
\end{center}
\caption{\small A simple example of counting for the modified basic region $\Vhat$ shown (shaded) in the picture. In this example,  $\delta\cap V=\emptyset$}. The original region $V$ had two repelling fixed points on its boundary which are now part of $\Vhat$.   Arcs of type $(0)$ and $(1)$ are colored in blue; arcs of type $(2)$ are colored in green; arcs of type $(3)$ are in red. Each arc has the same color as its image. In this example $\ind(f(\Gamma_{\Vhat})- \Gamma_{\Vhat})=7$.
\label{J}
\end{figure}
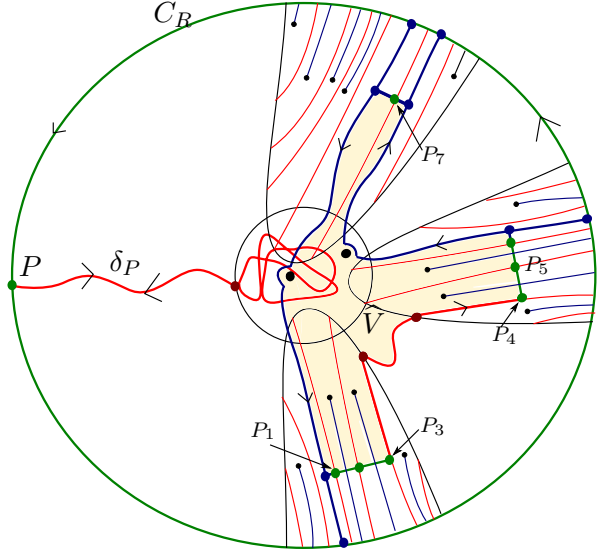

We only have to compute  $\ind(f(s_i)-s_i,0)$ when $s_i$ is of type $(0)$ and $(1)$, as the other three cases are already covered in the previous section. We calculate this in two separate claims.  
For each $s_i$ of type $(0)$ and $(1)$, let $n_i=1$ if the ray pair originally landed at a repelling fixed point, and $n_i=0$ if the ray pair originally landed at a parabolic fixed point.
\begin{claim}
If $s_i$ is of type  $(0)$, the  counting  gives  $\ind(f(s_i)-s_i,0)=1+n_i+\ind(s_i,P)$.
\end{claim}
\begin{proof}
Let $s_i$ be the arc of type $(0)$, $g_1$ and $g_2$ be  respectively the first and the second ray in $s_i$ that we encounter when moving  along $\Gamma_{\Vhat}$ counterclockwise. Let $\zeta$ be as in Lemma \ref{Enlarging rfp} or \ref{Enlarging pfp} depending on whether the original fixed point  was repelling ($n_i=1$) or parabolic ($n_i=0$). Then $f(s_i)$ starts at $P$, goes along $C_R$ counterclockwise  until meeting $g_1$, continues on $g_1$ until hitting $\zeta$, then along $f(\zeta)$, then moves back along the ray $g_2$ and continues along $C_R$ until it meets $P$ again.  See Figure \ref{indexcase0}.Ê Note that  $f(s_i)$  divides the plane into exactly three connected components. Denote by $U_1$ the bounded component with $P$ on the boundary and by $U_2$ the bounded component which contains the points $P_i$ and $P_{i+1}$.  Observe that $\ind(f(s_i),z)=1$ if $z\in U_1$ and equals 2 if $z\in U_2$. 

Let us first suppose that $n_i=1$. Then by Lemma~\ref{Enlarging rfp} $f(\zeta)$, and hence $f(s_i)$, is contained in $\C\setminus \Vhat$, which is closed. Hence  $s_i$, and in fact the whole $\Vhat$,  is contained in $\overline{U_2}$, so that $\ind(f(s_i),z)=2$ for all $z\in s_i\setminus f(s_i)$. By the Second Homotopy Lemma~\ref{Third homotopy lemma} part (b), with $\gamma=s_i$ and $\sigma=f(s_i)$ we obtain that $\ind(f(s_i)-s_i,0)=2+ \ind(s_i,P)$. 

Suppose now that $n_i=0$.  Then, by Lemma~\ref{Enlarging pfp}, $f(\zeta)$ is contained in the closure of $\Vhat$. This is a problem when trying to apply  Lemma~\ref{Third homotopy lemma} since there is a part of $s_i$ contained in $\overline{U_1}$, namely the rays and the arc $\zeta$, while the remaining part  belongs to $U_2$, precisely two arcs say $\xi_1$ going from the endpoint $P_i$ to the ray $g_1$, and $\xi_2$  going from the endpoint $P_{i+1}$ to the ray $g_2$. See Figure \ref{indexcase0}.  Let $\gamma=s_i$ parametrized by $[0,1]$ and $f(s_i)$ be $\sigma(t):=f(\gamma(t))$.
  
We claim that there exists $\sigmah$, a  local modification of $\sigma$ such that $\sigma \sim \sigmah$ relative to {$\{0,1\}$ (and in fact relative to a larger set)}, with a homotopy satisfying the conditions of the First Homotopy Lemma (so that $\ind(\sigma-\gamma,0)=\ind(\sigmah-\gamma,0)$), and such that $\sigmah$ and $\gamma$ do satisfy the hypothesis of the Second Homotopy Lemma with $N=1$. The curve  $\sigmah$ is constructed  as follows. Let $t_1$ be the smallest parameter such that $\gamma(t_1) \in g_1$ and let $t_1'>t_1$ be such that $\sigma(t_1')=\gamma(t_1)$. Fix $\epsilon>0$ arbitrarily small and consider the arc $\sigma^1=\sigma(t_1'-\epsilon,t_1'+\epsilon)$. 
Now consider a homotopy $\sigma^1_s$  for $s\in [0,1]$ relative to {$\{0,1\}$} which progressively deforms the curve sliding it along  $\xi_1=\gamma[0,t_1]$ so that the end curve $\sigma^1_1$ does not intersect $\xi_1$ at any point.  Observe that $\sigma[0,t_1]$ is very far away from $\sigma^1_s$ for all $s$, so that $\sigma_s^1(t)-\gamma(t)\neq 0$ for any $s,t$.  Analogously, define a homotopy $\sigma^2_s$ in a neighborhood of $\xi_2$ which  slides $\sigma$ along $\xi_2$. Finally define $\sigmah(t):=\sigma_1^j(t)$ if $j=1,2$ and $t\in [t_j'-\epsilon,t_j'+\epsilon]$ and $\sigmah(t)=\sigma(t)$ otherwise. If we denote by $\widehat{U}_1$ and $\widehat{U}_2$ the bounded connected components of $\C\setminus \sigmah$  (which are local modifications of $U_1$ and $U_2$ respectively), we see that $\gamma$ {is contained in} the closure of $\widehat{U}_1$, so that the hypotheses of the Second Homotopy Lemma are satisfied with $N=1$. 
\end{proof}

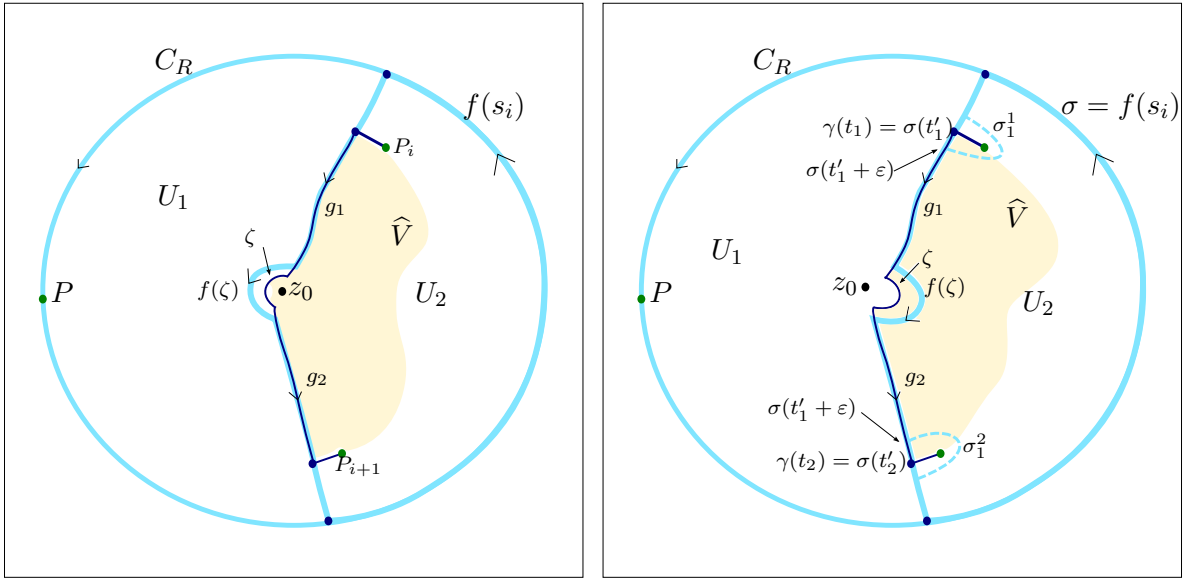
\begin{figure}[hbt!]
\begin{center}
\def\svgwidth{\textwidth}
\begingroup%
  \makeatletter%
  \providecommand\color[2][]{%
    \errmessage{(Inkscape) Color is used for the text in Inkscape, but the package 'color.sty' is not loaded}%
    \renewcommand\color[2][]{}%
  }%
  \providecommand\transparent[1]{%
    \renewcommand\transparent[1]{}%
  }%
  \providecommand\rotatebox[2]{#2}%
  \ifx\svgwidth\undefined%
    \setlength{\unitlength}{956.925bp}%
    \ifx\svgscale\undefined%
      \relax%
    \else%
      \setlength{\unitlength}{\unitlength * \real{\svgscale}}%
    \fi%
  \else%
    \setlength{\unitlength}{\svgwidth}%
  \fi%
  \global\let\svgwidth\undefined%
  \global\let\svgscale\undefined%
  \makeatother%
  \begin{picture}(1,0.49188808)%
    \put(0,0){\includegraphics[width=\unitlength]{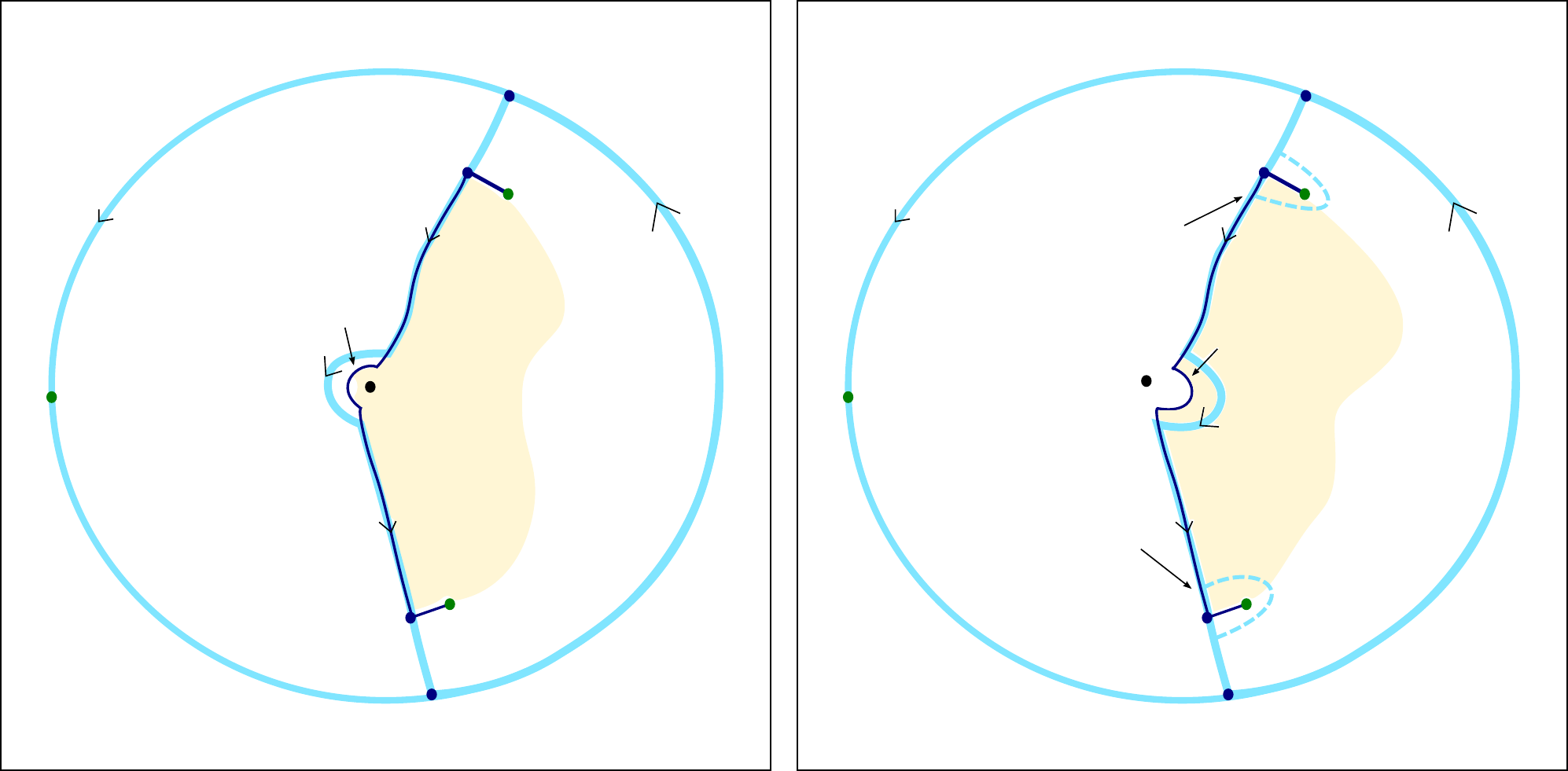}}%
    \put(0.127653,0.43501878){\color[rgb]{0,0,0}\makebox(0,0)[lb]{\smash{$C_R$}}}%
    \put(0.03988165,0.23603855){\color[rgb]{0,0,0}\makebox(0,0)[lb]{\smash{$P$}}}%
    \put(0.1301945,0.31918459){\color[rgb]{0,0,0}\makebox(0,0)[lb]{\smash{$U_1$}}}%
    \put(0.3491297,0.2363022){\color[rgb]{0,0,0}\makebox(0,0)[lb]{\smash{$U_2$}}}%
    \put(0.38925077,0.39697162){\color[rgb]{0,0,0}\makebox(0,0)[lb]{\smash{$f(s_i)$}}}%
    \put(0.280619,0.09137683){\color[rgb]{0,0,0}\makebox(0,0)[lb]{\smash{\ss $P_{i+1}$}}}%
    \put(0.3299988,0.36318009){\color[rgb]{0,0,0}\makebox(0,0)[lb]{\smash{\ss $P_i$}}}%
    \put(0.27234112,0.31274982){\color[rgb]{0,0,0}\makebox(0,0)[lb]{\smash{\ss $g_1$}}}%
    \put(0.25667195,0.16781001){\color[rgb]{0,0,0}\makebox(0,0)[lb]{\smash{\ss $g_2$}}}%
    \put(0.1639675,0.24124071){\color[rgb]{0,0,0}\makebox(0,0)[lb]{\smash{\ss $f(\zeta$)}}}%
    \put(0.20422499,0.28646303){\color[rgb]{0,0,0}\makebox(0,0)[lb]{\smash{\ss $\zeta$}}}%
    \put(0.24088334,0.24174893){\color[rgb]{0,0,0}\makebox(0,0)[lb]{\smash{\sc $z_0$}}}%
    \put(0.63575298,0.43501878){\color[rgb]{0,0,0}\makebox(0,0)[lb]{\smash{$C_R$}}}%
    \put(0.54798165,0.23603855){\color[rgb]{0,0,0}\makebox(0,0)[lb]{\smash{$P$}}}%
    \put(0.60032012,0.2725078){\color[rgb]{0,0,0}\makebox(0,0)[lb]{\smash{$U_1$}}}%
    \put(0.86447832,0.22857032){\color[rgb]{0,0,0}\makebox(0,0)[lb]{\smash{$U_2$}}}%
    \put(0.89735075,0.39697162){\color[rgb]{0,0,0}\makebox(0,0)[lb]{\smash{$\sigma=f(s_i)$}}}%
    \put(0.78044108,0.31274982){\color[rgb]{0,0,0}\makebox(0,0)[lb]{\smash{\ss $g_1$}}}%
    \put(0.7647719,0.16781001){\color[rgb]{0,0,0}\makebox(0,0)[lb]{\smash{\ss $g_2$}}}%
    \put(0.7807403,0.24355719){\color[rgb]{0,0,0}\makebox(0,0)[lb]{\smash{\ss $f(\zeta$)}}}%
    \put(0.77911216,0.26825182){\color[rgb]{0,0,0}\makebox(0,0)[lb]{\smash{\ss $\zeta$}}}%
    \put(0.70266214,0.24583056){\color[rgb]{0,0,0}\makebox(0,0)[lb]{\smash{\sc $z_0$}}}%
    \put(0.8139371,0.10777666){\color[rgb]{0,0,0}\makebox(0,0)[lb]{\smash{\ss$\sigma_1^2$}}}%
    \put(0.83942039,0.38256755){\color[rgb]{0,0,0}\makebox(0,0)[lb]{\smash{\ss$\sigma_1^1$}}}%
    \put(0.69519045,0.38038215){\color[rgb]{0,0,0}\makebox(0,0)[lb]{\smash{\ss $\gamma(t_1)=\sigma(t_1')$}}}%
    \put(0.68082661,0.34387259){\color[rgb]{0,0,0}\makebox(0,0)[lb]{\smash{\ss$\sigma(t_1'+\epsilon)$}}}%
    \put(0.65563381,0.09478343){\color[rgb]{0,0,0}\makebox(0,0)[lb]{\smash{\ss $\gamma(t_2)=\sigma(t_2')$}}}%
    \put(0.64840688,0.13976047){\color[rgb]{0,0,0}\makebox(0,0)[lb]{\smash{\ss$\sigma(t_1'+\epsilon)$}}}%
    \put(0.32782548,0.28453085){\color[rgb]{0,0,0}\makebox(0,0)[lb]{\smash{$\Vhat$}}}%
    \put(0.8502671,0.30017945){\color[rgb]{0,0,0}\makebox(0,0)[lb]{\smash{$\Vhat$}}}%
  \end{picture}%
\endgroup%
\end{center}
\caption{\small The computation of the index for an arc of type (0). Left: A repelling fixed point was previously in $\partial V$ and now belongs to $\Vhat$. The arc $s_i$ is shown in blue and belongs to $\ov{U}_2$. Its image is shown in light blue. Right:   A parabolic fixed point was previously in $\partial V$ and now is not in $\Vhat$. The arc $\gamma= s_i$ is shown in blue while its image $\sigma=f(s_i)$ is shown in light blue. The dashed curves show the local modification of $\sigma$ so that $\gamma$ {is contained in} the closure of $\hat{U}_1$.}
\label{indexcase0}
\end{figure}

\begin{claim}
If $s_i$ is of type $(1)$, the  counting  gives   $\ind(f(s_i)-s_i,0)=n_i+\ind(s_i,P)$.
\end{claim}
\begin{proof} Let $g_1$ and $g_2$ be  respectively the first and the second ray in $s_i$ that we encounter when moving  along $\Gamma_{\Vhat}$ counterclockwise. Let $\zeta$ be as in Lemma \ref{Enlarging rfp} or \ref{Enlarging pfp} depending on whether the original fixed point  was repelling ($n_i=1$) or parabolic ($n_i=0$). Then $f(s_i)$ starts at $P$, goes along $C_R$ until meeting $g_1$, continues on $g_1$ until hitting $\zeta$, then along $f(\zeta)$, then moves back along the ray $g_2$ and continues along $C_R$ until meeting $P$ again.  See Figure \ref{indexcase1}.Ê Note that  in this case $f(s_i)$ is a Jordan curve which therefore  divides the plane into exactly two connected components. Denote by $U_1$ the bounded component of $\C\setminus f(s_i)$ and by $U_0$ the unbounded one. Then  $\ind(f(s_i),z)=j$ if $z\in U_j$ for $j=0,1$. 

As in the claim above we first suppose that $n_i=1$ (see Figure \ref{indexcase1}). Then $s_i \subset U_1$ and by the Second Homotopy Lemma with $N=1$ we have $\ind(f(s_i)-s_i,0)=1+\ind(s_i,P)$.

Otherwise, if $n_i=0$, we encounter the same difficulty as above: namely, part of $s_i$ is in $U_0$ while the remaining part is in $U_1$. We proceed analogously, by constructing a homotopy between $\sigma:=f(s_i)$ and $\sigmah$ in such a way that $s_i$ is in the closure of $\hat{U}_0$ and hence $\ind(f(s_i)-s_i,0)=\ind(s_i,P)$ by the Second Homotopy Lemma with $N=0$.
\end{proof}

\begin{figure}[hbt!]
\begin{center}
\def\svgwidth{0.9\textwidth}
\begingroup%
  \makeatletter%
  \providecommand\color[2][]{%
    \errmessage{(Inkscape) Color is used for the text in Inkscape, but the package 'color.sty' is not loaded}%
    \renewcommand\color[2][]{}%
  }%
  \providecommand\transparent[1]{%
    \errmessage{(Inkscape) Transparency is used (non-zero) for the text in Inkscape, but the package 'transparent.sty' is not loaded}%
    \renewcommand\transparent[1]{}%
  }%
  \providecommand\rotatebox[2]{#2}%
  \ifx\svgwidth\undefined%
    \setlength{\unitlength}{939.77716bp}%
    \ifx\svgscale\undefined%
      \relax%
    \else%
      \setlength{\unitlength}{\unitlength * \real{\svgscale}}%
    \fi%
  \else%
    \setlength{\unitlength}{\svgwidth}%
  \fi%
  \global\let\svgwidth\undefined%
  \global\let\svgscale\undefined%
  \makeatother%
  \begin{picture}(1,0.50086342)%
    \put(0,0){\includegraphics[width=\unitlength]{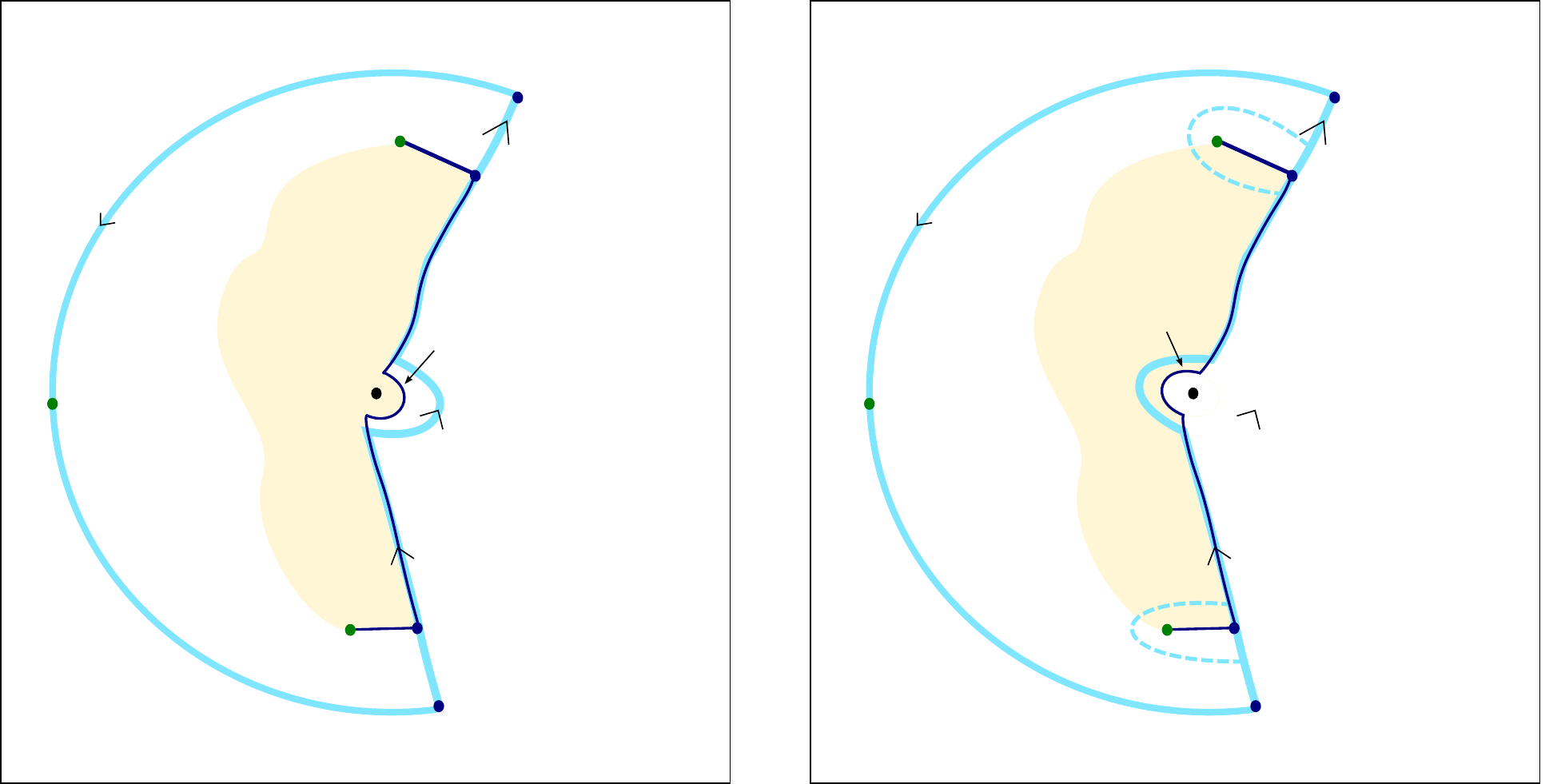}}%
    \put(0.13338731,0.44295644){\color[rgb]{0,0,0}\makebox(0,0)[lb]{\smash{$C_R$}}}%
    \put(0.04060936,0.24034548){\color[rgb]{0,0,0}\makebox(0,0)[lb]{\smash{$P$}}}%
    \put(0.06894015,0.31855969){\color[rgb]{0,0,0}\makebox(0,0)[lb]{\smash{$U_1$}}}%
    \put(0.35550017,0.24061394){\color[rgb]{0,0,0}\makebox(0,0)[lb]{\smash{$U_0$}}}%
    \put(0.23189281,0.46345896){\color[rgb]{0,0,0}\makebox(0,0)[lb]{\smash{$f(s_i)$}}}%
    \put(0.24401472,0.41881918){\color[rgb]{0,0,0}\makebox(0,0)[lb]{\smash{\ss $P_{i+1}$}}}%
    \put(0.27731044,0.31845648){\color[rgb]{0,0,0}\makebox(0,0)[lb]{\smash{\ss $g_2$}}}%
    \put(0.26135537,0.17087199){\color[rgb]{0,0,0}\makebox(0,0)[lb]{\smash{\ss $g_1$}}}%
    \put(0.26879253,0.21443015){\color[rgb]{0,0,0}\makebox(0,0)[lb]{\smash{\ss $f(\zeta$)}}}%
    \put(0.28229289,0.274222){\color[rgb]{0,0,0}\makebox(0,0)[lb]{\smash{\ss $\zeta$}}}%
    \put(0.73333703,0.42561364){\color[rgb]{0,0,0}\makebox(0,0)[lb]{\smash{\ss$\sigma_1^2$}}}%
    \put(0.68382178,0.09841803){\color[rgb]{0,0,0}\makebox(0,0)[lb]{\smash{\ss$\sigma_1^1$}}}%
    \put(0.20361344,0.10754435){\color[rgb]{0,0,0}\makebox(0,0)[lb]{\smash{\ss $P_i$}}}%
    \put(0.19486445,0.24691011){\color[rgb]{0,0,0}\makebox(0,0)[lb]{\smash{\sc $z_0$}}}%
    \put(0.17120946,0.33093022){\color[rgb]{0,0,0}\makebox(0,0)[lb]{\smash{$\Vhat$}}}%
    \put(0.65539525,0.44295644){\color[rgb]{0,0,0}\makebox(0,0)[lb]{\smash{$C_R$}}}%
    \put(0.56261727,0.24034548){\color[rgb]{0,0,0}\makebox(0,0)[lb]{\smash{$P$}}}%
    \put(0.59094809,0.31855969){\color[rgb]{0,0,0}\makebox(0,0)[lb]{\smash{$U_1$}}}%
    \put(0.87750809,0.24061394){\color[rgb]{0,0,0}\makebox(0,0)[lb]{\smash{$U_0$}}}%
    \put(0.75390072,0.46345896){\color[rgb]{0,0,0}\makebox(0,0)[lb]{\smash{$\sigma=f(s_i)$}}}%
    \put(0.79931837,0.31845648){\color[rgb]{0,0,0}\makebox(0,0)[lb]{\smash{\ss $g_2$}}}%
    \put(0.78336328,0.17087199){\color[rgb]{0,0,0}\makebox(0,0)[lb]{\smash{\ss $g_1$}}}%
    \put(0.67390955,0.25039659){\color[rgb]{0,0,0}\makebox(0,0)[lb]{\smash{\ss $f(\zeta$)}}}%
    \put(0.7121368,0.29370382){\color[rgb]{0,0,0}\makebox(0,0)[lb]{\smash{\ss $\zeta$}}}%
    \put(0.69321739,0.33093022){\color[rgb]{0,0,0}\makebox(0,0)[lb]{\smash{$\Vhat$}}}%
    \put(0.76915071,0.24640924){\color[rgb]{0,0,0}\makebox(0,0)[lb]{\smash{\sc $z_0$}}}%
    \put(0.83335456,0.38420664){\color[rgb]{0,0,0}\makebox(0,0)[lb]{\smash{\ss $\gamma(t_2)=\sigma(t_2')$}}}%
    \put(0.79701527,0.09957969){\color[rgb]{0,0,0}\makebox(0,0)[lb]{\smash{\ss $\gamma(t_i)=\sigma(t_1')$}}}%
  \end{picture}%
\endgroup%
\end{center}
\caption{\small The computation of the index for an arc of type (1). Left: A repelling fixed point was previously in $\partial V$ and now belongs to $\Vhat$. The arc $s_i$ is shown in blue and belongs to $\ov{U}_1$. Its image is shown in  light blue. Right:   A parabolic fixed point was previously in $\partial V$ and now is not in $\Vhat$. The arc $\gamma= s_i$ is shown in blue while its image $\sigma=f(s_i)$ is shown in  light blue. The dashed curves show the local modification of $\sigma$ so that $\gamma$ belongs to the closure of $\hat{U}_0$.}
\label{indexcase1}
\end{figure}
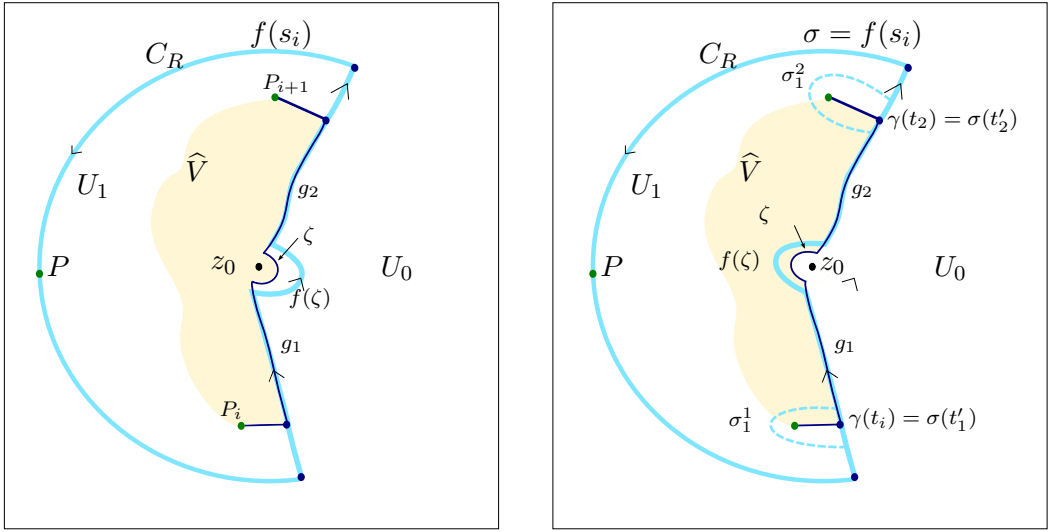

So in the {five} cases under consideration we obtain:
\begin{align*}
&\ind(f(s_i)-s_i,0)=1+n_i+\ind(s_i,P) &\text{ in case $(0)$; }\\
&\ind(f(s_i)-s_i,0)=n_i+\ind(s_i,P) &\text{ in case $(1)$; }\\
&\ind(f(s_i)-s_i,0)=1+\ind(s_i,P) &\text{ in case $(2)$; }\\
&\ind(f(s_i)-s_i,0)=\ind(s_i,P) &\text{ in cases {$(0')$ or }$(3)$.}
\end{align*}

Let $N_1$ be the number of repelling fixed points which were originally on $\partial V$, and $N_2$ be the number of arcs $s_i$ of type $(2)$. Observe that $N_1=\sum n_i$ (where the sum is over the indices for which $n_i$ is defined). Summing up, we have two different cases, giving the same final count. If $\delta\cap V=\emptyset$,  
\[
\begin{array}{rcl}
\ind(f(\GV)-\GV,0) &=&1+N_1+N_2+\sum_{i}\ind(s_i,P)\\
&=& 1+N_1+N_2 + \ind(\GV,P)\\
&=&1+N_1+N_2,
\end{array}
\]
where we have used that $P$ is contained in the unbounded connected component of $\C\setminus\GV$, so $\ind(\GV,P)=0$.

{If otherwise, $\delta\cap V \neq \emptyset$, then
\[
\begin{array}{rcl}
\ind(f(\GV)-\GV,0) &=&N_1+N_2+\sum_{i}\ind(s_i,P)\\
&=& N_1+N_2 + \ind(\GV,P)\\
&=&N_1+N_2 + 1,
\end{array}
\]
where we have used that $P$ is contained in the bounded connected component of $\C\setminus\GV$, so $\ind(\GV,P)=1$.
}

 We deduce from this that there is exactly one interior fixed point. Indeed,
\begin{itemize}
\item  there are exactly $N_1$ repelling fixed points which have been included in the region due to the modifications  described in  Lemma~\ref{Enlarging rfp}  (we do not need to count the parabolic ones, since they are no longer in $V'$ due to Lemma \ref{Enlarging pfp}); 
\item 
for each $s_i$ of type $(2)$ there is exactly one fixed point  in the interior $\Vhat$ which is the landing point of the unique fixed ray  asymptotically contained in that fundamental domain (notice that those rays must land alone or otherwise they would be part of a ray pair). 
\end{itemize}
This means that  there are $N_1+N_2$ fixed points in $\Vhat$ which are landing points of fixed rays, hence exactly one  fixed point is left. This fixed point must be interior because there are no more fixed rays left to land on it and because we assumed that $V$ (and hence $\Vhat$) contains no virtual fixed points at all.
 
 \subsubsection*{The case with multiple intersections}
 
 {We assumed at the beginning that  rays in $\partial V$ intersected  the circle $C_R$ (and therefore $r_\alpha$) in a single point. Otherwise, let us consider a modified region $\Vhat_R$ which includes the boundary rays only up to their {\em first} intersection $P$ with $r_\alpha$ (starting from the fixed point and moving to infinity). The counting for $\Vhat_R$ is then the one in the section above, independently of $R$. The region $\Vhat_R$ differs from $\Vhat$ by finitely many bounded sets $U_i$ whose boundary consists of    pieces of the ray and  pieces of $r_\alpha$, and each $U_i$ could have been either added or removed. 
So it is necessary and sufficient to show that none of the $U_i$ contains interior fixed points  for $R$ large enough.  By Theorem~\ref{Global counting theorem}, there are only finitely many interior fixed points for $f$, so they are all contained in a disk of radius say $R_1$.
We will show that for any fixed ray $g$ on $\partial V$, $R$ can be chosen large enough so that the $U_i$ whose boundary contains pieces of  $g$ do not intersect $D_{R_1}$, hence cannot contain interior fixed points. Since there are finitely many fixed rays on $\partial V$, up to taking  $R$ sufficiently large, we obtain that the number of interior fixed points in  $\Vhat_R$ equals the number of interior fixed points in $\Vhat$.

{Suppose the original ray $g$ is parametrized by $[0,\infty)$ where $g(0)$ is the fixed  point where it lands. Even after modification of the ray near the fixed point, the parametrization still makes sense for $t\geq t^*$, for some $t^*>0$. } For a point $z=g(t)$ we refer to the parameter  $t$ as the \emph{potential} of $z$. 
Since the  fixed rays on $\partial V$ land on a fixed point on one side, and tend to infinity on the other side, for any $R$ the ray  $g$ can  cross the circle $C_{R}$ only finitely many times. 
So we can let $T$ be the largest potential such that $g(T)$ intersects $D_{R_1}$. Let $T'$ be such that $f(g(T))=g(T')$. 

{Again because of convergence to infinity}  we can choose $R$ sufficiently large such that the first intersection (coming from the fixed point) with $g$ has potential $T''>T'$; for example, let $R>R_1+\diam g([T,T''])$. Since the ray is an injective curve hence an ordered set, and since $r_\alpha$ is a preimage of $C_R$, all points in $g\cap r_\alpha$ have potential {$t>T$}.  Thus every portion of $g$ in $\partial U_i$ has potential larger than $T$ at the extremities so has potential larger than $T$ everywhere. It follows that none of the $U_i$ bounded by $g$ and $r_\alpha$ can intersect $D_{R_1}$,  hence none of them contains interior fixed points.

\subsubsection*{The final counting}

  To conclude the proof of the Main Theorem we have to take into account the basic regions containing virtual fixed points. We will use  the global counting in Theorem \ref{Global counting theorem}  and  follow an argument analogous to the one in \cite[Proof of Theorem 3.3]{GM}.
 
Let $N=\#\FF$, and $n$ be the number of  landing  points of fixed rays which are asymptotically contained in some $F\in\FF$.  By Theorem \ref{Global counting theorem} there are $N+1$ fixed points, counted with multiplicity, which are landing points of rays asymptotically contained in some $F\in\FF$, or interior fixed points. Hence there are $N+1-n$ fixed points which therefore are  interior or virtual. 

On the other hand, let $m$ be the total number of basic regions and $v$ be the number of basic regions containing at least one virtual fixed point. It is easy to check that $m=N-n+1$. Hence  $m-v=N-n+1-v$ regions contain no virtual fixed point. Thus by the index counting above, there are $N-n+1-v$ interior fixed points, one inside each of these regions. Now let $v'\geq v$ be the actual number of virtual fixed points. Then the total number of fixed points which are interior or virtual is at least $N-n+1-v + v'$ (notice that {\em a priori} there could be interior fixed points  in the $v$ regions with at  least one virtual fixed point). 

Putting together both computations we obtain that 
\[
  N+1-n \geq N-n+1-v + v' .
\]
Hence $v'\leq v$ from which we deduce that $v'=v$. This implies that there is exactly one virtual fixed point in each of the $v$ regions which had at least one of them. But moreover, all fixed points have been accounted for, so these $v$ regions contain no interior fixed point.

\section{Proof of corollaries D and E}\label{Corollaries}
As explained in the introduction, Theorem C is an immediate consequence of the Main Theorem,  replacing $f$ by $f^p$ and observing that class $\BBt$  is closed under composition. 
In this section we prove some of the many corollaries of Theorem C. We recall that the period is not assumed to be exact.

\begin{cor}\label{Corollari}
If $f$ is a function in $\Bt$ whose periodic rays land,
\begin{enumerate}
\item There are only finitely many interior periodic  points of any given period.  In particular, there are only finitely many non-repelling cycles of any given period.
\item Any two periodic Fatou components can be separated by a periodic ray pair. 
\item There are no Cremer   periodic points on the boundary of periodic Fatou components.
\item If $z_0$ is a  parabolic point, for each repelling petal there is at least  one  ray  which lands at $z_0$ through that repelling petal. In particular, any two of its attracting petals are separated by a ray pair of the same period {as the virtual fixed point associated to the petal}.  
\item For any given period $p$, there are only finitely many (possibly none) periodic points of period $p$ which are landing points of more than one periodic ray. None of them is the landing point of infinitely many rays of the same period.
\end{enumerate}
\end{cor}

\begin{proof}[Proof of Corollary~\ref{Corollari}]
\begin{enumerate}
\item By Theorem C, for each $p$ there are finitely many basic regions for $f^p$, and each of them contains at most one interior fixed  point of $f^p$ which is an interior periodic point of period $p$ (not necessarily exact) for $f$.
\item Let $U_1, U_2$ be two periodic Fatou components of least common period $p$.  Since $f\in\BB$, there are no Baker domains \cite{EL}, hence the components are either  basins of attraction of attracting or parabolic $p-$periodic points, or Siegel disks. Then  $\ov{U_1}$ and $\ov{U_2}$ each contain  a point $z_1,z_2$ respectively which is fixed under $f^p$.
If both $z_1$ and $z_2$  are in the interior of their Fatou component, they are interior periodic points, so they are separated by  a $p-$periodic ray pair  by Theorem C. As dynamic rays cannot intersect Fatou components, and Fatou components are connected sets, it follows that the same ray pair separates $U_1$ from $U_2$.
If instead at least one of them is on the boundary, say $z_1$, then $z_1$ is a parabolic point and $U_1$ is a virtual fixed point for $f^p$.  In particular, by Theorem C, there are no other interior $p-$periodic points or virtual fixed points of $f^p$ in  the basic region $V$ containing $U_1$, so $U_2$ is separated from $U_1$ by two rays on the boundary of $V$ (observe that $U_2$ could also be another virtual fixed point of $f^p$ attached to the same parabolic point if $z_1=z_0$).   

\item Observe that Cremer points are interior periodic points, since by the Snail Lemma (see Lemma 16.2 in \cite{Mi}) no periodic rays can land at Cremer points. Hence they can be separated from any other interior periodic point or periodic Fatou component as in the previous case.

\item If $z_0$ is a multiple fixed point for $f^p$ which has only one  immediate basin attached, then there is at least a  fixed ray of $f^p$ landing at $z_0$, otherwise $z_0$ would be an interior fixed point and would be contained in the same basic region as its attracting petal, contradicting  Theorem C. If $z_0$ has more than one virtual fixed point of $f^p$ attached, and there were a repelling petal which does not contain any fixed ray for $f^p$ landing at $z_0$, then its two adjacent immediate basins (virtual fixed points)  would be contained in the same basic region for $f^p$, contradicting again Theorem C. 

\item If not, there would be infinitely many basic regions for $f^p$, contradicting Proposition A.
\end{enumerate}
\end{proof}

From the fact that Fatou components can be separated by periodic ray pairs we obtain the following additional corollary.
Given an invariant Siegel disk $\Delta$, we say that $U$ is a hidden component of $\Delta$ if $U$ is a bounded connected component of $\C\setminus\overline{\Delta}$. 
The proof of the next corollary follows the outline of Lemmas 2, 3 and 10 in \cite{CR}; see also \cite{Ro}. 
\begin{cor}
If $f\in\BBt$ and all periodic rays land, any hidden component of a bounded invariant Siegel disk is either a wandering domain or preperiodic to the Siegel disk itself.
\end{cor}

\begin{proof}
Observe  that hidden components are bounded Fatou components. In fact, $\partial\Delta$ is bounded and forward invariant, and $\partial U\subset \partial\Delta$, so by the Maximum Principle and Montel's Theorem $U$ is contained in the Fatou set.
Suppose that there is a hidden component $U$ which is not preperiodic to $\Delta$.
Observe that  $\partial f( U)\subset f(\partial U)$ because $f$ is holomorphic hence open, and $f(U)\cap \partial\Delta=\emptyset$ because $f(U)\subset F(f)$ and $\partial\Delta\subset J(f)$. Also $f(U)$ is bounded and connected because it is the image of a bounded connected set, and its boundary is contained in  $\partial\Delta$, so $f(U)$ is also a  hidden component of $\Delta$. Hence, if $U$ is not a  wandering domain, it is preperiodic to some periodic Fatou component $\widetilde{U}$ which is a hidden component of $\Delta$; by part $4.$ of Corollary \ref{Corollari}, unless $\widetilde{U}=\Delta$, they can be separated by two periodic rays landing together, hence $\partial\widetilde{U}\cap \partial\Delta$ is at most a  single point- which is not possible as $\partial\widetilde{U}\subset \partial\Delta$. 
\end{proof}

\end{document}